\documentclass[reqno]{amsart}

\usepackage{todonotes}
\usepackage{a4wide}
\usepackage{mathrsfs}
\usepackage{mathtools}
\usepackage{tikz-cd}
\usepackage{amsmath}
\usepackage{amssymb}
\usepackage{amsthm}
\usepackage{amsfonts}
\usepackage{bbm}
\usepackage{bm}
\usepackage{comment}
\numberwithin{equation}{section}
\usepackage[colorlinks,citecolor=green,linkcolor=red]{hyperref}
\usepackage{hyphenat}
\usepackage[inline]{enumitem}
\usepackage{stmaryrd}

\usepackage{graphicx} 
\usepackage{placeins}
\usepackage{float}
\usepackage{subcaption}
\usepackage{soul}

\newcommand{\N}{\mathbb{N}}
\newcommand{\R}{\mathbb{R}}

\renewcommand{\d}{{\mathrm d}}
\newcommand{\restr}[1]{\lower3pt\hbox{\(|_{#1}\)}}

\newcommand{\nchi}{{\raise.3ex\hbox{\(\chi\)}}}

\newcommand{\X}{{\rm X}}

\newcommand{\XX}{\mathbb{X}}

\newcommand{\inv}{^{-1}}
\newcommand{\ud}{\mathrm {d}}
\newcommand{\M}{\mathbb{M}}

\newtheorem{theorem}{Theorem}[section]
\newtheorem{corollary}[theorem]{Corollary}
\newtheorem{lemma}[theorem]{Lemma}
\newtheorem{proposition}[theorem]{Proposition}
\theoremstyle{definition}
\newtheorem{definition}[theorem]{Definition}

\newtheorem{remark}[theorem]{Remark}

\title{Limits of mapping packages and tangents of Sobolev maps}
\author{Milica Cakovi\'{c}}
\address{Department of Mathematics and Statistics,
P.O.\ Box 35, FI-40014 University of Jyvaskyla;
Department of Mathematics and Informatics, Faculty of Sciences, University of Novi Sad, Trg Dositeja Obradovi\'ca 3, 21000 Novi Sad, Serbia}
\email{milica.m.cakovic@jyu.fi}

\author{Elefterios Soultanis}
\address{Department of Mathematics and Statistics,
P.O.\ Box 35, FI-40014 University of Jyvaskyla}
\email{elefterios.e.soultanis@jyu.fi}

\thanks{Both authors were supported by Research Council of Finland grant no. 355122. The work in this manuscript was also partially supported by the Simons Foundation grant (award no. SFI-MPS-T-Institutes-00010825) and from State Treasury funds as part of a task commissioned by the Minister of Science and Higher Education under the project “Organization of the Simons Semesters at the Banach Center - New Energies in 2026-2028” (agreement no. MNiSW/2025/DAP/491).}


\begin{document}
\date{\today} 
\begin{abstract}
We show the existence of ultralimits of sequences of Haj\l asz--Sobolev mappings $f_i:X_i\to Y_i$ when $X_i$ converges to a limit space in the pointed measured Gromov (pmG) sense, partially extending recent results in \cite{IW26}. We moreover demonstrate that the mappings pmG-converge to the ultralimit in a suitable sense, and obtain Arzel\`a--Ascoli theorems in the context of pmG-convergence.

As an application, we establish a version of Preiss's phenomenon for Haj\l asz--Sobolev maps $f:(X,\mu)\to V$ into arbitrary Banach spaces. Besides extending it to maps into infinite dimensional targets, our result generalizes existing versions of Preiss's phenomenon \cite{DG2015,GMR15} by establishing it for pointed measured Gromov--Hausdorff tangents without a doubling assumption.
\end{abstract}
\maketitle
\tableofcontents

\section{Introduction}

Since the introduction of Gromov--Hausdorff (GH) convergence of metric spaces \cite{Gromov1981}, different variants have become ubiquitous in the literature, including \emph{pointed} GH convergence (pGH), \emph{pointed measured} GH convergence (pmGH), and \emph{pointed measured Gromov convergence} (pmG), as well as (the somewhat different) ultralimits \cite{vanDenDriWil1984}. This diversity of notions reflects the extreme usefulness of taking limits of metric spaces in a great variety of settings, and has led to a vast number of applications across geometric analysis. Among these we mention the recent work \cite{IW26} on ultralimits of Sobolev maps and the theory of metric tangents and Preiss's phenomenon (e.g. \cite{LeDonne2011,GMR15,Bate22} for spaces and \cite{DG2015} for mappings), both of which are the focus of the present article. 

Le Donne's Preiss phenomenon \cite{LeDonne2011} states that, at a $\mu$-generic point $x$ in a doubling metric measure space $(X,\mu)$, metric tangents $(Y,o)\in \operatorname{Tan}(X,x)$ remain tangents if the basepoint is shifted:
\begin{align*}
	(Y,y)\in \operatorname{Tan}(X,x)\quad\textrm{for all }y\in Y.
\end{align*}
The same phenomenon persists for blow-ups of the measure $\mu$ \cite{GMR15,Bate22}, and for blow-ups of a given Lipschitz function $f:X\to \R$ \cite{DG2015}. These results play a important role in understanding \emph{Lipschitz differentiability spaces} \cite{DG2015,CheKleSchi2016,Schioppa2016}, which are connected to rectifiability and embeddability questions on metric spaces \cite{che99,Bate22,terirect,CEBS26}.

Ultralimits of Sobolev maps \cite{IW26} generalize the following functiorial construction: given a sequence $f_i:(X_i,x_i)\to (Y_i,y_i)$ of $L$-Lipschitz maps and an ultrafilter $\omega$ on $\N$, there exists an ultralimit
\begin{align}\label{eq:lip-ultralim}
f_\omega:X_\omega\to Y_\omega,\quad f_\omega([x_i]_\omega):=[f_i(x_i)]_\omega.
\end{align}
The construction in \cite{IW26} is done for a constant sequence $X_i=\Omega\subset \R^n$ of Euclidean Lipschitz domains; it is useful for proving the stability of isoperimetric inequalities, which stem from minimal surface theory, geometric group theory and metric geometry, see \cite{IW26} and the references therein. The ultralimit of Lipschitz maps exists for arbitrary (equi-Lipschitz) sequences and thus remedies the need for compactness properties other basic limit theorems (like the Arzel\`a--Ascoli theorem) need. While this is desirable in many situations of geometric interest -- for example for mappings into infinite dimensional targets -- ultralimits are nevertheless usually not well behaved measure theoretically; for example Preiss's phenomenon is not known for tangents of Lipschitz maps into arbitrary Banach targets.

In this article, we construct the ultralimit $f_\omega$ of a sequence $f_i:X_i\to Y_i$ of Sobolev maps from varying domains $X_i$ which are (pointed) metric measure spaces, and show that $f_\omega$ is a limit of $(f_i)$ in a suitable topology. This topology is metrizable and coincides with earlier notions of convergence in the literature \cite{DavidSemmes1997,DG2015} in the case of equi-Lipschitz mappings from doubling spaces into Euclidean space. Our result partially extends the result of Ikonen--Wenger \cite{IW26} to maps from varying domains, and establishes a link between ultralimits and pmG-limits of mappings (cf. Definitions \ref{def:pmG-conv-new} and \ref{def:equi-lip-maps-conv}). Moreover, it motivates (and simultaneously justifies) a new definition of pmG-convergence for \emph{mapping packages}; we use this new notion to show the a.e. existence of tangents of Sobolev maps with arbitrary Banach target, and to establish Preiss's phenomenon for them. This will be useful for studying differentiability of maps into infinite dimensional targets.

Finally, we provide a proof of a folklore result, namely Preiss's phenomenon for pmGH-tangents of Lipschitz maps from metric measure spaces into Euclidean space, under the assumption that the measure vanishes on porous sets. In this context it is worth remarking that Preiss's phenomenon fails for pmGH-tangents if the measure is merely pointwise doubling \cite[Remark 3.4]{LeDonne2011}; vanishing on porous sets is a slightly stronger assumption and allows us to pass from pmG to pmGH-tangents.

\subsection{Ultralimit of Sobolev maps and convergence of mappings} A Borel map $f:(X,\mu)\to Y$ from a metric measure space $(X,\mu)$ to a complete metric space $Y$ is a \emph{Haj\l asz--Sobolev map}, if there exists a function $g\in L^p_{loc}(\mu)$ satisfying
\begin{align}\label{eq:haj}
    d(f(x),f(y))\le d(x,y)[g(x)+g(y)],\quad x,y\in X\setminus N
\end{align}
for some $\mu$-null set $N\subset X$. Any such $g$ is called a \emph{Haj\l asz gradient} of $f$, and the collection of Haj\l asz--Sobolev maps $X\to Y$ is denoted by $M^{1,p}_{loc}(X,Y)$.

In our first main result, we construct an ultralimit $f_\omega$ for a sequence of Haj\l asz--Sobolev maps $f_i:X_i\to Y_i$. Moreover, we define a complete separable metric $d_{pmG}$ on the set $\M$ of (equivalence classes of) \emph{mapping packages} (see Section \ref{sec:mappings_ms}) and show that $f_\omega$ is the ultralimit of the sequence $(f_i)$ in $(\M,d_{pmG})$. Convergence with respect to $d_{pmG}$ is a natural analogue of pmG-convergence of metric measure spaces for mappings: it coincides with the convergence considered in \cite{DG2015} for equi-Lipschitz maps from doubling spaces into Euclidean space (Remark \ref{rmk:pmGH-vs-guy}), and a sequence of zero maps $\bar{\bm 0}_{(X_i,\mu_i,\bar x_i)}:(X_i,\mu_i,x_i)\to\R$ converges to $\bar{\bm 0}_{(X,\mu,\bar x)}$ with respect to $d_{pmG}$ if and only if $(X_i,\mu_i,\bar x_i)$ pmG-converges to $(X,\mu,\bar x)$.

In the statement below, $\omega$ is a non-principal ultrafilter on $\N$, and $\XX$ denotes the (equivalence classes of) pointed metric measure spaces equipped with the topology of pmG-convergence (see Definition \ref{def:pmG_conv}), which is metrizable by a complete separable metric. 

\begin{theorem}\label{thm:ultralimit+graph}
Suppose $(X_i=(X_i,\mu_i,\bar x_i))$ is precompact in $\XX$ with $\omega$-limit $X_\omega=(X_\omega,\mu_\omega,\bar x_\omega)$, and $f_i:X_i\to Y_i$ are Haj\l asz--Sobolev maps into complete pointed metric spaces $Y_i$ with Haj\l asz gradients $g_i\in L_{loc}^p(\mu_i)$ satisfying $\bar x_i\in {\{g_i\le t\}}$ for some $t$ and
\begin{align}\label{eq:energy-unif-bdd}
\sup_i
\int_{B(\bar x_i,R)}g_i^p\ud\mu_i<\infty\quad \mathrm{for\ all }\ R>0.
\end{align}
Then there exists $f_\omega\in M^{1,p}_{loc}(X_\omega,Y_\omega)$ and $g_\omega\in L^p_{loc}(\mu_\omega)$ such that
\begin{itemize}
    \item[(i)] $g_\omega$ is a Haj\l asz gradient of $f_\omega$;
    \item[(ii)] $\displaystyle\int_{B(\bar x_\omega,R)}g_\omega^p\ud\mu_\omega\le \lim_\omega\int_{B(\bar x_i,R)}g_i^p\ud\mu_i$ for all $R>0$;
   \item[(iii)] $\lim_\omega d_{pmG}(f_\omega,f_i)=0$.
\end{itemize}
If $\iota_i:X_i\to (W,w)$, $i\in\N\cup\{\omega\}$ are isometric embeddings for which $\lim_\omega\iota_{i\ast}\mu_i=\iota_{\omega\ast}\mu_\omega$, then for every $\lambda>0$ the following holds: for $\mu_\omega$-a.e. $x\in \{g_\omega < \lambda\}$ we have 
    \begin{align}\label{eq:ultralimit+graph}
        f_\omega(x)=[f_i( x_i)]_\omega,\quad\mathrm{whenever}\quad  x_i\in \{g_i\le \lambda\},\quad \lim_\omega\iota_i( x_i)=\iota_\omega(x).
    \end{align}
\end{theorem}
Theorem \ref{thm:ultralimit+graph} generalizes the construction in \cite[Theorem 1.1]{IW26} to sequences of Sobolev maps from varying domains. Note that the Sobolev spaces $W^{1,p}(\Omega,Y)$ considered in \cite{IW26} coincide with $M^{1,p}(\Omega,Y)$ when $\Omega$ is a bounded Lipschitz domain. 

While the ultralimit of a sequence of equi-Lipschitz maps \eqref{eq:lip-ultralim} can be defined pointwise directly, Theorem \ref{thm:ultralimit+graph} is new even in this case.  	
The pmG-convergence in Theorem \ref{thm:ultralimit+graph}(iii) could probably be deduced more directly for maps $f_i:X_i\to V$ into a fixed finite dimensional target $V$. However, when $V$ is infinite dimensional, the validity of (iii) is surprising (even for equi-Lipschitz sequences), because there is no compactness available in the target space; in general $f_i(x_i)$ will not (sub)converge to $f_\omega([x_i]_\omega)$ in $V$, nor in the ultrapower $V^\omega$ (i.e. the ultralimit of the constant sequence $Y_i=(V,0)$).

They key point in our definition of $d_{pmG}$ is to look at \emph{graphs} of mappings. Given a mapping package $f:(X,\mu,\bar x)\to (Y,\bar y)$, the {graph} of $f$ is the metric measure space $(G(f),\mu_f)$, where
\begin{align*}
	\bar f=x\mapsto (x,f(x)):X\to X\times Y,\quad \mu_f=\bar f_\ast\mu\quad\mathrm{and}\quad G(f)=\operatorname{spt}\mu_f\subset X\times Y,
\end{align*}
equipped with the basepoint $(\bar x,\bar y)$. 

Graphs thus reduce mapping packages to metric measure spaces, whose pmG-convergence has been metrized e.g. in \cite{GMS15,SP21,Bate22}. Indeed, given two pointed metric measure spaces $X_i=(X_i,\mu_i,\bar x_i)$ ($i=0,1$), define
\begin{align*}
d_{pmG}(X_0,X_1)=\inf\{F_{(Z,z)}(\iota_{0\ast}\mu_0,\iota_{1\ast}\mu_1)\},
\end{align*}
where the infimum is taken over all pointed isometries $\iota_i:(X_i,\bar x_i)\to (Z,z)$ into metric spaces $Z$, and $F_{(Z,z)}$ is a Prokhorov type distance metrizing weak convergence of measures in $Z$, cf. \cite[Definition 2.12]{Bate22} and Section \ref{sec:MMTN}. Then $X_j\stackrel{pmG}{\longrightarrow} X$ if and only if $d_{pmG}(X_j,X)\stackrel{j\to\infty}{\longrightarrow}0$.

Simply taking the pmG-distance of the graphs, however, produces a weaker notion of convergence which does not recover that of \cite{DG2015}. The main issue is that \emph{the graph does not unambiguously determine the map}, a point discussed further in Remark \ref{rmk:graph}. Instead, we define the pmG-distance of mapping packages by considering only those isometric embeddings of the graphs which respect the product structure.

\begin{definition}\label{def:pmG-conv-new}
Let $f_i:(X_i,\mu_i,\bar x_i)\to (Y_i,\bar y_i)$ be mapping packages, $i=0,1$. The pmG-distance between $f_0$ and $f_1$ is defined as
\begin{align}
d_{pmG}(f_0,f_1)=\inf\{F_{(Z\times W,(z,w))}((\iota_{0},\widetilde\iota_0\circ f_0)_\ast\mu_0,(\iota_{1},\widetilde\iota_1\circ f_1)))_\ast\mu_1) \},
\end{align}
where the infimum is taken over all pointed isometric embeddings $\iota_i:(X_i,\bar x_i)\to (Z,z)$ and $\widetilde \iota_i:(Y_i,\bar y_i)\to (W,w)$ into complete metric spaces, $i=0,1$.
\end{definition}

The distance $d_{pmG}$ defines a notion of pmG-convergence of mappings which coincides with the convergence in \cite{DG2015} in that setting. It is moreover clear that 
\begin{align*}
d_{pmG}(\bar{\bm 0}_{(X_0,\mu_0,\bar x_0)},\bar{\bm 0}_{(X_1,\mu_1,\bar x_1)})=d_{pmG}((X_0,\mu_0,\bar x_0),(X_1,\mu_1,\bar x_1)),
\end{align*}
so that the inclusion $(\XX,d_{pmG})\hookrightarrow(\M,d_{pmG})$  given by $X\mapsto \bar{\bm 0}_X$ is isometric. The advantage of $d_{pmG}$ is that it enjoys a useful compactness statement which holds independently of the target of the mappings, and which can be used to replace the use of Arzel\`a--Ascoli in those situations. We will use pmG-convergence of mappings, along with the compactness property, to establish the existence of tangents of Sobolev maps into arbitrary Banach targets. We moreover prove that the collection of tangents enjoys Preiss's phenomenon.

\subsection{Tangents of Sobolev maps}  Suppose $(X,\mu)$ is a metric measure space and $f:(X,\mu)\to V$ a Bochner measurable map (essentially separably valued, $\mu$-measurable map) into a Banach space $V$. Given $x\in \operatorname{spt}\mu$ and $r>0$, the rescalings $T_{x,r}X$ and $T_{x,r}f$ are defined by 
\begin{align*}
    T_{x,r}X=\Big(r\inv X,\frac{\mu}{\varphi_\mu(x,r)},x\Big),\quad T_{x,r}f=\frac{f-f(x)}{r}:T_{x,r}X\to V.
\end{align*}
Here, $\displaystyle\varphi_\mu(x,r)=\int\varphi\Big(\frac{d(y,x)}{r}\Big)\ud\mu(y)$ is a normalizing factor given by a choice of a function $\varphi:[0,\infty)\to [0,\infty)$ such that ($\Phi$i) $\varphi$ is {non-increasing}, boundedly supported and semicontinuous;\footnote{A function is semicontinuous if it is either lower or upper semicontinuous.} ($\Phi$ii) $\varphi$ is continuous at 0, and $\varphi(0)>0$. For example, the choice $\varphi=\chi_{[0,1)}$ yields the standard normalization $\frac{\mu}{\mu(B(x,r))}$, while $\varphi(t)=(1-t)_+$ is used in \cite{GMR15}. We fix $\varphi$ for the remainder of this introduction.

\begin{definition}\label{def:graph-tan}
Let $f:(X,\mu)\to V$ be a Bochner measurable map into a Banach space $V$. A map $f_\infty:X_\infty\to Y$ into a Banach space $Y$ is a pointed measured Gromov (pmG) tangent of $f$ at $x\in\operatorname{spt}\mu$ if there exists a sequence $r_i\to 0$ so that $d_{pmG}(T_{x,r_i}f,f_\infty)\to 0$. The collection of all pmG-tangents of $f$ at $x$ is denoted by $\operatorname{Tan}_{pmG}(f,x)$. 
\end{definition}
Note that if $V$ is infinite dimensional, requiring tangent maps to have the same target space is too restrictive (we do not know if tangents with this additional condition always exist).  However, by Theorem \ref{thm:ultralimit+graph} we may always consider the target $Y$ to be the ultrapower $V^\omega$ of $V$ for a given non-principal ultrafilter $\omega$ if $f$ is a Haj\l asz--Sobolev map; in particular, we may always assume that tangent maps have a linear target space.

\begin{theorem}\label{thm:Gr-preiss}
    Let $(X,\mu)$ be a pointwise doubling metric measure space, $V$ a Banach space, and $f:(X,\mu)\to V$ a Haj\l asz--Sobolev map. Then for $\mu$-a.e. $x\in X$ the following hold.
\begin{itemize}
    \item[(i)] $\operatorname{Tan}_{pmG}(f,x)$ is non-empty and compact;
    \item[(ii)] if $f_\infty:(Y,\nu,o)\to W$ belongs to $\operatorname{Tan}_{pmG}(f,x)$ and $y\in \operatorname{spt}\nu$, $\rho>0$, then $\frac{f_\infty-f_\infty(y)}{\rho}:\Big(\rho\inv Y,\frac{\nu}{\varphi_\nu(y,\rho)},y\Big)\to W$ belongs to $\operatorname{Tan}_{pmG}(f,x)$.
\end{itemize}
\end{theorem}

We end the introduction by briefly discussing the relationship between pmG and pmGH tangents. For simplicity, we restrict our attention here to \emph{Lipschitz} maps. We refer to Definition \ref{def:equi-lip-maps-conv} and Remark \ref{rmk:pmGH-vs-guy} for the definition of and discussion on pmGH-convergence of Lipschitz maps.

When $(X,\mu)$ is doubling, the equivalence of pmG and pmGH-convergence yields the equality $\operatorname{Tan}_{pmG}(f,x)=\operatorname{Tan}_{pmGH}(f,x)$. Thus, under these assumptions pmGH-tangents satisfy Preiss's phenomenon (i.e. Theorem \ref{thm:Gr-preiss}(ii)); see \cite[Proposition 3.1]{DG2015} and also \cite{GEB24}. However, Preiss's phenomenon for pmGH-tangent of metric measure spaces (and thus in particular for maps) may fail for pointwise doubling spaces by \cite[Remark 3.4]{LeDonne2011}. 

In Theorem \ref{thm:pmGH-preiss} below, we prove that Preiss's phenomenon persists for pmGH-tangents under an additional assumption: that $\mu$ vanishes on \emph{porous} sets (see Section \ref{sec:MMTN} for the definition); we remark that such measures are always pointwise doubling. For metric measure spaces, this result has been discussed e.g. in \cite{CheKleSchi2016,Schioppa2016}, where a proof idea was outlined. We therefore regard it as a folklore result. We formulate it below for Lipschitz maps into finite dimensional targets and provide a detailed proof in Section \ref{sec:Tangents3}.

\begin{theorem}\label{thm:pmGH-preiss}
Let $(X,\mu)$ be a metric measure space where $\mu$ vanishes on porous sets, and let $f:(X,\mu)\to V$ be a Lipschitz map into a finite dimensional Banach space $V$. Then for $\mu$-a.e. $x\in X$ the following holds. 
\begin{itemize}
    \item[(i)] $\operatorname{Tan}_{pmGH}(f,x)=\operatorname{Tan}_{pmG}(f,x)$ is non-empty and compact;
    \item[(ii)] if $f_\infty:(Y,\nu,o)\to V$ belongs to $\operatorname{Tan}_{pmGH}(f,x)$ and $y\in Y$, $\rho>0$, then $\frac{f_\infty-f_\infty(y)}{\rho}:(\rho\inv Y,\frac{\nu}{\varphi_\nu(y,\rho)},y)\to V$ belongs to $\operatorname{Tan}_{pmGH}(f,x)$.
\end{itemize}
\end{theorem}
By considering the constant function $f=0$, we obtain the analogous statements also for $\operatorname{Tan}_{pmGH}(X,\mu,x)$ and $\operatorname{Tan}_{pmG}(X,\mu,x)$. Theorem \ref{thm:pmGH-preiss} in this special case is used in the recent work \cite{CEBS26} of the second named author.

\subsubsection*{\normalfont{\textbf{Acknowledgements}}}

The authors would like to thank Toni Ikonen for fruitful discussions. At the final stages of preparation, a commercial AI subscription was used to upgrade Theorem \ref{thm:Gr-preiss} from Lipschitz to Haj\l asz--Sobolev maps. In the process we established a stronger formulation of Theorem \ref{thm:ultralimit+graph}(iii); the crucial AI contribution is Lemma \ref{lem:common-ambient-space}. 

\section{Preliminaries}
\label{sec:preliminaries}
\subsection{Ultralimits} 
\label{sec:ultralimits}
A \textit{non-principal filter} on $\N$ is a collection $\omega\subset 2^{\N}$ of subsets of natural numbers that satisfy the following properties:
\begin{itemize}
    \item [(i)] If $A, B\in \omega$, then $A\cap B\in \omega$.
    \item [(ii)] If $A\subset B$ and $A\in \omega$, then $B\in \omega$.
    \item[(iii)] If $A\subset \N$, then either $A\in \omega$ or $\N\setminus A\in \omega$.
    \item[(iv)] The collection $\omega$ does not contain singletons.   
\end{itemize}
A collection $\omega\subset 2^\N$ is a non-principal ultrafilter if and only if the characteristic function $\chi_\omega$ of $\omega$ is a non-trivial, finitely additive measure on $2^\N$, such that $\chi_\omega(\{n\})=0$ for all $n\in \N$. Hence, we write $\omega(A)=1$ and $\omega(A)=0$ if $A\in \omega$ and $A\notin\omega$, respectively. We moreover occasionally say that a property (P) holds $\omega$-a.e. if $\{i\in\N:\ (P)\quad\mathrm{holds}\}\in\omega$.

Suppose $(X,d)$ is a metric space and $\omega$ a non-principal ultrafilter on $\N$. If $(x_i)\subset X$, we say that $x\in X$ is the \emph{ultralimit} (or $\omega$-\emph{limit}) of $(x_i)$, denoted $x=\lim_\omega x_i$, if \[
\{i\in\N\,:\,d(x_\omega,x_i)<\varepsilon\}\in\omega,\quad\textnormal{for all } \varepsilon>0.
\]
In a compact metric space $(X,d)$ any sequence $(x_i)\subset X$ has the ultralimit $x_\omega\in X$. In particular, any precompact sequence in a metric space $(X,d)$ has an ultralimit in $X$.

Let $(X_i,d_i,\bar x_i)$, $i\in\N$ be a sequence of pointed metric spaces. A sequence $(x_i)\in \prod_i X_i$, is said to be \textit{bounded} if 
\[
\sup_i d_i(x_i,\bar x_i)<+\infty.
\]
On the set of all bounded sequences we define an equivalence relation as follows: two bounded sequences $(x_i),(y_i)\in \prod_i X_i$ are \textit{equivalent} if 
\[
\lim_\omega d_i(x_i,y_i)=0.
\]
By $[x_i]_\omega$ we denote the equivalence class of the sequence $(x_i)$. We denote the set of all equivalence classes by $X_\omega$. Define
\[
d_\omega([x_i]_\omega,[y_i]_\omega)\coloneqq \lim_\omega d_i(x_i,y_i),
\]
for all bounded sequences $(x_i),(y_i)\in \prod_i X_i$.
Then $d_\omega$ is a metric on $X_\omega$. Metric space $(X_\omega,d_\omega, x_\omega)$, where $x_\omega\coloneqq [\bar x_i]_\omega$ is said to be the \textit{ultralimit (or $\omega$-limit)} of the sequence of pointed metric spaces $(X_i,d_i,\bar x_i)$. For a Banach space $V$ we denote by $V^\omega$ the $\omega$-limit (or \textit{ultrapower}) of the constant sequence $(X_i,\bar x_i)=(V,0)$, $i\in \N$, which is itself a Banach space. 

We refer the interested reader to \cite{DK2018,IW26,SP21} for more details and applications on ultralimits.

\subsection{Metric and measure theoretic notions}

\subsubsection*{Notation and conventions} \label{sec:MMTN} Let $(X,d)$ be a metric space, $x\in \X$, and $r>0$. We denote the \emph{open} and \emph{closed balls} of radius $r$ centered at $x$, respectively, by $B(x,r)$ and $\bar B(x,r)$. Moreover, $S(x,r)\coloneqq \bar B(x,r)\setminus B(x,r)$. The \emph{distance between sets} $A,B\subset X$ is denoted $\operatorname{dist}_X(A,B)$ or $d(A,B)$ and ${\rm dist}(x,B):=d(\{x\},B)$. For $\varepsilon>0$, we let $A^\varepsilon$ or $B(A,\varepsilon)$ denote the \emph{$\varepsilon$-neighbourhood} of $A$, defined as $\{y\in X: d(y,A)<\varepsilon\}$. We say that $A$ is \textit{$\varepsilon$-dense} in a set $B$ if $A\subset B^\varepsilon$. If $(X,d_X)$ and $(Y,d_Y)$ are two metric spaces, a map $f:(X,d_X) \to (Y,d_Y)$ is said to be 
\textit{Lipschitz} if there exists $L>0$ such that 
\begin{align}\label{eq:lip}
d_Y(f(x),f(\tilde x))\leq L d(x,\tilde x),\quad \textnormal{for every }x,\tilde x\in X.
\end{align}
We denote by $\operatorname{LIP}(f)$ the smallest constant in \eqref{eq:lip}.

Let $X=(X,d)$ be a metric space. We denote by $\mathscr{B}(X)$ the  \textit{Borel $\sigma$-algebra} on $\X$, generated by the topology of $(X,d)$. A Borel measure $\mu$ on $X$ is called \textit{boundedly finite} if $\mu(B(x,r))<+\infty$ for every $x
\in \X$ and $r>0$. Every boundedly finite Borel measure on a complete separable metric space is Radon, cf. \cite[Theorem 7.1.7]{Bogachev2018}. 

\subsubsection*{Porosity} Next, we recall the definition of a porous subset of a metric space, given in \cite{Bate2018}.

\begin{definition}[Porous set] \label{def:porous_set} Let \((X,d)\) be a metric space and $\eta>0$. We say that a set \(S\subset X\) is $\eta$-\textit{porous at} \(x_0\in S\) if there exists a sequence \((x_n)\subseteq X\) such that \(x_n\to x_0\) and 
\[
S\cap B(x_n,\eta\,{\rm d}(x_n,x_0))=\varnothing.
\]
A set \(S\) is said to be \textit{porous at} $x_0\in S$ if it is $\eta$-porous at $x_0$ for some $\eta>0$. Moreover, a set $S\subset X$ is said to be \textit{porous} if it is porous at each \(x\in S\).
\end{definition}

This notion is equivalent to porosity in the sense of Mor\'an--Preiss--Zaj\'icek \cite[Definition 1.1]{MMPZ2003}.

\begin{lemma}
Let $(X,d)$ be a metric space. A set $S\subset X$ is porous at $x_0$ in $S$ if and only if $\limsup_{r\to 0}\frac{\gamma(S,x_0,r)}{r}>0$, where 
\[
\gamma(S,x_0,r)=\sup\{s>0 : (\exists z\in \X) (\d(z,x_0)+s\leq r \,\land\, B(z,s)\cap S=\varnothing)\}.
\]
\end{lemma}
\begin{proof}
Let \(S\subset \X\) be porous in the sense of Definition \ref{def:porous_set} and \(x_0\in S\) be an arbitrary point with a corresponding constant \(\eta\) and a sequence \((x_n)\subset \X\). Let us define \(r_n\coloneqq (1+\eta) d(x_n,x_0)\), \(n\in\N\). Clearly, \(r_n\to 0\). Then we have 
\[
\gamma(S,x_0,r_n)\geq \eta d(x_n,x_0)=\frac{\eta}{1+\eta}r_n,\quad n\in \N.
\]
Thus, \(\limsup_{r\to 0} \frac{\gamma(S,x_0,r)}{r}\geq \frac{\eta}{1+\eta}>0\) and \(S\) is porous at \(x_0\) in the sense of \cite{MMPZ2003}. Conversely, suppose that \(\limsup_{r\to 0} \frac{\gamma(S,x_0,r)}{r}>0\). Then there exist a sequence \(s_n\searrow 0\) and a constant \(\eta\in (0,1)\) such that \(\gamma(S,x_0,s_n)>\eta s_n\), \(n\in \N\). Therefore, for every \(n\in \N\) there exists \(z_n\in \X\) such that \(\d(z_n,x_0)+\eta s_n\leq s_n\) and \(B(z_n,\eta s_n)\cap S=\varnothing\). Hence, there exist a sequence \((z_n)\subset \X\) and a constant \(\tilde\eta\coloneqq\frac{\eta}{1-\eta}>0\) such that \(\d(z_n,x_0)\leq (1-\eta)s_n\searrow 0\) and \(B(z_n,\tilde\eta d(z_n,x_0))\cap S=\varnothing\) for every \(n\in \N\). Therefore, \(S\) is porous at \(x_0\) in the sense of Definition \ref{def:porous_set}.
\end{proof}

\subsubsection*{Metric measure spaces and doubling conditions}

\begin{definition}[(Pointed) Metric measure space] A \textit{metric measure space} $(X,\mu)$ consists of a complete separable metric space $X=(X,d)$ and a boundedly finite Borel regular measure $\mu$ on $X$. A \textit{pointed metric measure space} is a triple $(X,\mu,\bar x)$, where $(X,\mu)$ is a metric measure space and $\bar x\in \X$. We denote by $\XX$ the collection of all pointed metric measure spaces. 
\end{definition}

A metric measure space $(X,\mu)$ is said to be \textit{doubling}, if there exists $C>0$ such that
\begin{align}\label{eq:doubl}
0<\mu(B(x,2r))\le C\mu(B(x,r))
\end{align}
for all $x\in X$ and $r>0$. We moreover say that $\mu$ is \textit{$(C,R)$-doubling along a subset} $A\subset X$ if \eqref{eq:doubl} holds for all $x\in A$ and $r\in (0,R]$.\footnote{In \cite{Bate2018} the terminology \emph{uniformly pointwise $(C,R)$-doubling at each point of $A$} is used.} The quantity
\begin{align}\label{eq:pt-doubl-const}
C_\mu(x)=\limsup_{r\to 0}\frac{\mu(B(x,2r))}{\mu(B(x,r))},\quad x\in \operatorname{spt}\mu
\end{align}
is called the pointwise doubling constant of $\mu$ at $x$, and $(X,\mu)$ is called \textit{pointwise doubling} if $C_\mu(x)<\infty$ for $\mu$-a.e. $x\in X$. If $(X,\mu)$ is pointwise doubling, then the Lebesgue differentiation theorem holds on $(X,\mu)$; in particular, for any $\mu$-measurable $E\subset X$ we have
\begin{align}\label{eq:leb-dens}
    \lim_{r\to 0}\frac{\mu(B(x,r)\cap E)}{\mu(B(x,r))}=1
\end{align}  
for $\mu$-a.e. $x\in E$. A point satisfying \eqref{eq:leb-dens} is called a \textit{Lebesgue density point} of the set $E$. If all porous sets in a metric measure space are \(\mu\)-null, then $\mu$ is pointwise doubling \cite[Theorem 3.6]{MMPZ2003}. 

\subsubsection*{Hausdorff convergence and weak convergence of measures}
Let $X$ be a complete separable metric space. We denote by $\mathcal M_{loc}(X)$ the space of boundedly finite Borel regular measures on $X$, and by $C_{bbs}(X)$ the space of bounded, continuous functions on $X$ with bounded support. A sequence $(\mu_i)\subset \mathcal M_{loc}(X)$ is said to \textit{converge weakly} to $\mu\in \mathcal M_{loc}(X)$, denoted $\mu_i\rightharpoonup \mu$, if
\begin{align}\label{def:weak*_conv}
\int_{\X} f\,d\mu_n\stackrel{n\to\infty}{\longrightarrow} \int_\X f\,d\mu
\end{align}
for every $f\in C_{bbs}(X)$.

Let $x\in X$, $L,r>0$. According to \cite[Definition 2.12]{Bate22}, for $\mu,\nu\in\mathcal{M}_{loc}(X)$, we denote
\begin{align}
    F_{(X,x)}^{L,r}(\mu,\nu)\coloneqq \sup\left\{\int g\,\d\mu-\int g\,\d\nu\,\left|\right.\,g:\X\to [-1,1], L\textnormal{-Lipschitz},\,{\rm spt}g\subset B(x,r)\right\}.
\end{align}
In addition, we have
\begin{align}  
\label{eq:F_x}
F_{(X,x)}(\mu,\nu)\coloneqq \inf\{\varepsilon\in (0,1/2)\,:\,F_{(X,x)}^{1/\varepsilon,1/\varepsilon}(\mu,\nu)<\varepsilon\},
\end{align}
if the above set is non-empty, and $F_{(X,x)}(\mu,\nu)=1/2$ is no such $\varepsilon\in (0,1/2)$ exists. When $X$ is a complete and separable metric space, $F_x$ is a complete and separable metric that metrizes weak convergence of measures. For the proof, see \cite[Lemma 2.14]{Bate22}. 

We say that a family $\mathcal F\subset \mathcal{M}_{loc}(\X)$ is \textit{boundedly tight}, provided that 
\begin{itemize}
   \item[(i)] $\sup_{\mu\in \mathcal F}\mu(B(x,R))<+\infty$, and
	\item[(ii)] there exists a compact set $K\subset X$ so that $\sup_{\mu\in\mathcal F}\mu(B(x,R)\setminus K)\le \varepsilon$
\end{itemize}
for all $\varepsilon, R>0$ and every $x\in \X$. Next, we recall Prokhorov's theorem characterizing the precompactness of the family $\mathcal F \subset \mathcal M_{loc}(X)$, with respect to convergence in duality with $C_{bbs}(X)$ (see \cite[Theorem 2.11]{Bate2018} or \cite[Theorem 2.3.4]{Bogachev2018}).

\begin{theorem}[Prokhorov]
\label{thm:prokhorov}
    Let $X$ be a complete and separable metric space. A family $\mathcal F\subset \mathcal{M}_{loc}(X)$ is precompact with respect to weak convergence (that is, every sequence of measures in $\mathcal F$ has a weakly convergent subsequence in $\mathcal{M}_{loc}(X)$) if and only if $\mathcal{F}$ is boundedly tight. 
\end{theorem}

The following elementary lemma establishes the existence of sequences converging to points in the support of a limit measure, and will be useful in the sequel.
\begin{lemma}
\label{lemma:support_point}
    Let \((X,d)\) be a complete and separable metric space. Let \((\mu_i)_{i\in\N\cup\{\infty\}}\) be a sequence of boundedly finite Borel regular measures concentrated on sets \(A_i\subset \X\), \(i\in\N\cup\{\infty\}\), respectively. Suppose that \(\mu_i\rightharpoonup \mu_\infty\). Then, if \(x\in {\rm spt}\mu_\infty\), there exists a sequence \(x_i\in A_i\) such that \(x_i\to x\), as \(i\to\infty\).
\end{lemma}

\begin{proof}
Fix any \(x\in {\rm spt}\mu_\infty\). Suppose the opposite, that there is no sequence \(x_i\in A_i\), \(i\in\N\) converging to \(x\). Then there exist \(\varepsilon>0\) and a subsequence \(A_{i_k}\), \(k\in\N\) such that \(B(x,\varepsilon)\cap A_{i_k}=\varnothing\), \(k\in\N\). Moreover, by weak convergence of measures, we have
\[
0<\mu_\infty(B(x,\varepsilon))\leq  \liminf_{k\to\infty}\mu_{i_k}(B(x,\varepsilon))=\liminf_{k\to \infty}\mu_{i_k}(B(x,\varepsilon)\cap A_{i_k})=0,
\]
so we get a contradiction. Hence, there exists a sequence \(x_i\in A_i\) such that \(x_i\to x\).
\end{proof}

\subsection{Convergence and tangents of metric measure spaces}
\label{sec:conv_tangents_mms}
For two pointed metric measure spaces \((X,\mu,\bar x)\) and \((Y,\nu,\bar y)\), we say that they are \textit{isomorphic} if there exists an isometric embedding \(\iota: {\rm spt}\mu\cup\{\bar x\}\to {\rm spt}\nu\cup\{\bar y\}\) such that \(\iota_*\mu=\nu\) and \(\iota(\bar x)=\bar y\). 
We write $(X,\mu,\bar x)\simeq (Y,\nu,\bar y)$. The equivalence class of the space \((X,\mu,\bar x)\) is denoted by \([X,\mu,\bar x]\).

 We denote by \(\XX_{pmG}\) the collection of all equivalence classes of pointed metric measure spaces. In \cite[Corollary 4.13]{Bate22}, it has been shown that there exists a complete and separable metric \(d_{pmG}\) on \(\XX_{pmG}\). More precisely, 
 for two given pointed metric measure spaces $(X,\mu_X,\bar x)$ and $(Y,\mu_Y,\bar y)$ we have 
 \begin{align}
 \label{eq:metric_d_pmG}
 d_{pmG}((X,\mu_X,\bar x),(Y,\mu_Y,\bar y))\coloneqq \inf_{(Z,\bar z)}F_{(Z,\bar z)}(\iota_\ast\mu_X,
\tilde\iota_\ast\mu_Y),
 \end{align}
where the infimum is taken among all pointed metric spaces $(Z,\bar z)$ and all isometric embeddings $\iota: (X,\bar x)\to (Z,\bar z)$ and $\tilde \iota:(Y,\bar y)\to (Z,\bar z)$. Here $F_{(Z,\bar z)}$ denotes the Prokhorov type metric on $\mathcal{M}_{loc}(Z)$ defined in \eqref{eq:F_x}, which metrizes weak convergence in $\mathcal{M}_{loc}(Z)$ if $Z$ is a complete and separable metric space. The quantity $d_{pmG}$ is a pseudometric on $\XX$ and two pointed metric measure spaces $(X,\mu,\bar x)$ and $(Y,\nu,\bar y)$ are isomorphic if and only if $d_{pmG}((X,\mu,\bar x),(Y,\nu,\bar y))=0$, so $d_{pmG}$ is a metric on $\XX_{pmG}$.
 
 The metric $d_{pmG}$ mentioned above metrizes the so-called pmG-convergence of pointed metric measure spaces introduced in \cite{GMS15}. The equivalence of convergence in $d_{pmG}$ and pmG-convergence has been proven in \cite[Proposition 2.20, Lemma 2.22]{Bate22}.  We recall the definition of pmG-convergence below. Note that even if the base point of each space in the sequence belongs to the support of the corresponding measure, the base point of the limit space does not need to belong to the support of the limiting measure.

\begin{definition}[Pointed measured Gromov convergence] \label{def:pmG_conv} A sequence \((X_i,\mu_i,\bar x_i)\) of pointed metric measure spaces converges to a pointed metric measure space $(X_\infty,\mu_\infty,\bar x_\infty)$ in the pointed measures Gromov (pmG) sense, if there exist a complete separable metric space \((W,d_W)\) and a sequence of isometric embeddings \(\iota_i\,:\, X_i\to W\), \(i\in \N\cup\{\infty\}\) such that 
\begin{align}
&\lim_{i\to\infty}\iota_i(\bar x_i)=\iota_\infty(\bar x_{\infty})\\
&\iota_{i\ast}\mu_i\rightharpoonup \iota_{\infty\ast}\mu_\infty\quad \text{as } i\to \infty.
\end{align}
We write $(X_i,\mu_i,\bar x_i)\overset{pmG}\longrightarrow (X_\infty,\mu_\infty,\bar x_\infty)$.
\end{definition}

\begin{remark}\label{rem:precompactinXX}
By \cite[Theorem 11.4]{SP21}, a sequence \(((X_i,\mu_i,\bar x_i))\subset\XX\) is precompact with respect to pmG-convergence if and only if it is
\begin{enumerate}
    \item[(i)] \textit{uniformly boundedly finite} (UBF): for any $R>0$ we have $\sup_{i\in\N}\mu_i(B(\bar x_i,R))<+\infty$, and
    \item[(ii)] \textit{boundedly measure-theoretically totally bounded} (BMTB): for every $R,r,\varepsilon>0$ there exist $N\in\N$ and points $(x_n^i)_{n=1}^N\subset X_i$ such that $\sup_{i\in\N}\mu_i(B(\bar x_i,R)\setminus \bigcup_{n=1}^N B(x_n^i,r))\leq \varepsilon$. 
\end{enumerate}

Equivalently, in the terminology of \cite[Corollary 3.22]{GMS15}, a sequence of spaces with the base point in the support of the corresponding measure is precompact if and only if it is UBF, BMTB, and $\inf_{i\in\N}\mu_i(B(\bar x_i,R))>0$, for every $R>0$. The last \textit{bounded away from zero} condition is equivalent to the fact that the base point in the limit space is in ${\rm spt}\mu_\infty$.
\end{remark}

\begin{definition}[Pointed measured Gromov-Hausdorff convergence] 
\label{def:pmGH-conv}
The sequence $(X_i,\mu_i,\bar x_i)$, $i\in \N$ of pointed metric measure spaces converges in the \textit{pointed measured Gromov-Hausdorff} sense (in short, pmGH-sense) to the pointed metric measure space $(X_\infty,\mu_\infty,\bar x_\infty)$ if there exists complete separable pointed metric space $(Z,\bar z)$ and isometric embeddings $\iota_i:(X_i,\bar x_i)\to (Z,\bar z)$, $i\in \N\cup\{\infty\}$ such that $\iota_{i\ast}\mu_i\rightharpoonup \iota_{\infty\ast}\mu_\infty$ and $H_{(Z,\bar z)}(X_i,X_\infty)\overset{i\to\infty}{\longrightarrow} 0$, where $H_{(Z,\bar z)}$ is infimum among all $\varepsilon\in (0,1/2)$ 
such that
\[
\iota_i(X_i)\cap B(\bar z,1/\varepsilon)\subset B(\iota_\infty(X_\infty),\varepsilon)\,\land\, \iota_\infty(X_\infty)\cap B(\bar z,1/\varepsilon)\subset B(\iota_i(X_i),\varepsilon),
\]
if such $\varepsilon(0,1/2)$ exists, and $H_{(Z,\bar z)}(X_i,X_\infty)=1/2$ if no such $\varepsilon\in (0,1/2)$ exists.
\end{definition}

In \cite[Proposition 3.30]{GMS15} it is proven that $(X_i,\mu_i,\bar x_i)\overset{pmGH}{\longrightarrow} (X_\infty,\mu_\infty,\bar x_\infty)$ implies the convergence $(X_i,\mu_i,\bar x_i)\overset{pmG}{\longrightarrow} (X_\infty,\mu_\infty,\bar x_\infty)$. If we assume that the spaces $(X_i,\mu_i,\bar x_i)$, $i\in \N$ are doubling with a common doubling constant and ${\rm spt}\mu_\infty=X_\infty$, then we also have the opposite implication, i.e. the two types of convergence are equivalent (see \cite[Proposition 3.33]{GMS15}).

\subsubsection*{Tangents}
 We fix $\varphi$ as at the beginning of Section 1.2. In particular, there are constants $0<a\le 1\le A$ and $c,M>0$ such that $c\chi_{[0,a]}\le\varphi\le M\chi_{[0,A)}$. Given a metric measure space $X=(X,\mu)$ and $x\in X$, $r>0$, denote 
\begin{align*}
\mu_{ x,r}:=\frac{\mu}{\varphi_\mu( x,r)},\quad \mathrm{where}\quad \varphi_\mu( x,r)=\int\varphi\Big(\frac{d( x,y)}{r}\Big)\ud\mu(y).
\end{align*}
We denote $T_{x,r}X:=(r\inv X,\mu_{x,r},x)$, where $r\inv X$ is the metric space $(X,r\inv d)$.

Observe that $c\mu(B(x,ar))\le\varphi_\mu(x,r)\le M\mu(B(x,Ar))$. Consequently,
\begin{align}\label{eq:normalizing-const}
0<\liminf_{r\to0}\frac{\mu(B(x,r))}{\varphi_\mu(x,r)}
\le\limsup_{r\to0}\frac{\mu(B(x,r))}{\varphi_\mu(x,r)}<\infty
\end{align}
whenever $C_\mu(x)<\infty$. More generally, for each fixed $S>0$, the ratios $\mu(B(x,Sr))/\varphi_\mu(x,r)$ are bounded for sufficiently small $r$. 

\begin{definition}[pmG-tangent] \label{def:pmG_tan} Let \((X,\mu, \bar x)\) be a pointed metric measure space. A pointed metric measure space \((Y,\nu,o)\) is said to be a \textit{pointed measured Gromov tangent} (pmG-tangent in short) of \((X,\mu,\bar x)\) if there exists a sequence of scalars \(r_i\searrow 0\) such that \((Y,\nu,o)\) is the pmG-limit of \(T_{\bar x,r_i}X\), \(i\in\N\). We denote by \({\rm Tan}_{pmG}(X,\mu,\bar x)\) the family of all pmG-tangents of \((X,\mu,\bar x)\).
\end{definition}

\begin{definition}[pmGH-tangent] 
Let \((X,\mu, \bar x)\) be a pointed metric measure space. Then the pointed metric measure space \((Y,\nu,o)\) is a \textit{pointed measured Gromov--Hausdorff tangent} (pmGH-tangent in short) of \((X,\mu,\bar x)\) if there exists a sequence of scalars \(r_i\searrow 0\) such that \((Y,\nu,o)\) is the pmGH-limit of \(T_{\bar x,r_i}X\), \(i\in\N\). We denote by \({\rm Tan}_{pmGH}(X,\mu,\bar x)\) the family of all pmGH-tangents of \((X,\mu,\bar x)\).
\end{definition}

\subsection{Mappings between metric spaces}
\label{sec:mappings_ms}

Let $(X,\mu)$ be a metric measure space and $Y$ a complete metric space. A mapping package is an \textit{essentially separably valued} (i.e. there exist a Borel set $N\subset X$, $\mu(N)=0$, and a separable set $Y_0\subset Y$ such that $f(X\setminus N)\subset Y_0$) $\mu$-measurable map $f:(X,\mu)\to Y$. If $(X,\mu,\bar x)\in\XX$ and $(Y,y)$ is a complete pointed metric space, we say that a mapping package $f:(X,\mu)\to Y$ is pointed if $f(\bar x)=y$. If no basepoint in $Y$ is specified we consider a mapping package $f:(X,\mu,\bar x)\to Y$ pointed by setting $y=f(\bar x)$ as basepoint in $Y$. We denote by $\M$ the collection of all pointed mapping packages. Moreover, given $L>0$ and collections $\mathcal F\subset \XX$ and $\mathcal C$ of complete (pointed) metric spaces, we denote
\begin{align*}
    \M(\mathcal F,\mathcal C)=\{(f:X\to Y)\in \M:\ X\in \mathcal F,\, Y\in \mathcal C\}\quad\mathrm{and}\quad \M_L(\mathcal F,\mathcal C)=\M_L\cap \M(\mathcal F,\mathcal C),
\end{align*}
where $\M_L=\{f\in\M:\ \operatorname{LIP}(f)\le L\}$.

\subsubsection*{Graphs} Let $f:(X,\mu,\bar x)\to Y$ be a mapping package. Consider the essentially separably valued $\mu$-measurable map $\bar f:X\to X\times Y$ given by $\bar f(x)=(x,f(x))$. The graph of $f$, denoted $G(f)$, is the pointed metric measure space 
\begin{align}\label{eq:graph}
    (\{\bar f(\bar x)\}\cup\operatorname{spt}\mu_f,\mu_f,\bar f(\bar x))\in\XX,\quad \mathrm{where}\quad \mu_f=\bar f_\ast\mu.
\end{align}
Since $\bar f$ is essentially separably valued, $\mu_f$ is concentrated on a separable subset of $X\times Y$. Moreover, since $\bar f\inv B(\bar f(x),r)\subset B(x,r)$, $\mu_f$ is boundedly finite. Thus, $G(f)$ is indeed a pointed metric measure space.

\begin{remark}
If $f:(X,\mu)\to Y$ is an $L$-Lipschitz map, then 
\begin{align*}
B\big(x,\frac{r}{\sqrt{1+L^2}}\big)\subset \bar f\inv B(\bar f(x),r)\subset B(x,r),\quad x\in X, \,r>0.
\end{align*}
In particular $\operatorname{spt}\mu_f=\bar f(\operatorname{spt}\mu)=\{(x,f(x)):\ x\in\operatorname{spt}\mu\}\subset X\times Y$.
\end{remark}

\subsubsection*{Convergence of mapping packages} Here we define notions of pointed measured Gromov (pmG) and pointed measured Gromov--Hausdorff (pmGH) convergence of mapping packages for equi-Lipschitz sequences into a fixed proper metric space. Definition \ref{def:equi-lip-maps-conv} below stipulates the convergence $X_i\stackrel{pmG}{\longrightarrow}X_\infty$ (resp. $X_i\stackrel{pmGH}{\longrightarrow}X_\infty$) of the domain spaces, along with a suitable uniform convergence of the mappings.

Let $L>0$ and $Y$ be a proper metric space.
\begin{definition}\label{def:equi-lip-maps-conv}
A sequence $f_i:(X_i,\mu_i,\bar x_i)\to Y$ of $L$-Lipschitz mapping packages converges to a mapping package $f_\infty:(X_\infty,\mu_\infty,\bar x_\infty)\to Y$ 
\begin{itemize}
    \item[(pmG)] in the pointed measured Gromov sense, denoted $f_i\stackrel{pmG}{\longrightarrow}f_\infty$, if there exist isometric embeddings $\iota_i:X_i\to (Z,\bar z)$ ($i\in\N\cup\{\infty\}$) into a complete separable pointed metric space so that $F_{(Z,\bar z)}(\iota_{i\ast}\mu_i,\iota_{\infty\ast}\mu_\infty)\stackrel{i\to\infty}{\longrightarrow}0$ and, for each $x\in \{\bar x_\infty\}\cup\operatorname{spt}\mu_\infty$, we have
    \begin{align}\label{eq:map-precpt}
        f_i(x_i)\stackrel{i\to \infty}{\longrightarrow} f_\infty(x)\quad\mathrm{whenever}\quad x_i\in \{\bar x_i\}\cup\operatorname{spt}\mu_i, \quad \iota_i(x_i)\stackrel{i\to\infty}{\longrightarrow}\iota_\infty(x);
    \end{align}
    \item[(pmGH)] in the pointed measured Gromov--Hausdorff sense, denoted $f_i\stackrel{pmGH}{\longrightarrow}f_\infty$, if there exist isometric embeddings $\iota_i:X_i\to (Z,\bar z)$ ($i\in\N\cup\{\infty\}$) into a complete separable pointed metric space so that $H_{(Z,\bar z)}(\iota_i(X_i),\iota_\infty(X_\infty))+F_{(Z,\bar z)}(\iota_{i\ast}\mu_i,\iota_{\infty\ast}\mu_\infty)\stackrel{i\to\infty}{\longrightarrow}0$ and, for each $x\in X_\infty$, we have
    \begin{align}\label{eq:map-precpt-pmGH}
        f_i(x_i)\stackrel{i\to \infty}{\longrightarrow} f_\infty(x)\quad\mathrm{whenever}\quad x_i\in X_i, \quad \iota_i(x_i)\stackrel{i\to\infty}{\longrightarrow}\iota_\infty(x).
    \end{align}
\end{itemize}
\end{definition}

\begin{remark}\label{rmk:pmGH-vs-guy}
An equi-Lipschitz sequence of maps $f_i:(X_i,\mu_i,\bar x_i)\to \R^n$ pmGH-converges to $f_\infty:(X_\infty,\mu_\infty,\bar x_\infty)\to \R^n$ if and only if it converges in the sense of \cite[Section 2]{DG2015} and the approximate isometries $\phi_i:X_i\to X_\infty$ in that definition can be chosen to satisfy $\phi_{i\ast}\mu_i\stackrel{i\to\infty}{\rightharpoonup}\mu_\infty$.
\end{remark}

\subsubsection*{Haj\l asz--Sobolev maps}
Let $p\in(1,\infty]$. Given a mapping package $f:(X,\mu,\bar x)\to Y$, a Borel function $g\in L^p_{loc}(\mu)$ is a Haj\l asz gradient of $f$ if there exists a $\mu$-null set $N\subset X$ with 
\begin{align}\label{eq:hajlasz}
    d(f(x),f(y))\le d(x,y)[g(x)+g(y)],\quad x,y\in X\setminus N.
\end{align}
Here $L^p_{loc}(\mu)$ denotes the $\mu$-measurable functions $g$ on $X$ with $\displaystyle\int_{B(\bar x,R)}g^p\ud\mu<\infty$ for all $R>0$. The Haj\l asz--Sobolev space $M^{1,p}_{loc}(X,Y)$ consists of maps $f:(X,\mu)\to Y$ with a Haj\l asz gradient $g\in L^p_{loc}(\mu)$. We call elements of $M^{1,p}_{loc}(X,Y)$ Haj\l asz--Sobolev maps. Observe that, by redefining $g$ on a $\mu$-null set we may assume that \eqref{eq:hajlasz} holds for all $x,y\in X$. Moreover, if $E_\lambda=\{g\le \lambda\}$, then $f|_{E_\lambda}$ is $2\lambda$-Lipschitz, and thus has a $2\lambda$-extension to $\overline{E_\lambda}$. A Haj\l asz--Sobolev map $f$ with Haj\l asz gradient $g$ for which $f|_{\overline{E_\lambda}}$ is $2\lambda$-Lipschitz is called a \emph{good} (or \emph{$g$-good}) representative of $\mu$. For simplicity we consider pointwise defined maps and not equivalence classes in this article, however our notions and results are independent of the chosen representative. We do not insist on this point. 

\section{Construction of the ultralimit mapping}
\label{sec:AR-ASC}

In this section we give the construction of the ultralimit in Theorem \ref{thm:ultralimit+graph}. We establish the properties claimed in Theorem \ref{thm:ultralimit+graph} apart from (iii), whose proof we postpone to Section \ref{sec:mapping-package}. In the sequel, by a precompact sequence $(X_i)$ in $(\XX,d_{pmG})$ we mean a sequence of pointed metric measure spaces such that every of its subsequences has a further pmG-convergent subsequence.

\begin{lemma}\label{lem:precpt}
    Let $(X_i)\subset \XX$ be a precompact sequence in $(\XX,d_{pmG})$. Then there exists a complete separable metric space $(Z,z)$ and pointed isometric embeddings $\iota_i:X_i\to Z$ so that $\iota_i(x_i)=z$ for all $i\in\N$ and $\bar\mu_i:=\iota_{i\ast}\mu_i$ is precompact in $\mathcal M_{loc}(Z)$.
\end{lemma}
In the proof, we use the metric $ d_{pmG}$ defined in \eqref{eq:metric_d_pmG}. 
\begin{proof}
For each pair $i,j\in\N$ let $\tilde\iota_i:X_i\to Z_{i,j}$, $\tilde\iota_j:X_j\to Z_{i,j}$ be isometric embeddings with $\tilde\iota_i(x_i)=z_{i,j}=\tilde\iota_j(x_j)$ such that 
\[
F_{(Z_{i,j},z_{i,j})}(\tilde\iota_{i\ast}\mu_i,\tilde\iota_{j\ast}\mu_j)<2d_{pmG} (X_i,X_j).
\]
Let $(Z,z)$ be the completion of the space obtained by \cite[Lemma 2.17]{Bate22}. Note that $z=\iota_i(x_i)$ for all $i$ by construction. There exists isometric embeddings $\iota_i:X_i\to Z$ so that $d_Z(\iota_i(\tilde x_i),\iota_j(\tilde x_j))\le d_{i,j}(\tilde \iota_i(\tilde x_i),\tilde\iota_j(\tilde x_j))$ for all $\tilde x_i\in X_i$, $\tilde x_j\in X_j$, $i\le j$. Thus $F_{(Z,z)}\le F_{(Z_{i,j},z_{i,j})}$ for $i\le j$. 

We prove that the collection $(\iota_{i\ast}\mu_i)\subset \mathcal M_{loc}(Z)$ is totally bounded with respect to $F_{(Z,z)}$. Indeed, given $\varepsilon>0$, let $M\in \N$ and $\iota_m$, $m=1,\ldots,  M$ be such that 
\[
(X_i)\subset\bigcup_{m=1}^MB_{d_{pmG}}(X_{j_m},\varepsilon/4).
\]
Define $I_m:=\{i\in\N: X_i\in B_{d_{pmG}}(X_{j_m},\varepsilon/4)\}$ and note that $\N=\bigcup_{m=1}^M I_m$. 
For each $m$ let $i_m\in I_m$ be the smallest element. Then $B_{d_pmG}(X_{j_m},\varepsilon/4)\subset B_{d_{pmG}}(X_{i_m},\varepsilon/2)$, $m=1,\ldots,M$ and moreover 
\[
F_{(Z,z)}(\iota_{i_m\ast}\mu_{i_m},\iota_{i\ast}\mu_i)\le 2d_{pmG}(X_{i_m},X_i)<\varepsilon,\quad i\in I_m.
\]
Thus $\displaystyle (\iota_{i\ast}\mu_i)\subset \bigcup_{m=1}^MB_{F_{(Z,z)}}(\iota_{i_m\ast}\mu_{i_m},\varepsilon)$. This completes the proof of the Lemma.
\end{proof}

\begin{lemma}\label{lem:omega-limit-p-int-funct}
Let $Z$ be a complete separable metric space. Suppose $(\tilde \mu_i)\subset \mathcal M_{loc}(Z)$ be a precompact family of measures, $z\in Z$ and $\tilde g_i\in L^p_{loc}(\tilde\mu_i)$ satisfy
    \begin{align*}
        \sup_i\int_{B(z,R)}\tilde g_i^p\ud\tilde\mu_i<\infty,\quad R>0.
    \end{align*}
    Then there exists $\tilde g_\omega\in L^p_{loc}(\tilde\mu_\omega)$ such that 
    \begin{align*}
        \int \varphi \tilde g_\omega\ud\tilde\mu_\omega=\lim_\omega\int \varphi \tilde g_i\ud\tilde\mu_i,\quad \varphi\in C_{bbs}(Z)
    \end{align*}
    and moreover
    \begin{align*}
        \int_{B(z,R)}\tilde g_\omega^p\ud\tilde \mu_\omega\le \lim_\omega \int_{B(z,R)}\tilde g_i^p\ud\tilde\mu_i,\quad R>0.
    \end{align*}
Here $\tilde\mu_\omega\in \mathcal M_{loc}(Z)$ is the $\omega$-limit of the sequence $(\tilde\mu_i)$.
\end{lemma}
\begin{proof}
Consider the functional
\begin{align*}
    L_\omega(\varphi)=\lim_\omega\int \varphi\tilde g_i\ud\tilde\mu_i,\quad \varphi\in C_{bbs}(Z).
\end{align*}
Since $\lim_\omega F_{(Z,z)}(\tilde \mu_i,\tilde\mu_\omega)=0$, we have 
\begin{align}\label{eq:omega-limit-meas}
\int \varphi\ud\tilde\mu_\omega=\lim_\omega\int\varphi\ud\tilde\mu_i
\end{align}
for every $\varphi\in \mathcal{C}_{bbs}(Z)$, due to the continuity of the function $\mathcal{M}_{loc}(Z)\ni \mu\mapsto \int\varphi\,\d\mu$ for a fixed $\varphi\in \mathcal{C}_{bbs}(Z)$. Thus, we have
\begin{align}\label{eq:norm-est}
    |L_\omega(\varphi)|&\le \lim_\omega\Big(\int |\varphi|^q\ud\tilde\mu_i\Big)^{1/q}\Big(\int_{B(z,R)} \tilde g_i^p\ud\tilde \mu_i\Big)^{1/p}\nonumber\\
    &\le \Big(\int |\varphi|^q\ud\tilde\mu_\omega\Big)^{1/q}\cdot\lim_\omega\Big(\int_{B(z,R)}\tilde g_i^p\ud\tilde\mu_i\Big)^{1/p}
\end{align}
for all $\varphi\in C_{bbs}(Z)$ supported in a ball $B(z,R)$. Here $q\in [1,\infty)$ is the dual exponent of $p\in (1,\infty]$. Since $C_{bbs}(Z)$ is dense in $L^q_{loc}(\tilde\mu_\omega)$\footnote{in the sense that for all $h\in L^q_{loc}(\tilde\mu_\omega)$ there exists $(\varphi_i)\subset C_{bbs}(Z)$ so that $\int_{B(z,R)}|\varphi_i-h|^q\ud\tilde\mu_\omega\to 0$ as $i\to\infty$ for all $R$.}, it follows that $L_\omega$ extends to a bounded linear functional on $L^q_{loc}(\tilde\mu_\omega)$. Therefore, there exists a function $\tilde g_\omega\in L^p_{loc}(\tilde\mu_\omega)$ so that
\begin{align*}
    L_\omega(\varphi)=\int \varphi \tilde g_\omega\ud\tilde\mu_\omega,\quad \varphi\in C_{bbs}(Z).
\end{align*}
Moreover, from \eqref{eq:norm-est} and the dual characterization of the $L^p$-norm it follows that 
\begin{equation*}
    \int_{B(z,R)} \tilde g_\omega^p\ud\tilde\mu_\omega\le \lim_\omega\int_{B(z,R)} \tilde g_i^p \ud \tilde\mu_i,\quad R>0. \qedhere 
\end{equation*} 
\end{proof}

\begin{lemma}\label{lem:omega-ineq}
    Let $Z$ be a complete separable metric space. Let $\tilde\mu_i$ and $\tilde g_i$ be as in Lemma \ref{lem:omega-limit-p-int-funct}, $\tilde\mu_\omega=\lim_\omega\tilde\mu_i$ in $\mathcal M_{loc}(Z)$, and $\tilde g_\omega\in L_{loc}^p(\mu_\omega)$ given by Lemma \ref{lem:omega-limit-p-int-funct}. For any $\lambda>0$ we have $\{\tilde g_\omega < \lambda\}$ is $\mu_\omega$-essentially contained in $\lim_\omega\{\tilde g_i\leq \lambda\}$, i.e. $\mu_\omega(\{\tilde g_\omega< \lambda\}\setminus\lim_\omega\{\tilde g_i\leq \lambda\})=0$.
\end{lemma}
\begin{proof}
Let $z\in Z\setminus \lim_\omega\{\tilde g_i\leq \lambda\}$. We first claim there exists $r_z>0$ so that $\{i\in\N:\ \tilde g_i>\lambda\,\,\mathrm{on}\,\, B(z,r_z)\}\in\omega$. Indeed, otherwise, for every $n\in\N$ we have
\begin{align*}
    \mathcal{F}_n\coloneqq\{i\in\N:\ \exists y_i^{n}\in B(z,n\inv)\,\, \mathrm{with}\,\, \tilde g_i(y_i^{n})\leq \lambda\}\in \omega.
\end{align*}
Define $y_i=y_i^{n_i}$ where $n_i\in\N$ is maximal $n\in\N$ such that $i\in \mathcal{F}_n$ and $y_i=y_{\tilde i}^1$ for arbitrary $\tilde i\in \mathcal{F}_1$ if such $n_i$ does not exist. This implies $z=\lim_\omega y_i\in \lim_\omega\{\tilde g_i\leq \lambda\}$, which is a contradiction. Now for any non-negative $\varphi\in C_{bbs}(Z)$ supported in $B(z,r_z)$ we have
\begin{align*}
\lambda\int\varphi\ud\tilde\mu_\omega=\lim_\omega\lambda\int\varphi\ud\tilde\mu_i\le \lim_\omega\int \tilde g_i\varphi\ud\tilde\mu_i=\int \varphi \tilde g_\omega\ud\tilde\mu_\omega.
\end{align*}
This implies that $\tilde g_\omega\ge \lambda$ $\mu_\omega$-a.e. on $B(z,r_z)$. This shows that $Z\setminus \lim_\omega\{\tilde g_i\le \lambda\}$ is $\mu_\omega$-essentially contained in $\{\tilde g_\omega\ge \lambda\}$, which implies the claim.
\end{proof}

For the proof of Theorem  \ref{thm:constr_of_ultralim} below, recall the setting of Theorem \ref{thm:ultralimit+graph}: $(X_i=(X_i,\mu_i,\bar x_i))$ is a precompact sequence in $\XX$ with $\omega$-limit $X_\omega=(X_\omega,\mu_\omega,\bar x_\omega)$, and $f_i:X_i\to Y_i$ are Haj\l asz--Sobolev maps into complete pointed spaces with Haj\l asz gradients $g_i\in L^p_{loc}(\mu_i)$ satisfying $\bar x_i\in \{g_i\le t\}$ for some $t$ as well as the uniform enegry bound \eqref{eq:energy-unif-bdd}.

\begin{theorem}\label{thm:constr_of_ultralim} 
Under the assumptions of Theorem \ref{thm:ultralimit+graph}, there exists a Haj\l asz--Sobolev mapping $f_\omega:X_\omega\to Y_\omega$ and $g_\omega\in L^p_{loc}(\mu_\omega)$ satisfying (i), (ii), and -- given pointed isometric embeddings $\iota_i:X_i\to (W,w)$ ($i\in \N\cup\{\omega\}$) with $\lim_\omega\iota_{i\ast}\mu_i=\iota_{\omega\ast}\mu_\omega$ -- \eqref{eq:ultralimit+graph} in Theorem \ref{thm:ultralimit+graph}.
\end{theorem}

\begin{remark}
Note that, in Theorem \ref{thm:ultralimit+graph} (and in Theorem \ref{thm:constr_of_ultralim}), $X_\omega$ is the $\omega$-limit of the precompact sequence $(X_i)$ in $\XX$, while $Y_\omega$  
is the pointed ultralimit of pointed metric spaces defined in Section \ref{sec:ultralimits}.
\end{remark}

\begin{proof}[Proof of Theorem \ref{thm:constr_of_ultralim}] We first prove the existence of $f_\omega:X_\omega\to Y_\omega$ and $g\in L^p_{loc}(\mu_\omega)$ as in the claim. Fix isometric embeddings $\iota_i:X_i\to (W,w)$ as in Lemma \ref{lem:precpt}, and let $X_\omega=({\rm spt}\mu_\omega\cup\{w\},\mu_\omega,w)$ be the $\omega$-limit of $(\iota_i(X_i),\iota_{i\ast}\mu_i,\iota_i(\bar x_i))$ in $\XX$, where $\mu_\omega$ is $\omega$-limit of $(\iota_{i\ast}\mu_i)$ in $\mathcal{M}_{loc}(W)$. Thus the isometric embedding $\iota_\omega$ is simply the inclusion map and $\bar x_\omega=w$. By Lemma \ref{lem:omega-limit-p-int-funct} there exists $g_\omega\in L^p_{loc}(\mu_\omega)$ such that 
\[
\int_{B(\bar x_\omega,R)}g_\omega^p\,\d\mu_\omega\leq \lim_\omega \int_{B(\bar x_i,R)}g_i^p\,\d\mu_i, \quad \text{for all }R>0.
\]
Given $\rho>1$,  we define $E_i^k\coloneqq \{g_i\leq \rho^k\}$ and $E_\omega^k\coloneqq \{\lim_\omega w_i\in W\,:\,w_i\in \iota_i(E_i^k)\}$. Fix an arbitrary $R>0$. By Lemma \ref{lem:omega-ineq} and Chebyshev's inequality, we have
\begin{align*}
\mu_\omega(B(\bar x_\omega,R)\setminus E_\omega^k) & \leq \mu_\omega (B(\bar x_\omega,R)\setminus \{g_\omega\leq \rho^k\})\leq \frac{1}{\rho^{kp}}\int_{B(\bar x_\omega,R)}g_\omega^p\,\d\mu_\omega\\
& \leq \frac{1}{\rho^{kp}}\lim_\omega\int_{B(\bar x_i,R)}g_i^p\,\d\mu_i\leq \frac{1}{\rho^{kp}}C(R),
\end{align*}
where $C(R)=\sup_i \int_{B(\bar x_i,R)}g_i^p\,\d\mu_i$. Consequently $\mu_\omega\Big(X_\omega\setminus \bigcup_kE_\omega^k\Big)=0$. 
By \eqref{eq:hajlasz}, $f_i|_{E_i^k}$ is $2\rho^k$-Lipschitz. We moreover have $E_\omega^k\subset E_\omega^{k+1}$ and $(f_i|_{E_i^{k+1}})|_{E_i^{k}}=f_i|_{E_i^k}$ for all $k$ and $i$. We may thus define
\begin{align*}
f_\omega(x)=[f_i(x_i)]_\omega,
\end{align*}
whenever $x\in E_\omega^k$ for some $k\in\N$, i.e. there is $k$ such that $x_i\in E_i^k$ for all $i$, and $\lim_\omega\iota_i(x_i)=x$. Note that $f_\omega$ is defined on $\displaystyle \bigcup_kE_\omega^k$ which has full $\mu_\omega$-measure.  Moreover $f_\omega$ is well defined: let $x\in \bigcup_kE_\omega^k$ and let $\tilde x_i\in E_i^k$, $\hat x_i\in E_i^m$ for some $k$,$m$ and all $i$ be such that $\lim_\omega\iota_i(\tilde x_i)=x=\lim_\omega\iota_i(\hat x_i)$. Without loss of generality assume $m\ge k$. Then 
\begin{align*}
    \lim_\omega d(f_i(\tilde x_i),f_i(\hat x_i))\le \lim_\omega d_i(\tilde x_i,\hat x_i)[g_i(\tilde x_i)+g_i(\hat x_i)]\le 2\rho^m\lim_\omega d_W(\iota_i(\tilde x_i),\iota_i(\hat x_i))=0,
\end{align*}
so that $[f_i(\tilde x_i)]_\omega=[f_i(\hat x_i)]_\omega$. Equation \eqref{eq:ultralimit+graph} is satisfied by construction together with the observation that $\{g_i\le \lambda\}\subset E_i^k$ and  $\mu_\omega(\{g_\omega<\lambda\}\setminus E_\omega^k)=0$ for $k\ge \log\lambda/\log\rho$.

Next we prove (i), i.e. that $g_\omega$ is a Haj\l asz gradient of $f_\omega$. 
By Lemma \ref{lem:omega-ineq} we have that $g_\omega\ge \rho^k$ $\mu_\omega$-a.e. on a neighbourhood $x$ if $x\notin E_\omega^k$. 
Now let $C\subset \operatorname{spt}\mu_\omega$ be a closed set so that $g_\omega|_C$ is continuous and $\mu_\omega(X_\omega\setminus C)<\varepsilon$ (see e.g. \cite[Lusin's Theorem 12.8]{AliprantisBorder99}). For $x,y\in C\cap \bigcup_kE_\omega^k$ which belong to ${\rm spt}(\mu_\omega|_C)$, let $k$ be the smallest integer such that $x\in E_\omega^k$ and $m$ the smallest integer such that $y\in E_\omega^m$. Then $g_\omega(x)\ge \rho^{k-1}$ and $g_\omega(y)\ge \rho^{m-1}$ by the claim above (and continuity). We may estimate
\begin{align}\label{eq:haj-est}
d_{Y^\omega}(f_\omega(x),f_\omega(y))&=\lim_\omega d_{Y_i}(f_i(\tilde x_i),f_i(\tilde y_i))\le \lim_\omega d_i(\tilde x_i,\tilde y_i)[g_i(\tilde x_i)+g_i(\tilde y_i)]\le d(x,y)[\rho^k+\rho^m]\nonumber\\
&\le \rho d(x,y)(g_\omega(x)+g_\omega(y)).
\end{align}
By the arbitrariness of $C$ it follows that \eqref{eq:haj-est} holds $\mu_\omega$-a.e.. Thus $\rho g_\omega$ is a Haj\l asz gradient of $f_\omega$. Since $\rho>1$ is arbitrary, it follows that $g_\omega$ is a Haj\l asz gradient of $f_\omega$. The inequality (ii) is an immediate consequence of Lemma \ref{lem:omega-limit-p-int-funct}. 
\end{proof}

\begin{remark}\label{eq:f_omega-def-set}
In the construction of $f_\omega$ above, note that \eqref{eq:ultralimit+graph} holds for every $x\in \bigcup_kE_\omega^k$. 
\end{remark}

\begin{remark}
It is a straightforward consequence of \eqref{eq:ultralimit+graph} that $f_\omega$ constructed in Theorem \ref{thm:constr_of_ultralim} satisfies the following property. If $h_i:Y_i\to Z_i$ are pointed $L$-Lipschitz maps with $\omega$-limit $h_\omega:Y_\omega\to Z_\omega$ in the sense of \eqref{eq:lip-ultralim}, then $h_i\circ f_i:X_i\to Z_i$ are Haj\l asz--Sobolev maps with Haj\l asz gradient $Lg_i$ satisfying \eqref{eq:energy-unif-bdd}, and their ultralimit $(h_i\circ f_i)_\omega\in M^{1,p}_{loc}(X_\omega,Z_\omega)$ satisfies
\begin{align}\label{eq:postcomp}
	(h_i\circ f_i)_\omega=h_\omega\circ f_\omega\quad\mu_\omega\mathrm{-a.e.\ on\ }X_\omega.
\end{align}
\end{remark}

\section{Mapping packages}\label{sec:mapping-package}

In this section, we establish several basic properties of the pmG-distance in Definition \ref{def:pmG-conv-new}, and the notion of pmG-convergence it gives rise to on the collection of mapping packages $\M$. Recall that, given two mapping packages $f_i:X_i\to Y_i$ with $X_i=(X_i,\mu_i,\bar x_i)$ ($i=0,1$), their pmG-distance is given by
\begin{align}\label{eq:pmg-dist}
	d_{pmG}(f_0,f_1)=\inf\{ F_{(Z\times \widetilde Z,(z,\tilde z))}((\iota_0,\tilde\iota_0\circ f_0)_\ast\mu_0,(\iota_1,\tilde\iota_1\circ f_1)_\ast\mu_1)\},
\end{align}
where the infimum is taken over all isometric embeddings $\iota_i:X_i\to (Z,z)$ and $\tilde\iota_i:Y_i\to (\widetilde Z,\tilde z)$, $i=0,1$ into complete metric spaces $(Z,z)$ and $(\tilde Z,\tilde z)$, respectively. We denote the convergence  $d_{pmG}(f_i,f)\stackrel{i\to\infty}{\longrightarrow}0$ by $f_i\stackrel{pmG}{\longrightarrow} f$ and say that $f_i$ pmG-converges to $f$.

\begin{remark}
	There is a natural analogue of $d_{pmG}$ for the pointed measured Gromov--Hausdorff distance of mapping packages, obtained by replacing the Prokhorov distance $F$ in \eqref{eq:pmg-dist} by the sum $F+H$ of $F$ and the Hausdorff distance $H$, cf. \cite[Definition 2.19]{Bate22}. We do not pursue this notion here, although we briefly consider pmGH-convergence of mapping packages at the end of Section \ref{sec:tangents} by simpler means in the special case of Lipschitz maps into a fixed proper space, see Definition \ref{def:equi-lip-maps-conv}.
\end{remark}

Observe that $(\iota,\tilde\iota\circ f)_\ast\mu=(\iota,\tilde\iota)_\ast\mu_f$, where $\mu_f=(\operatorname{id},f)_\ast\mu$ is the measure in the graph $G(f)$ of $f$. Thus, the infimum above is taken over a subclass of all isometric embeddings of $G(f)$. This immediately yields the estimate 
\begin{align}\label{eq:pmg-vs-gr}
	d_{pmG}(G(f_0),G(f_1))\le d_{pmG}(f_0,f_1).
\end{align}
However, it can happen that $d_{pmG}(X_0,X_1)+d_{pmG}(G(f_0),G(f_1))=0$ and $d_{pmG}(f_0,f_1)>0$, as the following remark shows.

\begin{remark}[The isometry class of $G(f)$ does not determine $f$]\label{rmk:graph}
	The following example shows the existence of functions $f,g:X\to\R$ such that $G(f)$ and $G(g)$ are isometric but there are no isometries $h:X\to X$, $A:\operatorname{Im}(f)\to \operatorname{Im}(g)$ for which $A\circ f=g\circ h$.
	
	Let $X=\{(0,0),(1,0),(0,1),(1/\sqrt 2,1/\sqrt 2)\}\subset (\R^2,\|\cdot\|_\infty)$. Define $f,g:X\to\R$ by
	\begin{align*}
		f(0,1)&=f(1/\sqrt 2,1/\sqrt 2)=f(1,0)=0,\quad f(0,0)=1\\
		g(0,1)&=g(1,0)=0,\quad g(0,0)=1/2,\quad f(1/\sqrt 2,1/\sqrt 2)=\sqrt{7/8}.
	\end{align*}
	Consider the map $\tilde h:X\to X$ which fixes $(0,1)$ and $(1,0)$ and swaps $(0,0)$ and $(1/\sqrt 2,1/\sqrt 2)$. Then $\tilde h$ is bi-Lipschitz but not isometric, however $\bar h:G(f)\to G(g)$, $(x,f(x))\mapsto (\tilde h(x),g(\tilde h(x)))$ is an isometry.
	
	In particular choosing as basepoint either $(0,1)$ or $(1,0)$ and equipping $X$ with the measure $\mu=\sum_{x\in X}\delta_{x}$, we have that $d_{pmG}(G(f),G(g))=0$ but ${d_{pmG}}(f,g)>0$ (see Lemma \ref{lem:pmG-dist-zero}).
\end{remark}

\subsection{Pointed measured Gromov convergence of mapping packages} We begin with the following elementary estimate which will prove useful later on. 

\begin{lemma}\label{lem:dist-of-mp-change-meas}
Let $(X,\bar x)$ be a complete separable pointed metric space. Suppose $\mu,\nu\in \mathcal M_{loc}(X)$ satisfy $\nu\le\mu$, and $f:(X,\mu,\bar x)\to (Y,\bar y)$ is an essentially separably valued $\mu$-measurable map. Denote by $f^\nu$ the mapping package $f:(X,\nu,\bar x)\to (Y,\bar y)$. Then 
\begin{align*}
d_{pmG}(f,f^\nu)\le \inf\{\varepsilon\in (0,1/2): (\mu-\nu)(B(\bar x,\varepsilon\inv))<\varepsilon \}
\end{align*}
with the convention $\inf\varnothing=\frac 12$.
\end{lemma}
\begin{proof}
We have $d_{pmG}(f,f^\nu)\le F_{(X\times Y,(\bar x,\bar y))}(\mu_f,\nu_f)$. For any $1/\varepsilon$-Lipschitz function $g:X\times Y\to [-1,1]$ supported in $B((\bar x,\bar y),1/\varepsilon)$, the function $g\circ \bar f=g\circ(\operatorname{id},f)$ is supported in $\bar f\inv B((\bar x,\bar y),1/\varepsilon)\subset B(\bar x,1/\varepsilon)$. Thus 
\begin{align*}
\Big|\int g\ud\mu_f-\int g\ud\nu_f\Big|=\Big|\int g(x,f(x))\ud\mu-\int g(x,f(x))\ud\nu\Big|\le \|g\|_\infty(\mu-\nu)(B(\bar x,1/\varepsilon)).
\end{align*} 
Supremizing over $g$ yields $F^{1/\varepsilon,1/\varepsilon}_{(X\times Y,(\bar x,\bar y))}(\mu_f,\nu_f)\le (\mu-\nu)(B(\bar x,1/\varepsilon))$, from which the claim readily follows.
\end{proof}

The following lemma gives a more concrete description of pointed measured Gromov convergence of mapping packages. 

\begin{lemma}\label{lem:map-package-conv}
Let $f_i\in \M$, $i\in\N\cup\{\infty\}$. Then the following are equivalent.
\begin{itemize}
    \item[(1)] ${d_{pmG}}(f_i,f_\infty)\stackrel{i\to\infty}{\longrightarrow} 0$;
    \item[(2)] There are pointed isometric embeddings $\iota_i:X_i\to (Z,z)$ and $\tilde\iota_i:Y_i\to (\widetilde Z,\tilde z)$ into complete separable (resp. complete) pointed metric spaces so that $(\iota_i,\tilde\iota_i\circ f_i)_\ast\mu_i\rightharpoonup (\iota_\infty,\tilde\iota_\infty\circ f_\infty)_\ast\mu_\infty$ in $\mathcal M_{loc}(Z\times\widetilde Z)$;
    \item[(3)] There are isometric embeddings $\iota_i:X_i\to (Z,z)$ and $\tilde\iota_i:Y_i\to (\widetilde Z,\tilde z)$ into complete separable (resp. complete) pointed metric spaces so that 
    \begin{itemize}
        \item[(i)] $\displaystyle\sup_i\mu_i(B(\bar x_i,R_0)\setminus f_i\inv B(f_i(\bar x_i),R))\stackrel{R\to \infty}{\longrightarrow}0,\quad R_0>0$;
        \item[(ii)] $\iota_{i\ast}\mu_i\rightharpoonup\iota_{\infty\ast}\mu_\infty$ in $\mathcal M_{loc}(Z)$, and 
        \item[(iii)] $\displaystyle (\tilde\iota_i\circ f_i)_\ast(\varphi\circ\iota_i\,\mu_i)\rightharpoonup (\tilde\iota_\infty\circ f_\infty)_\ast(\varphi\circ\iota_\infty\,\mu_\infty)$ in  $\mathcal M_{loc}(\widetilde Z)$ for any $\varphi\in C_{bbs}(Z)$.
    \end{itemize}
\end{itemize}
\end{lemma}
\begin{proof}
For each $i,j\in\N\cup\{\infty\}$ with $i\le j$ let $\hat\iota_i:X_i\to (Z_{i,j},z_{i,j})$, $\hat\iota_j:X_j\to (Z_{i,j},z_{i,j})$ and $\bar\iota_i:Y_i\to (\widetilde Z_{i,j},\tilde z_{i,j})$, $\bar\iota_j:Y_j\to (\widetilde Z_{i,j},\tilde z_{i,j})$ be isometric embeddings so that
\begin{align*}
F_{i,j}((\hat\iota_i,\bar\iota_i\circ f_i)_\ast \mu_i,(\hat\iota_j,\bar\iota_j\circ f_j)_\ast\mu_j)\le {d_{pmG}}(f_i,f_j)+2^{-ij}.
\end{align*}
Let $(Z,z)$ (resp. $(\widetilde Z,\tilde z)$) be the completion of the space obtained from $\hat\iota_i,\hat\iota_j,Z_{i,j}$ (resp. $\bar\iota_i,\bar\iota_j,\widetilde Z_{i,j}$) together with isometric embeddings $\iota_i:X_i\to (Z,z)$, $\iota_j:X_j\to (Z,z)$ and $\tilde\iota_i:Y_i\to (\widetilde Z,\tilde z)$, $\tilde\iota_j:X_j\to (\widetilde Z,\tilde z)$ given by \cite[Lemma 2.17]{Bate22}. Then 
\begin{align*}
F_{(z,\tilde z)}((\iota_i,\tilde\iota_i\circ f_i)_\ast\mu_i,(\iota_j,\tilde\iota_j\circ f_j)_\ast\mu_j)\le F_{i,j}((\hat\iota_i,\bar\iota_i\circ f_i)_\ast \mu_i,(\hat\iota_j,\bar\iota_j\circ f_j)_\ast\mu_j)
\end{align*}
for all $i,j\in\N\cup\{\infty\}$ with $i\le j$. Assuming (1), i.e. ${d_{pmG}}(f_i,f_\infty)\to 0$, this implies (2) by \cite[Proposition 2.13]{Bate22}.

Now assume the existence of isometric embeddings $\iota_i:X_i\to (Z,z)$ and $\tilde\iota_i:Y_i\to (\widetilde Z,\tilde z)$ as in (2). For any $g\in C_{bbs}(Z)$ and $\tilde g\in C_{bbs}(\widetilde Z)$ we have $g\otimes\tilde g\in C_{bbs}(Z\times\widetilde Z)$ and consequently
\begin{align*}
    \lim_{i\to\infty} \int g\circ\iota_i(x)\tilde g(\tilde\iota_i\circ f_i(x))\ud\mu_i(x)=\int g\circ\iota_\infty(x)\tilde g(\tilde\iota_\infty\circ f_\infty(x))\ud\mu_\infty(x).
\end{align*}
In particular, for any $g\in C_{bbs}(Z)$ we have 
\begin{align}\label{eq:push-conv}
(\tilde\iota_i\circ f_{i})_\ast((g\circ \iota_i)\mu_i)\rightharpoonup (\tilde\iota_\infty\circ f_{\infty})_\ast((g\circ \iota_\infty)\mu_\infty)\quad\mathrm{in}\quad \mathcal M_{loc}(\widetilde Z),
\end{align}
establishing (iii). Considering the functions $ g_{R_0}=(1-\operatorname{dist}(B_{ Z}( z,R_0),\cdot))_+$ 
, the weak convergence above implies tightness and thus
\begin{align*}
    \sup_i\mu_i(B(\bar x_i,R_0)\setminus f_i\inv B(f_i(\bar x_i),R))\stackrel{R\to \infty}{\longrightarrow}0,\quad R_0>0.
\end{align*}
Moreover, for any $g\in C_{bbs}(Z)$, let $R_0$ be such that $\operatorname{spt}(g\circ\iota_i)=\iota_i\inv(\operatorname{spt}g)\subset B(\bar x_i,R_0)$. Note also that $\operatorname{spt}(\tilde g_R\circ\tilde\iota_i\circ f_i)\subset f_i\inv(\tilde\iota_i\inv(\operatorname{spt}\tilde g_R))\subset f_i\inv B(f_i(\bar x_i),R+1)$, where $\tilde g_R=(1-{\rm dist}(B_{\widetilde Z}(\tilde z,R),\cdot))_+\in C_{bbs}(\widetilde Z)$. Then for large $R>0$ we have
\begin{align*}
\left|\int g\ud\iota_{i\ast}\mu_i-\int g\ud\iota_{\infty\ast}\mu_\infty\right|\le &\left|\int g\circ\iota_i(x)\tilde g_{R}(\tilde\iota_i\circ f_i(x))\ud\mu_i(x)-\int g\circ\iota_\infty(x)\tilde g_{R}(\tilde\iota_\infty\circ f_\infty(x))\ud\mu_\infty(x)\right|\\
 & + 2\sup_{i\in\N\cup\{\infty\}}\mu_i(B(\bar x_i,R_0+1)\setminus f_i\inv B(f_i(\bar x_i),R)),
\end{align*}
which implies $\iota_{i\ast}\mu_i\rightharpoonup \iota_{\infty\ast}\mu_\infty$. This completes the proof that (2) implies (3).

It remains to prove that (3) implies (1). This is essentially an approximation argument, which we sketch. From (iii) it follows that 
\begin{align*}
    \lim_{i\to\infty}\int g\circ\iota_i\, \tilde g\circ\tilde\iota_i\circ f_i\ud\mu_i=\int g\circ\iota_\infty\, \tilde g\circ\tilde\iota_\infty\circ f_\infty\ud\mu_\infty
\end{align*}
for all $g\in C_{bbs}(Z)$ and $\tilde g\in C_{bbs}(\widetilde Z)$. Consequently 
\begin{align}\label{eq:stone-weierstrass1}
    \lim_{i\to\infty}\int\varphi(\iota_i(x),\tilde\iota_i\circ f_i(x))\ud\mu_i=\int\varphi(\iota_\infty(x),\tilde\iota_\infty\circ f_\infty(x))\ud\mu_\infty
\end{align}
for all $\varphi\in C_{bbs}(Z)\otimes C_{bbs}(\widetilde Z)$. Fix $\varepsilon>0$. By (ii) and (iii) the families 
\[
\{\iota_{i\ast}\mu_i\}_{i\in\N\cup\{\infty\}}\quad\mathrm{and}\quad \{(\tilde\iota_i\circ f_i)_\ast(\tilde g_{1/\varepsilon}\circ\iota_i\,\mu_i)\}_{i\in\N\cup\{\infty\}}
\]
are tight. Thus, there exist compact sets $K\subset Z$ and $\widetilde K\subset \widetilde Z$ such that 
\begin{align}\label{eq:stone-weierstrass2}
\sup_{i\in\N\cup\{\infty\}}\mu_i(B(\bar x_i,1/\varepsilon)\setminus (\tilde\iota_i\circ f_i)\inv (\widetilde K))+\sup_{i\in\N\cup\{\infty\}}\iota_{i\ast}\mu_i(B(z,1/\varepsilon)\setminus K)<\varepsilon
\end{align}
Now $C(K)\otimes C(\widetilde K)$ is dense in $C(K\times \widetilde K)$ by the Stone--Weierstrass theorem. Together \eqref{eq:stone-weierstrass1}, \eqref{eq:stone-weierstrass2} and an approximation argument, we obtain \eqref{eq:stone-weierstrass1} for all $\varphi\in C_{bbs}(Z\times \widetilde Z)$.
\end{proof}

Next, we establish separability and completeness of the pseudometric ${d_{pmG}}$, as well as a description of when ${d_{pmG}}(f_0,f_1)=0$. In the statement, we use the following notation. Given a mapping package $f:(X,\mu,x)\to Y$ we denote by $\operatorname{essIm}(f)$ the collection $y\in Y$ for which there exists some $\tilde x\in X$ so that $\mu(B(\tilde x,\varepsilon)\cap f\inv B(y,\varepsilon))>0$ for all $\varepsilon>0$. We remark that if $(\tilde x,y)\in G(f)$ then $y\in \operatorname{essIm}(f)$, and $y\in \operatorname{essIm}(f)$ if and only if $(\tilde x,y)\in G(f)$ for some $\tilde x\in X$. As a result, $\operatorname{essIm}(f)$ is separable (assuming $f$ is essentially separably valued). In fact ${\rm essIm}(f)=p_Y(G(f))$.

\begin{lemma}\label{lem:pmG-dist-zero}
We have that ${d_{pmG}}(f_0,f_1)=0$ if and only if there exists an isometry $\iota:\{x_0\}\cup\operatorname{spt}\mu_0\to \{x_1\}\cup\operatorname{spt}\mu_1$ of pointed metric measure spaces and an isometry $\tilde\iota:\{f_0(x_0)\}\cup\overline{\operatorname{essIm}(f_0)}\to \{f_1(x_1)\}\cup\overline{\operatorname{essIm}(f_1)}$ with $\tilde\iota(f_0(x_0))=f_1(x_1)$ so that
\begin{align*}
    \tilde\iota\circ f_0=f_1\circ\iota\quad\mu_0\textnormal{-a.e.}
\end{align*}
Moreover $(\M_{pmG},{d_{pmG}})$ is a complete separable metric space.
\end{lemma}
\begin{proof}
The proof of the triangle inequality and completeness of ${d_{pmG}}$ are similar to the proof in \cite[Proposition 2.20]{Bate22} using Lemma \ref{lem:map-package-conv}. To see separability, note that we may always replace the target of a mapping package $f:(X,\mu,x)\to Y$ by $\operatorname{essIm}(f)$ since the resulting mapping package has zero ${d_{pmG}}$-distance from $f$. Thus we may assume $f$ has separable image. Since every separable space admits an isometric embedding into $C([0,1])$, it is not difficult to see that the collection of maps $f:(X,\mu,x)\to C([0,1])$ where $X$ is finite and $\mu$ and $d$ take only rational values is dense in ${d_{pmG}}$, cf. the proof of \cite[Corollary 4.13]{Bate22}.

Suppose $f_0,f_1\in \M$ satisfy ${d_{pmG}}(f_0,f_1)=0$. Given isometric embeddings as in the definition of ${d_{pmG}}$, we argue as in the proof of Lemma \ref{lem:map-package-conv}: let $\varepsilon<\delta$ be given, and let $g:Z\to [-1,1]$ be $1/\delta$-Lipschitz with $\operatorname{spt}g\subset B(z,1/\delta)$ and $\tilde g_R=(1-\operatorname{dist}(B(\tilde z,R),\cdot))_+\in C_{bbs}(\widetilde Z)$ where $R=\frac 1\varepsilon$. We have
\begin{align*}
\left|\int g\ud\iota_{0\ast}\mu_0-\int g\ud\iota_{1\ast}\mu_1\right|\le &\left|\int g\circ\iota_0(x)\tilde g_{R}(\tilde\iota_0\circ f_0(x))\ud\mu_0(x)-\int g\circ\iota_1(x)\tilde g_{R}(\tilde\iota_1\circ f_1(x))\ud\mu_1(x)\right|\\
& + \mu_0(B(x_0,1/\delta+1)\setminus f_0\inv B(f_0(x_0),1/\varepsilon))\\
&+\mu_1(B(x_1,1/\delta+1)\setminus f_1\inv B(f_1(x_1),1/\varepsilon))\\
\le & F_{(Z\times \tilde Z,(z,\tilde z))}^{2/\varepsilon,2/\varepsilon}((\iota_0,\tilde\iota_0\circ f_0)_\ast\mu_0,(\iota_1,\tilde\iota_1\circ f_1)_\ast\mu_1)\\
&+\mu_0(B(x_0,1/\delta+1)\setminus f_0\inv B(f_0(x_0),1/\varepsilon))\\
&+\mu_1(B(x_1,1/\delta+1)\setminus f_1\inv B(f_1(x_1),1/\varepsilon)).
\end{align*}
Since ${d_{pmG}}(f_0,f_1)=0$, the first term on the RHS above can be made smaller than  $\varepsilon/2$ upon infimizing over embeddings. By the tightness of the measures $f_{i\ast}(\mu_i|_{B(x_i,R)})$ for $i=0,1$ and arbitrary $R>0$, the remaining two terms on the RHS tend to zero as $\varepsilon\to 0$. Thus for some $\iota_0,\iota_1$ and small enough $\varepsilon>0$ we have $F_{(Z,z)}^{1/\delta,1/\delta}(\iota_{0\ast}\mu_0,\iota_{1\ast}\mu_1)<\delta$. This implies $d_{pmG}((X_0,\mu_0,x_0),(X_1,\mu_1,x_1))=0$. By \cite[Theorem 4.11 and Lemma 4.12]{Bate22} this implies the existence of the required isometry $\iota:\{x_0\}\cup\operatorname{spt}\mu_0\to \{x_1\}\cup\operatorname{spt}\mu_1$. Denote shortly $(f_{0\ast}(\varphi\circ\iota\,\mu_0),\{f_0(x_0)\})=(\{f_0(x_0)\}\cup{\rm spt}f_{0\ast}(\varphi\circ\iota\,\mu_0),f_{0\ast}(\varphi\circ\iota\,\mu_0),\{f_0(x_0)\})$ and $(f_{1\ast}(\varphi\mu_1),\{f_1(x_1)\})=(\{f_1(x_1)\}\cup {\rm spt}f_{1\ast}(\varphi\mu_1),f_{1\ast}(\varphi\mu_1),\{f_1(x_1)\})$.
Observe that 
\begin{align*}
    d_{pmG}\big((f_{0\ast}(\varphi\circ\iota\,\mu_0),\{f_0(x_0)\}),(f_{1\ast}(\varphi\mu_1),\{f_1(x_1)\})\big)=0,
\end{align*}for any compactly supported Lipschitz function $\varphi:X_1\to \R$ by a similar argument as above. 
Indeed, the infimum in ${d_{pmG}}(f_0,f_1\circ\iota)={d_{pmG}}(f_0,f_1)=0$ can be taken over $(Z,z)=(X_1,x_1)$ and embeddings $\iota:X_0\to X_1$, $\operatorname{id}: X_1\hookrightarrow X_1$. Note that, since $\mu_1=\iota_\ast\mu_0$ we have $f_{1\ast}(\varphi\mu_1)=(f_1\circ\iota)_\ast(\varphi\circ\iota\,\mu_0)$. Thus
\begin{align}\label{eq:push-equal}
    d_{pmG}\big((f_{0\ast}(\varphi\,\mu_0),\{f_0(x_0)\}),((f_{1}\circ\iota)_\ast(\varphi\,\mu_0),\{f_1(x_1)\})\big)=0,
\end{align}
for any compactly supported Lipschitz function $\varphi:X_0\to \R$.
Note that for $\mu_0$-a.e. $x\in X_0$ we have $(x,f_0(x))\in G(f_0)$ and $(\iota(x),f_1(\iota(x)))\in G(f_1)$. For such $x$ we have $f_0(x)\in \operatorname{essIm}(f_0)$ and $f_1(\iota(x))\in \operatorname{essIm}(f_1)$. We claim that the map 
\[
\bar\iota:G(f_0)\to G(f_1),\quad (x,f_0(x))\mapsto (\iota(x),f_1(\iota(x)))
\]
is an isometry (note that $\bar\iota$ is almost everywhere defined but can be extended to all of $G(f_0)$ after we show the claim).

Let $C_j\subset X_0$ be increasing closed sets with $\mu_0(X_0\setminus C_j)<2^{-j}$ and $f_0|_{C_j}, (f_1\circ\iota)|_{C_j}$ are continuous. Take $x,y\in C_j\cap \operatorname{spt}(\mu_0|_{C_j})=\operatorname{spt}(\mu_0|_{C_j})$ and consider $\varphi=(1-\varepsilon\inv d(x,\cdot))_++(1-\varepsilon\inv d(y,\cdot))_+$. By \eqref{eq:push-equal} there exists an isometry
\begin{align*}
\iota_\varepsilon:\{f_0(x_0)\}\cup\operatorname{spt}f_{0\ast}(\varphi\,\mu_0)\to \{f_1(x_1)\}\cup\operatorname{spt}(f_1\circ\iota)_\ast(\varphi\,\mu_0).
\end{align*}
For any $\delta>0$ we may choose $\varepsilon>0$ small enough so that $f_0(B(x,\varepsilon)\cap C_j)\subset B(f_0(x),\delta)$ and $(f_1\circ\iota)(B(x,\varepsilon)\cap C_j)\subset B(f_1(\iota(x)),\delta)$ and similarly for $y$. This implies $f_0(x),f_0(y)\in {\rm spt}f_{0\ast}(\varphi\mu_0)$ and $f_1(\iota(x)),f_1(\iota(y))\in{\rm spt}(f_1\circ \iota)_\ast(\varphi\mu_0)$. It follows that $\iota_\varepsilon(f_0(x))\in B(f_1(\iota(x)),\delta)$, $\iota_\varepsilon(f_0(y))\in B(f_1(\iota(y)),\delta)$ and 
\begin{align*}
    d(\iota_\varepsilon(f_0(x)),\iota_\varepsilon(f_0(y)))=d(f_0(x),f_0(y)).
\end{align*}

Consequently,
\begin{align*}
    d(f_0(x),f_0(y))-2\delta\le d(f_1(\iota(x)),f_1(\iota(y)))\le d(f_0(x),f_0(y))+2\delta.
\end{align*}
Since $\delta>0$ is arbitrary, we obtain the equality $d(f_0(x),f_0(y))=d(f_1(\iota(x)),f_1(\iota(y)))$. This yields 
\begin{align}\label{eq:graph-isom}
d^2(\bar\iota(x,f_0(x)),\bar\iota(y,f_0(y)))=&d^2(\iota(x),\iota(y))+d^2(f_1(\iota(x)),f_1(\iota(y)))\nonumber\\
=&d^2(x,y)+d^2(f_0(x),f_0(y)),
\end{align}
proving that $\bar\iota$ is an isometry. Now we may define
\begin{align*}
\tilde\iota:\{f_0(x_0)\}\cup\operatorname{essIm}(f_0)\to \{f_1(x_1)\}\cup\operatorname{essIm}(f_1),\quad y\mapsto p_{Y_1}(\bar\iota(p_{Y_0}\inv(y)))
\end{align*}
 This is a well-defined isometry by \eqref{eq:graph-isom}. Observe that $\tilde\iota (f_0(x_0))=f_1(x_1)$ and $\tilde\iota (f_0(x))=f_1(\iota(x))$ for $\mu_0$-a.e. $x\in X_0$ by the construction of $\bar\iota$. Thus the identity $\tilde\iota\circ f_0=f_1\circ\iota$ $\mu_0$-a.e. follows. 
\end{proof}

\subsubsection*{Lipschitz mapping packages and compactness} For equi-Lipschitz sequences we have the following simpler characterization of convergence with respect to $d_{pmG}$.

\begin{corollary}\label{cor:lip-map-conv}
	Let $(f_i)\subset \M$ be a sequence of $L$-Lipschitz mapping packages. Suppose $f_\infty\in \M$. Then ${d_{pmG}}(f_i,f_\infty)\stackrel{i\to\infty}{\longrightarrow}0$ if and only if there are pointed isometric embeddings $\iota_i:X_i\to (Z,z)$, $\tilde\iota_i:Y_i\to (\widetilde Z,\tilde z)$ into a complete separable metric space $(Z,z)$ and a complete metric space $(\tilde Z,\tilde z)$, respectively, so that
	\begin{itemize}
		\item[(i)] $\iota_{i\ast}\mu_i\rightharpoonup\iota_{\infty\ast}\mu_\infty$ in $\mathcal M_{loc}(Z)$ and
		\item[(ii)] $\tilde\iota_i(f_i(x_i))\to \tilde\iota_\infty(f_\infty(x))$ whenever $x\in \operatorname{spt}\mu_\infty$ and $x_i\in X_i$ satisfy $\iota_i(x_i)\to \iota_\infty(x)$.
	\end{itemize}
	In particular $f_\infty$ is $L$-Lipschitz if ${d_{pmG}}(f_i,f_\infty)\stackrel{i\to\infty}{\longrightarrow} 0$.
\end{corollary}
\begin{proof}
	Assume ${d_{pmG}}(f_i,f_\infty)\to 0$ as $i\to\infty$, and let $\iota_i,\tilde\iota_i$ be as in Lemma \ref{lem:map-package-conv}(2). In particular (i) holds. Let $x\in \operatorname{spt}\mu_\infty$ and $x_i\in \operatorname{spt}\mu_i$ be such that $\iota_i(x_i)\to \iota_\infty(x)$. Fix a continuous function $\varphi:[0,\infty)\to [0,\infty)$ supported on $[0,1]$. The functions
	\begin{align*}
		g_i(z)=\varphi\left(\frac{d(\iota_i(x_i),z)}{r}\right)
	\end{align*}
	satisfy $g_i\to g_\infty$ uniformly on bounded sets. This and Lemma \ref{lem:map-package-conv}(3iii) imply $\tilde\iota_i\circ f_i(g_i\iota_{i\ast}\mu_i)\rightharpoonup\tilde\iota_\infty\circ f_\infty(g_\infty\iota_{\infty\ast}\mu_\infty)$. Since 
	\begin{align*}
		\tilde\iota_i\circ f_i(g_i\iota_{i\ast}\mu_i)=(\tilde\iota_i\circ f_i)_\ast(\varphi(d(x_i,\cdot)/r)\mu_i)=:\nu_i,\quad i\in\N\cup\{\infty\},
	\end{align*}
	we have that $\operatorname{spt}\nu_i\subset \tilde\iota_iB(f_i(x_i),Lr)$, $i\in\N\cup\{\infty\}$. This and the weak convergence above imply (ii) in the claim. 
	
	The converse implication follows since (i) and (ii) together with the fact that $f_i$ is $L$-Lipschitz implies Lemma \ref{lem:map-package-conv}(3iii). The claim about $f_\infty$ being $L$-Lipschitz follows directly from (ii).
\end{proof}
	
\begin{corollary}\label{cor:map-precpt}
	Suppose $f_i:X_i:\to (Y,y)$ are $L$-Lipschitz mapping packages into a fixed proper pointed space $(Y,y)$, $i\in\N\cup\{\infty\}$. Then the following are equivalent.
	\begin{itemize}
		\item[(i)] $d_{pmG}(f_i,f_\infty)\stackrel{i\to\infty}{\longrightarrow}0$;
		\item[(ii)] for any subsequence of $(f_i)$ there exists  $\tilde f_\infty:(\{\bar x_\infty\}\cup\operatorname{spt}\mu_\infty,\mu_\infty,\bar x_\infty) \to (Y,y)$ with $d_{pmG}(f_\infty,\tilde f_\infty)=0$ so that, up to a further subsequence (labeled by the same indices), we have $f_i\stackrel{pmG}{\longrightarrow}\tilde f_\infty$ in the sense of Definition \ref{def:equi-lip-maps-conv};
	\end{itemize}
\end{corollary}
\begin{remark}
	The convergence $d_{pmG}(f_i,f_\infty)\stackrel{i\to\infty}{\longrightarrow}0$ in Corollary \ref{cor:map-precpt} does not imply convergence in the sense of Definition \ref{def:equi-lip-maps-conv} for the full sequence because of the ambiguity of mapping packages having zero pmG-distance; cf. Lemma \ref{lem:pmG-dist-zero}. Corollary \ref{cor:map-precpt} essentially states that pmG-convergence in the sense of Definition \ref{def:equi-lip-maps-conv} is equivalent to convergence with respect to $d_{pmG}$ up to this ambiguity.
\end{remark}
In the proof, we will repeatedly use the observation that if $f_i\stackrel{pmG}{\longrightarrow} f$ in the sense of Definition \ref{def:equi-lip-maps-conv} then $d_{pmG}(f_i,f)\stackrel{i\to\infty}{\longrightarrow}0$. This follows directly from Corollary \ref{cor:lip-map-conv}.

\begin{proof}
	Assume that $d_{pmG}(f_i, f_\infty)\stackrel{i\to\infty}{\longrightarrow}0$. Since then $X_i\stackrel{pmG}{\longrightarrow}X_\infty$ (by Corollary \ref{cor:lip-map-conv}), there exists a complete separable pointed space $(Z,\bar z)$ and isometric embeddings $\iota_i:(X_i,\bar x_i)\to (Z,\bar z)$ for which $F_{(Z,\bar z)}(\iota_{i\ast}\mu_i,\iota_{\infty\ast}\mu_\infty)\stackrel{i\to\infty}{\longrightarrow}0$.  Fix a subsequence of $(f_i)$ labeled by the same indices, and a countable dense subset $D=\{x^k:\ k\in\N\}\subset \{\bar x_\infty\}\cup\operatorname{spt}\mu_\infty$ where $x^0=\bar x_\infty$. For each $k\in\N$ choose $x_i^k\in X_i$ with $\iota_i(x_i^k)\to \iota_\infty(x^k)$; for $k=0$ choose $x_i^0=\bar x_i$. Since 
	\begin{align*}
		d_Y(f_i(x_i^k),y)\le Ld_{X_i}(x_i^k,\bar x_i)\stackrel{i\to \infty}{\longrightarrow} Ld_{X_\infty}(x^k,\bar x_\infty),
	\end{align*}
	the sequence $(f_i(x_i^k))_{i\in\N}$ is bounded in $Y$ for all $k\in\N$. Since $Y$ is proper we may pass to a further subsequence (labeled by the same indices) so that 
	\[
	f_i(x_i^k)\stackrel{i\to\infty}{\longrightarrow} c^k\in Y
	\]
	for all $k\in\N$. The map $D\to Y$ given by $x^k\mapsto c^k$ is well-defined and $L$-Lipschitz, since 
	\begin{align*}
		d_Y(c^k,c^m)=\lim_{i\to\infty} d_Y(f_i(x_i^k),f_i(x_i^m))\le L\lim_{i\to\infty} d_{X_i}(x_i^k,x_i^m)=Ld_{X_\infty}(x^k,x^m).
	\end{align*} 
	Thus it extends to an $L$-Lipschitz map $\tilde f_\infty:\{\bar x_\infty\}\cup\operatorname{spt}\mu_\infty\to Y$. Note that $\tilde f_\infty(\bar x_\infty)=y$ and $\tilde f_\infty(x^k)=c^k$ by construction. If $x\in \{\bar x_\infty\}\cup\operatorname{spt}\mu_\infty$ and $x_i\in X_i$ are such that $\iota_i(x_i)\to \iota_\infty(x)$ then 
	\begin{align*}
		d_Y(f_i(x_i),\tilde f_\infty(x))& \le d_Y(f_i(x_i),f_i(x_i^k))+d_Y(f_i(x_i^k),c^k)+d_Y(c^k,\tilde f_\infty(x))\\
		&\le Ld_{X_i}(x_i,x_i^k)+d_Y(f_i(x_i^k),c^k)+Ld_{X_\infty}(x^k,x)
	\end{align*}
	for all $k\in\N$. Thus
	\begin{align*}
		\limsup_{i\to\infty}d_Y(f_i(x_i),\tilde f_\infty(x))\le 2Ld_{X_\infty}(x,x^k)
	\end{align*}
	and, since $D=\{x^k\}$ is dense, taking infimum over $k$ yields $\lim_{i\to\infty}d_Y(f_i(x_i),\tilde f_\infty(x))=0$. This proves that $f_i\stackrel{pmG}{\longrightarrow} \tilde f_\infty$ in the sense of Definition \ref{def:equi-lip-maps-conv} for the chosen subsequence. It follows from Corollary \ref{cor:lip-map-conv} that $d_{pmG}(f_i,\tilde f_\infty)\stackrel{i\to\infty}{\longrightarrow} 0$ and thus $d_{pmG}(f_\infty,\tilde f_\infty)=0$. This implies (ii). 
	
	Conversely, assume that (ii) holds. If (i) is false, there exists $a>0$ and a subsequence (labeled by the same indices) so that
	\begin{align}\label{eq:a}
		d_{pmG}(f_i,f_\infty)\ge a>0
	\end{align} for all $i$. But by (ii) there exists $\tilde f_\infty$ such that $d_{pmG}(f_\infty,\tilde f_\infty)=0$ and $f_i\stackrel{pmG}{\longrightarrow} \tilde f_\infty$ in the sense of Definition \ref{def:equi-lip-maps-conv}. By Corollary \ref{cor:lip-map-conv} in particular $d_{pmG}(f_i,\tilde f_\infty)\stackrel{i\to\infty}{\longrightarrow}0$. Thus $d_{pmG}(f_i, f_\infty)\stackrel{i\to\infty}{\longrightarrow}0$, contradicting \eqref{eq:a}. This completes the proof.
\end{proof}

\subsection{Ultralimits and compactness}

\begin{lemma}
\label{lem:common-ambient-space}
Let $(Y_i,y_i)$, be complete pointed metric spaces, $i\in\N\cup\{\infty\}$. Denote $y_i^0\coloneqq y_i$, $i\in \N\cup\{\infty\}$. Suppose that there exist sequences $(y_i^k)_k\subset Y_i$, $i\in\N$ such that  
\begin{align*}
\lim_{i\to\infty}d_{Y_i}(y^k_i,y^m_i)=d_{Y_\infty}(y^k_\infty,y^m_\infty),\quad k,m\in\N.
\end{align*}
Then there exist subsequences  $(Y_i,y_i)$, with corresponding sequences $(y_i^k)_k$, $i\in\N$, a  complete pointed metric space $(W,w)$ and pointed isometric embeddings $\tilde\iota_i:(Y_i,y_i)\to (W,w)$ so that 
\begin{align*}
\lim_{i\to\infty}d_{W}(\tilde\iota_i(y^k_i),\tilde \iota_\infty(y_\infty^k))=0\quad\textrm{for\ all\ } k\in\N.
\end{align*}  
\end{lemma}

\begin{proof}
 For every $j\in\N$ there exist $N_j\in\N$ such that for all $i\geq N_j$ and all $0\leq k,m\leq j$ it holds that 
 \[
 |d_{Y_i}(y_i^k,y_i^m)-d_{Y_\infty}(y^k,y^m)|\leq 2^{-j}.
 \]
 By considering the subsequence $\tilde Y_j\coloneqq (Y_{N_j}, y_{N_j})$, and the corresponding  sequences $(\tilde y_j^k)_k\coloneqq (y_{N_j}^k)_k\subset \tilde Y_j$, $j\in\N$ we get a subsequence of spaces such that 
 \begin{equation}
 \label{eq:good_subsequence}
 |d_{\tilde Y_j}(\tilde y_j^k,\tilde y_j^m)-d_{Y^\infty}(y^k,y^m)|\leq 2^{-j}, \quad \text{for all }0\leq k,m\leq j, \,j\in\N.
 \end{equation}
 We consider such a subsequence of pointed metric spaces and denote it as the original one, with $(Y_i,y_i)$, together with the corresponding sequences $(y_i^k)_k\subset Y_i$, $y_i^0=y_i$, $i\in\N$. Now, we construct a single complete pointed metric space containing isometric copies of all the spaces $(Y_i,y_i)$, $i\in\N$. Let 
 \[
 \tilde W\coloneqq \bigsqcup_{i\in\N\cup\{\infty\}} Y_i 
 \]
 be a disjoint union of spaces $Y_i$, $i\in\N\cup\{\infty\}$. We are going to define a pseudometric on $\tilde W$. First, we define a set of allowed moves $\mathcal R\subset \tilde W\times \tilde W$ within the set $\tilde W$, as follows: 

\[
(u,v)\in\mathcal{R}
\iff
\begin{cases}
u,v\in Y_i
& \text{for some } i\in\mathbb{N}\cup\{\infty\},\\[2mm]
\{u,v\}=\{y_i^k,y_\infty^k\}
& \text{for some } i\in\mathbb{N},\ 0\leq k\leq i.
\end{cases}
\]
Further, we define a cost of every move within $\tilde W$, i.e. the function $d:\mathcal{R}\to [0,\infty)$, as follows:
\[
d(u,v)=
\begin{cases}
 d_{Y_i}(u,v), & u,v\in Y_i\text{ for some } i \in \N\cup\{\infty\},\\
 0, & \{u,v\}=\{y_\infty,y_i\} \text{ for some } i\in\N,\\
 2^{-i}, & \{u,v\}=\{y_i^k,y_\infty^k\}\text{ for some } i\in\N,\, 1\leq k\leq i.
\end{cases}
\]
For $u,v \in \tilde W$ a chain between $u$ and $v$ is an $n$-tuple $(x_0,...,x_n)\in\tilde W^{n+1}$ for some $n\in \N$, such that $x_0=u$, $x_n=v$ and $(x_{l-1},x_l)\in \mathcal R$ for every $l=1,...,n$. We define ${\rm ch}(u,v)$ as the set of all chains between $u, v\in\tilde W$. Notice that for every $u,v\in \tilde W$ the set of chains is non-empty, since we can always use basepoints if necessary. For a chain $\gamma=(x_0,...,x_{n_\gamma})\in {\rm ch}(u,v)$ we define its cost as $L(\gamma)\coloneqq\sum_{l=1}^{n_\gamma} d(x_{l-1},x_l)$. 
Finally, for $u,v\in \tilde W $ we define $D(u,v)$ to be the infimum of all costs of chains between $u$ and $v$, that is
\[
D(u,v)\coloneqq \inf_{\gamma\in {\rm ch}(u,v)} L(\gamma),\quad u,v\in\tilde W.
\]
Clearly, $D$ is non-negative, finite-valued, and, by reversing chains and using the symmetry of $d$, it is also symmetric. Moreover, $D(u,u)=0$ for all $u\in \tilde W$ by choosing the chain $(x_0,x_1)=(u,u)$ with cost $0$. To prove that $D$ satisfies the triangle inequality, let $u,v,z\in \tilde W$ and fix an arbitrary $\varepsilon>0$. There exist chains $\gamma_1=(u,x_1,...,x_{n-1},z)\in {\rm ch}(u,z)$ and $\gamma_2=(z,y_1,...,y_{m-1},v)\in {\rm ch}(z,v)$ such that $L(\gamma_1)< D(u,z)+\varepsilon/2$ and $L(\gamma_2)<D(z,v)+\varepsilon/2$. Then $\gamma=(u,x_1,...,x_{n-1},z,y_1,...,y_{m-1},v)\in {\rm ch}(u,v)$ and $D(u,v)\leq L(\gamma)=L(\gamma_1)+L(\gamma_2)\leq D(u,z)+D(z,v)+\varepsilon$. By arbitrariness of $\varepsilon>0$ the triangle inequality $D(u,v)\leq D(u,z)+D(z,v)$ follows. Therefore, $D$ is a pseudometric on $\tilde W$. Notice that $D$ is not a metric since $D(y_i,y_j)=0$ for every $i,j\in\N\cup\{\infty\}$. We claim that $D$ restricts to $d_{Y_i}$ on every $Y_i$, $i\in\N\cup\{\infty\}$. First, we notice that for every $i\in\N$ it holds that
\begin{equation}\label{eq:target-gluing-distortion}
	\bigl|d_{Y_i}(y_i^k,y_i^\ell)
	-d_{Y_\infty}(y^k_\infty,y^\ell_\infty)\bigr|
	\le c_i^k+c_i^\ell,
	\qquad 0\le k,\ell\le i,
\end{equation}
where $c_i^k=2^{-i}$, $1\leq k\leq i$ and $c_i^0=0$. Indeed, if at least one of $k,\ell$ is nonzero, the right-hand side is at least $2^{-i}$, and the estimate follows from the choice of the subsequence (see \eqref{eq:good_subsequence}). If $k=\ell=0$, both sides vanish. Let $u,v\in Y_\infty$ and $\gamma\in {\rm ch}(u,v)$. If the elements of the chain  $\gamma$ are all in $Y_\infty$, then by triangle inequality for $d_{Y_\infty}$ we have $L(\gamma)\geq d_{Y_\infty}(u,v)$. Otherwise, $\gamma$ contains finitely many subchains $\gamma_j$, $j=1,...,n$ from an element of $Y_\infty$ to an element of $Y_\infty$ through elements of $Y_{i_j}$, for some $i_j\in\N$, $j=1,...,n$. Every such transition through $Y_{i_j}$ starts with a move from $y^{k_j}_\infty$ to $y^{k_j}_{i_j}$, then continues through movements within $Y_{i_j}$ to $y_{i_j}^{\ell_j}$ and ends with a move from $y_{\ell_j}^{i_j}$ to $y^{\ell_j}_\infty$ for some $0\leq k_j,\ell_j\leq i_j$. A cost of a chain $\gamma_j$ satisfies $L(\gamma_j)\geq c_{i_j}^{k_j}+d_{Y_i}(y_{i_j}^{k_j},y_{i_j}^{\ell_j})+c_{i_j}^{\ell_j}\geq d_{Y_\infty}(y^{k_j}_\infty,y^{\ell_j}_\infty)$, where the last inequality is due to \eqref{eq:target-gluing-distortion}. Hence, we can change every subchain $\gamma_j$ in $\gamma$ with the direct move from $y_\infty^{k_j}$ to $y_\infty^{\ell_j}$ and obtain a new chain $\tilde \gamma\in {\rm ch}(u,v)$ with all elements in $Y_\infty$, such that $L(\gamma)\geq L(\tilde \gamma)\geq d_{Y_\infty}(u,v)$. Since $L(\gamma)\geq d_{Y_\infty}(u,v)$ for all $\gamma\in {\rm ch}(u,v)$, it follows  $D(u,v)\geq d_{Y_\infty}(u,v)$. Especially, if we consider the chain $\gamma=(u,v)$, we conclude $D(u,v)\leq d_{Y_\infty}(u,v)$, and thus $D(u,v)=d_{Y_\infty}(u,v)$ for all $u,v\in Y_\infty$. Now fix $i\in\N$ and $u,v\in Y_i$. Let $\gamma\in{\rm ch}(u,v)$. If all the elements of $\gamma$ are inside $Y_i$, then by triangle inequality for $d_{Y_i}$ we have $L(\gamma)\geq d_{Y_i}(u,v)$. Otherwise $\gamma$ is made of five subchains $\gamma_1=(u,...,y_{i}^{k})$,
$\gamma_2=(y_{i}^k,y_\infty^{k})$, $\gamma_3=(y_\infty^{k},...,y_\infty^{\ell})$, $\gamma_4=(y^{\ell}_\infty,...,y_i^{\ell})$ and $\gamma_5=(y_i^{\ell},...,v)$, for some $y_i^k,y_i^\ell\in Y_i$, $y_\infty^k,y_\infty^\ell\in Y_\infty$, $0\leq k,\ell\leq i$, such that $\gamma_1$ and $\gamma_5$ are fully contained in $Y_i$. The preceding arguments imply
\begin{align*}
L(\gamma)& =\sum_{j=1}^5L(\gamma_j)\geq d_{Y_i}(u, y_i^k)+c_{i}^k+d_{Y_\infty}(y_\infty^k,y_\infty^\ell)+c_i^\ell+d_{Y_i}(y_i^\ell,v)\\
& \underset{\eqref{eq:target-gluing-distortion}} {\geq} d_{Y_i}(u,y_i^k)+d_{Y_i}(y_i^k,y_i^\ell)+d_{Y_i}(y_i^\ell,v)\geq d_{Y_i}(u,v). 
\end{align*}
Hence, $D(u,v)\geq d_{Y_i}(u,v)$. To achieve $D(u,v)=d_{Y_i}(u,v)$, use the chain $\gamma=(u,v)$ with the single move from $u$ to $v$. Say that elements $u,v\in \tilde W$ are equivalent ($u\sim v$) if and only if $D(u,v)=0$. Then  $(W_0,D_0)$ is the metric space, where $W_0$ is the set of eqivalence classes $[u]$ of elements $u\in\tilde W$, and the metric $D_0$ is given with 
\[D_0([u],[v])=D(u,v), \quad u,v\in\tilde W.\]
Finally, we define the metric space $(W, d_W)$ to be the completion of $(W_0,D_0)$. Then the canonical projections sending element to its equivalence class $\tilde \iota_i: Y_i\to W_0\subset W$ are isometric embeddings of $Y_i$ to $W$, $i\in \N\cup\{\infty\}$. Moreover, $d_W(\tilde\iota_i(y_i),\tilde\iota_j(y_j))=D(y_i,y_j)=0$, for all $i,j\in\N\cup\{\infty\}$. Therefore, isometric embeddings $\tilde\iota_i$ are sending base points to the same point $ w \in W$. Furthermore, we have 
\[
d_W(\tilde \iota_i(y_i^k),\tilde\iota_{\infty}(y_\infty^k))=D(y_{i}^k,y_\infty^k)\leq d(y_i^k,y_\infty^k)\leq 2^{-i}, \quad \text{for all } 0\leq k\leq i, \,i\in\N.
\]
Consequently, for every $k\in\N$ it follows that $d_W(\tilde \iota_i(y_i^k),\tilde\iota_{\infty}(y_\infty^k))\to 0$ as $i\to\infty$. This concludes the proof of the lemma.
\end{proof}

For the next proposition, fix a non-principal ultrafilter $\omega$ on $\N$. 

\begin{proposition}\label{prop:ultralim-pmG-conv-lip}
Let $L>0$. Suppose that $(X_i)\subset (\XX,d_{pmG})$ is precompact with $\omega$-limit $X_\omega$ in $\XX$, and that $f_i:X_i\to (Y_i,y_i)$ are $L$-Lipschitz mapping packages for each $i\in\N$. If $f_\omega:X_\omega\to (Y_\omega,y_\omega)$ denotes the ultralimit of $(f_i)$ given by Theorem \ref{thm:constr_of_ultralim}, then 
\begin{align*}
\lim_\omega d_{pmG}(f_i,f_\omega)=0.
\end{align*}
\end{proposition}
	
\begin{proof}
Suppose the claim is not true. Then for some $\varepsilon>0$ we have that 
\[
A:=\{i:d_{pmG}(f_i,f_\omega)\ge \varepsilon\}\in \omega.
\]
Let $\iota_i:X_i\to (Z,z)$ be isometric embeddings such that $(\iota_{i\ast}\mu_i)_i\subset \mathcal M_{loc}(Z)$ is precompact and $\lim_\omega F_{(Z,z)}(\iota_{i_\ast}\mu_i,\iota_{\omega\ast}\mu_\omega)=0$; see Lemma \ref{lem:precpt}. Note that $f_i$ has Haj\l asz gradient $g_i\equiv L$ for each $i\in\N\cup\{\omega\}$, and choose a dense sequence $(x^k)_{k\in\N}\subset\{\bar x_\omega\}\cup\operatorname{spt}\mu_\omega$ with $x^0=\bar x_\omega$ so that \eqref{eq:ultralimit+graph} in the construction of the ultralimit $f_\omega:X_\omega\to Y_\omega$ holds for $x^k$ for all $k\in\N$. For each $k\in\N$, take sequences $x_i^k\in \{\bar x_i\}\cup\operatorname{spt}\mu_i$, with $x_i^0=\bar x_i$, such that $\lim_{\omega} d_Z(\iota_i(x_i^k),\iota_\omega(x^k))=0$, cf. Lemma \ref{lemma:support_point}.

Set $y_i^k:=f_i(x_i^k)$, $ y^k:=f_\omega(x^k)$. By \eqref{eq:ultralimit+graph} and the construction of the ultralimit $Y_\omega$ we have that
\begin{align}\label{eq:lip-bound}
d_{Y_\omega}(y^k,y^\ell)=\lim_\omega d_{Y_i}(y_i^k,y_i^\ell)=\lim_\omega d_{Y_i}(f_i(x_i^k),f_i(x_i^\ell))\le Ld_{X_\omega}(x^k,x^\ell),\quad k,\ell\in\N.
\end{align}
Take an arbitrary $i_0\in A$. For each $m\ge1$, choose $i_m\in A$, with $i_m\ge \max\{i_{m-1},m\}$,
such that
\begin{align*}
	F_{(Z,z)}
	(\iota_{i_m\ast}\mu_{i_m},\iota_{\omega\ast}\mu_\omega)
	<2^{-m},\\
	\max_{0\le k\le m}
	d_Z(\iota_{i_m}(x_{i_m}^k),\iota_\omega(x^k))
	<2^{-m},\\
	\max_{0\le k,\ell\le m}
	\bigl|d_{Y_{i_m}}(y_{i_m}^k,y_{i_m}^\ell)
	-d_{Y_\omega}(y^k,y^\ell)\bigr|
	<2^{-m}.
\end{align*}
Such a choice is possible because each of these finitely many conditions holds on a set belonging to $\omega$ and thus their intersection belongs to $\omega$. Now the subsequence $(i_m)$ satisfies
\begin{itemize}
\item[(1)] $F_{(Z,z)}(\iota_{i_m\ast}\mu_{i_m},\iota_{\omega\ast}\mu_\omega)\stackrel{m\to\infty}{\longrightarrow} 0$;
\item[(2)] $d_Z(\iota_{i_m}(x_{i_m}^k),\iota_{\omega}(x^k)))\stackrel{m\to\infty}{\longrightarrow}0$ for all $k\in\N$; and 
\item[(3)] $d_{Y_{i_m}}(y_{i_m}^k,y_{i_m}^\ell)\stackrel{m\to\infty}{\longrightarrow} d_{Y_\omega}(y^k,y^\ell)$ for all $k,\ell\in\N$.
\end{itemize}
By Lemma \ref{lem:common-ambient-space} there exist pointed isometric embeddings $\tilde\iota_{i_m}:(Y_{i_m},y_{i_m})\to (W,w)$ and $\tilde\iota_\omega:(Y_\omega,y_\omega)\to (W,w)$ into a common pointed complete space $(W,w)$ such that
\begin{align}\label{eq:isom-anchor-points}
\lim_{m\to\infty}d_{W}(\tilde\iota_{i_m}(y^k_{i_m}),\tilde\iota_{\omega}(y^k)))=0,\quad k\in\N.
\end{align}

By \eqref{eq:lip-bound} the map $\{x^k:k\in\N\}\to Y_\omega$ given by $x^k\mapsto y^k$ is well-defined and $L$-Lipschitz. Thus it has a unique $L$-Lipschitz extension denoted $\tilde f_\omega:\{\bar x_\omega\}\cup\operatorname{spt}\mu_\omega\to Y_\omega$. By construction we have 
\begin{align*}
\tilde f_\omega(x^k)=y^k=f_\omega(x^k)\quad\textrm{for\ all\ }k\in\N.
\end{align*}
Since $(x^k)_k$ is dense in $\{\bar x_\omega\}\cup\operatorname{spt}\mu_\omega$ it follows that $f_\omega=\tilde f_\omega$ $\mu_\omega$-a.e. on $\{\bar x_\omega\}\cup\operatorname{spt}\mu_\omega$. In particular $d_{pmG}(f_\omega,\tilde f_\omega)=0$.

For every $x\in \{\bar x_\omega\}\cup\operatorname{spt}\mu_\omega$ and $x_{i_m}\in X_{i_m}$ such that $\iota_{i_m}(x_{i_m})\to \iota_\omega(x)$ in $Z$ we have that 
\begin{align*}
	&d_{W}(\tilde\iota_{i_m}( f_{i_m}(x_{i_m})),\tilde\iota_\omega( \tilde f_\omega(x))) \\
	\le & d_W(\tilde\iota_{i_m}(f_{i_m}(x_{i_m})),\tilde\iota_{i_m}(f_{i_m}(x_{i_m}^k)))+d_W(\tilde\iota_{i_m}(f_{i_m}(x_{i_m}^k)),\tilde\iota_\omega(y^k))+d_W(\tilde\iota_\omega(y^k),\tilde\iota_\omega(\tilde f_\omega(x)))\\
	=&d_{Y_{i_m}}(f_{i_m}(x_{i_m}),f_{i_m}(x_{i_m}^k))+d_{W}(\tilde\iota_{i_m}(y_{i_m}^k),\tilde\iota_\omega(y^k))+ d_{Y_\omega}(\tilde f_\omega(x^k),\tilde f_\omega(x))\\
	\le &Ld_{X_{i_m}}(x_{i_m},x_{i_m}^k)+d_{W}(\tilde\iota_{i_m}(y_{i_m}^k),\tilde\iota_\omega(y^k))+Ld_{X_\omega}(x^k,x)
\end{align*}
for all $k\in\N$. By \eqref{eq:isom-anchor-points} the middle term on the RHS vanishes as $m\to\infty$, and thus 
\[
\limsup_{m\to\infty}d_{W}(\tilde\iota_{i_m}( f_{i_m}(x_{i_m})),\tilde\iota_\omega( \tilde f_\omega(x)))\le 2Ld_{X_\omega}(x^k,x).
\]
Since $D=\{x^k\}$ is dense, taking infimum over $k$ yields $\displaystyle \tilde\iota_{i_m}(f_{i_m}(x_{i_m}))\stackrel{m\to\infty}{\longrightarrow} \tilde\iota_{\omega}(\tilde f_{\omega}(x))$. Since $x$ was arbitrary, Corollary \ref{cor:lip-map-conv} implies that $d_{pmG}(f_{i_m},\tilde f_\omega)\stackrel{m\to\infty}{\longrightarrow}0$. This contradicts the assumption  $i_m\in A$, which requires that
\[
d_{pmG}(f_{i_m},\tilde f_\omega)=d_{pmG}(f_{i_m},f_\omega)\ge \varepsilon,
\]
and completes the proof.
\end{proof}

Proposition \ref{prop:ultralim-pmG-conv-lip} implies that pmG-limits of equi-Lipschitz sequences can be realized as ultralimits and moreover that equi-Lipschitz sequences enjoy compactness with respect to pmG-convergence. The lack of any compactness assumptions on the target means no Arzel\`a--Ascoli theorems are available; in light of this the latter fact is somewhat surprising.

In both corollaries below, $(f_i)\subset\M$ is a sequence of $L$-Lipschitz mapping packages.

\begin{corollary}
If $d_{pmG}(f_i,f_\infty)\stackrel{i\to \infty}{\longrightarrow} 0$ for some $f_\infty\in\M$, then $d_{pmG}(f_\infty,f_\omega)=0$. 
\end{corollary}
\begin{corollary}
If the domains $X_i$ of the maps $f_i$ form a precompact collection in $\XX$, then there exists a subsequence such that $d_{pmG}(f_{i_k},f_\omega)\stackrel{k\to\infty}{\longrightarrow}0$.
\end{corollary}

\begin{proof}[Proof of Theorem \ref{thm:ultralimit+graph}]
Let $f_i:(X_i,\mu_i,\bar x_i)\to (Y_i,y_i)$ and $g_i\in L^p_{loc}(\mu_i)$ be as in the claim. The construction of the ultralimit $f_\omega:X_\omega\to Y_\omega$ is given in Theorem \ref{thm:constr_of_ultralim}, where properties (i), (ii) and \eqref{eq:ultralimit+graph} are also proved. It thus remains to prove Theorem \ref{thm:ultralimit+graph}(iii). We do this using Proposition \ref{prop:ultralim-pmG-conv-lip}; indeed, Proposition \ref{prop:ultralim-pmG-conv-lip} implies (iii) for equi-Lipschitz sequences. Let $\iota_i:(X_i,\bar x_i)\to (Z,z)$ be pointed isometric embeddings given by Lemma \ref{lem:precpt}. We identify $(X_i,\mu_i,\bar x_i)$ and $(X_\omega,\mu_\omega,\bar x_\omega)$ with their isometric images in $Z$, so that $\lim_\omega  F_{(Z,z)}(\mu_\omega,\mu_i)=0$.

Fix a parameter $\rho>1$; for each $i,k\in\N$, define $E_i^k=\{g_i\le \rho^k\}$ and denote $\mu_i^k:=\mu_i|_{E_i^k}$, $X_i^k=(\{\bar x_i\}\cup\operatorname{spt}\mu_i^k,\mu_i^k,\bar x_i)$. Let $\mu_\omega^k$ be the $\omega$-limit of $(\mu_i^k)$ in $\mathcal M_{loc}(Z)$, and $X_\omega^k$ the corresponding pointed metric measure space.

The map $f_i^k:X_i^k\to (Y_i,y_i)$ obtained as the unique extension of $f_i|_{E_i^k}$ is $2\rho^k$-Lipschitz and $\bar x_i\in E_i^k$ for $k>\log t/\log \rho$. Let $f_\omega:X_\omega\to Y_\omega$ and $f_\omega^k:X_\omega^k\to Y_\omega$ denote the ultralimits of the sequences $(f_i)$ and $(f_i^k)$, respectively, for each $k\in\N$. We claim that
\begin{align}\label{eq:restr-property}
	f_\omega=f_\omega^k\quad\mu_\omega^k\textrm{-a.e.}
\end{align}
 Indeed, $\mu_\omega^k$ is concentrated on the $\omega$-limit $E_\omega^k\coloneqq \{\lim_\omega x_i\,:\,x_i\in E_i^k\}$ of the sets $E_i^k$. By Remark \ref{eq:f_omega-def-set} for $\mu_\omega$-a.e. $x\in E_\omega^k$ we have
\begin{align*}
f_\omega(x)=[f_i(x_i)]_\omega=f_\omega^k(x),
\end{align*}
since $x_i\in E_i^k$ for all $i\in\N$ and $x=\lim_\omega x_i$. This establishes \eqref{eq:restr-property}. Note as well that $f_i=f_i^k$ $\mu_i^k$-a.e. by the definition of $f_i^k$ and $\mu_i^k$. Thus we may estimate the distances $d_{pmG}(f_i,f_i^k)$ and $d_{pmG}(f_\omega,f_\omega^k)$ using Lemma \ref{lem:dist-of-mp-change-meas}. 

As in the proof of Theorem \ref{thm:constr_of_ultralim}, the uniform integrability condition \eqref{eq:energy-unif-bdd}, Lemmas \ref{lem:omega-limit-p-int-funct} and \ref{lem:omega-ineq}, and Chebyshev's inequality yield
\begin{align*}
(\mu_\omega-\mu_\omega^k)(B(\bar x_\omega,R))\le \rho^{-kp}\int_{B_Z(z,R)}g_i^p\ud\mu_\omega\le C(R)\rho^{-kp},\quad\textrm{for\ all\ }R>0,
\end{align*}
where
$C(R)=\sup_i\int_{B_Z(z,R)}g_i^p\ud\mu_i$.
Fix $\varepsilon>0$. Let $k_\varepsilon\in\N$ be large enough so that $C(1/\varepsilon)\rho^{-kp}<\varepsilon$ whenever $k\ge k_\varepsilon$. The two estimates above, together with Lemma \ref{lem:dist-of-mp-change-meas}, now yield
\begin{align}\label{eq:pmg-dist-less-than-e}
d_{pmG}(f_i,f_i^k)<\varepsilon,\quad \textrm{for\ all\ }i\in\N\cup\{\omega\},\ \ k\ge k_\varepsilon.
\end{align}
Using the triangle inequality we have 
\begin{align*}
d_{pmG}(f_i,f_\omega)\le d_{pmG}(f_i,f_i^k)+d_{pmG}(f_i^k,f_\omega^k)+d_{pmG}(f_\omega^k,f_\omega).
\end{align*}
Since $f_i^k$ are Lipschitz for each $k$, Proposition \ref{prop:ultralim-pmG-conv-lip} and \eqref{eq:pmg-dist-less-than-e} give
\begin{align*}
	\lim_\omega d_{pmG}(f_i,f_\omega)\le \limsup_{i\to\infty}d_{pmG}(f_i,f_i^k)+0+\limsup_{i\to\infty}d_{pmG}(f_\omega,f_\omega^k)<2\varepsilon
\end{align*}
for every $k\ge k_\varepsilon$. Since $\varepsilon>0$ is arbitrary, this completes the proof.
\end{proof}

\section{Tangents}\label{sec:tangents}

\subsection{Compactness and measurability of tangents of mapping packages}
\label{sec:Tangents1}
The main result of this subsection is the following theorem.
\begin{theorem}\label{thm:tan-meas+cpt}
Suppose $(X,\mu)$ is a pointwise doubling metric measure space and $f:(X,\mu)\to V$ a Lipschitz map into a Banach space $V$. There exists a $\mu$-null Borel set $N\subset X$ so that $\operatorname{Tan}_{pmG}(f,x)$ is non-empty and compact in $(\M_{pmG},d_{pmG})$ for $x\in X\setminus N$, and the correspondence $x\mapsto \operatorname{Tan}_{pmG}(f,x):X\setminus N\twoheadrightarrow (\M_{pmG},d_{pmG})$ is measurable.
\end{theorem}

We recall here that a \emph{correspondence} $\varphi:X\twoheadrightarrow Y$, i.e. a set valued map $\varphi:X\to 2^Y$, is measurable if $\varphi^\ell(F)\in\mathscr{B}(X)$ for every closed $F\subset Y$,
where 
\begin{align*}
    \varphi^\ell(A)=\{x\in X:\ \varphi(x)\cap A\ne \varnothing\},\quad A\subset Y,
\end{align*}
see \cite[Chapter 18]{AliprantisBorder99}.

\subsubsection*{Rescaling of Lipschitz maps into Banach spaces}

Recall that given a metric measure space $(X,\mu)$ and $\bar x\in X$, $T_{\bar x,r}X$ denotes the pointed metric measure space $\big(r\inv X,\frac{\mu}{\varphi_\mu(\bar x,r)},\bar x\big)$ (if $\varphi_\mu(\bar x,r)=0$ we set $T_{\bar x,r}X=(r\inv X,0,\bar x)$). Similarly, given a mapping package $f:(X,\mu)\to V$ into a Banach space $V$ and $\bar x\in\operatorname{spt}\mu$, $T_{\bar x,r}f$ denotes the mapping package
\begin{align*}
    T_{\bar x,r}f:T_{\bar x,r}X\to V,\quad T_{\bar x,r}f=\frac{f-f(\bar x)}{r}.
\end{align*}
\begin{remark}\label{rmk:lip}
	Note that if $f:(X,\mu)\to V$ is $L$-Lipschitz, then the rescaling $T_{x,r}f$ is $L$-Lipschitz for all $x\in {\rm spt}\mu$ and $r>0$.
\end{remark}

\begin{lemma}\label{lem:T-cont}
Let $f_i:(X_i,\mu_i,\bar x_i)\to V_i$ and $f_\infty:(X_\infty,\mu_\infty,\bar x_\infty)\to V$ be $L$-Lipschitz mapping packages into Banach spaces with $d_{pmG}(f_i, f_\infty)\stackrel{i\to\infty}{\longrightarrow}0$. Let $\iota_i:X_i\to (Z,\bar z)$, $\tilde\iota_i:V_i\to (W,\bar w)$ ($i\in \N\cup\{\infty\}$) be as in Corollary \ref{cor:lip-map-conv}. Suppose $z\in \operatorname{spt}\mu_\infty$ and $z_i\in \operatorname{spt}\mu_i$ be such that $\iota_i(z_i)\to \iota_\infty(z)$. Then for every $r>0$ there exist $r_i\to r$ such that
\[d_{pmG}(T_{z_i,r_i}f_i, T_{z,r}f_\infty)\stackrel{i\to\infty}{\longrightarrow}0.
\]
If $D_\varphi$ is the set of discontinuities of $\varphi$ and $\mu(\{w:d(z,w)/r\in D_\varphi\})=0$, one may take $r_i=r$.
\end{lemma}
\begin{proof}
The set $D_\varphi\subset(0,\infty)$ is countable. For fixed $z\in X$, $\mu_\infty(\{w:\ d(z,w)=s \})>0$ for at most countably many $s$. Thus the condition  
\begin{align}\label{eq:regular-radii}
\mu_\infty(\{w:d(z,w)/s\in D_\varphi\})=0
\end{align}
fails for at most countably many $s>0$. If $s>0$ satisfies \eqref{eq:regular-radii}, the convergences $\iota_{i\ast}\mu_i\rightharpoonup\iota_{\infty\ast}\mu_\infty$ and $\iota_i(z_i)\to \iota_\infty(z)$ imply
\[
\varphi_{\mu_i}(z_i,s)\longrightarrow\varphi_{\mu_\infty}(z,s)>0.
\]
Indeed (identifying $X_i$ with its image $\iota_i(X_i)$ for $i\in\N\cup\{\infty\}$), whenever $d(z_i,z)<\varepsilon$ we have
\begin{align*}
\varphi_\varepsilon^+(w):=\varphi\Big(\frac{(d(z,w)-\varepsilon)_+}{s}\Big)\ge \varphi\Big(\frac{d(z_i,w)}{s}\Big)\ge \varphi\Big(\frac{d(z,w)+\varepsilon}{s}\Big)=:\varphi_\varepsilon^-(w)
\end{align*}
and the functions $\varphi_\varepsilon^\pm$ are $\mu$-a.e. continuous, for a.e. $\varepsilon>0$. The portmanteau theorem yields
\begin{align*}
\lim_{i\to\infty}\int\varphi_\varepsilon^\pm\ud\mu_i=\int\varphi_\varepsilon^\pm\ud\mu_\infty,
\end{align*}
and the dominated convergence theorem yields
\begin{align*}
\lim_{m\to\infty}\int\varphi_{\varepsilon_m}^\pm\ud\mu_\infty=\varphi_{\mu_\infty}(z,s)
\end{align*}
for $\varepsilon_m\downarrow0$ chosen such that $\varphi_{\varepsilon_m}^\pm$ are $\mu$-a.e. continuous. Together these estimates imply the claimed convergence.

Since $\tilde\iota_i(f_i(z_i))\to \tilde\iota_\infty(f_\infty(z))$ (cf. Corollary \ref{cor:lip-map-conv}) it is not difficult to see that $T_{z_i,s}f_i\stackrel{pmG}{\longrightarrow}T_{z,s}f_\infty$. 

If $\varphi$ is lower semicontinuous, it is right continuous; choose a sequence $s_k\uparrow r$ such that \eqref{eq:regular-radii} holds for $\mu$ replaced by $\mu_i$ and $z$ by $z_i$ for each $i\in\N\cup\{\varnothing\}$ and $k\in\N$. If it is upper semicontinuous, it is left continuous on $(0,\infty)$; choose $s_k\downarrow r$ similarly. In each case we have
\[
\varphi_{\mu_i}(z_i,s_k)\stackrel{k\to\infty}{\longrightarrow}\varphi_{\mu_i}(z_i,r)
\quad\textrm{and}\quad T_{z_i,s_k}f_i\underset{k\to\infty}{\stackrel{pmG}{\longrightarrow}} T_{z_i,r}f_i\quad\textrm{for\ all\ }i\in\N\cup\{\infty\}.
\]
For each fixed $k$, we have $T_{z_i,s_k}f_i\stackrel{pmG}{\longrightarrow}T_{z,s_k}f_\infty$. Choose $k(i)$ such that $d_{pmG}(T_{z_i,s_{k(i)}}f_i,T_{z_i,r}f_i)\le 2^{-i}$. Setting $r_i:=s_{k(i)}$, the triangle inequality yields
\begin{align*}
d_{pmG}(T_{z_i,r_i}f_i,T_{z,r}f_\infty) \le & \ d_{mpG}(T_{z_i,s_{k(i)}}f_i,T_{z_i,s_k}f_i)+d_{pmG}(T_{z_i,s_k}f_i,T_{z,s_k}f_\infty)\\
& +d_{pmG}(T_{z,s_k}f_\infty,T_{z,r}f_\infty).
\end{align*}
Taking $k\to\infty$ the convergences established above yield $d_{pmG}(T_{z_i,r_i}f_i,T_{z,r}f_\infty)\le 2^{-i}$, which proves the claim.
\end{proof}

\subsubsection*{Measurability and compactness} Next, given a map $f:(X,\mu)\to V$, denote
\begin{align*}
T(f,x,R)=\{T_{x,r}f:\ r\in {(0,R)}\}\quad\textrm{and}\quad \Phi(f)=\{T_{y,r}f:\ y\in\operatorname{spt}\mu,\ r>0\}.
\end{align*}
Note that $\Phi$ is well-defined on the equivalence classes in $\M_{pmG}$. In what follows, $\M_{pmG}(L)$ denotes the equivalence classes of $L$-Lipschitz mapping packages.

Recall that a correspondence $\varphi:X\twoheadrightarrow Y$ is \textit{weakly measurable} if $\varphi^\ell(U)$ is Borel for every open $U\subset Y$, and \emph{lower hemicontinuous} if $\varphi^\ell(U)$ is open for every open $U\subset Y$; see \cite[Chapter 18]{AliprantisBorder99}.

\begin{lemma}\label{lem:Phi-lhc}
For $L>0$, the correspondence $\Phi:(\M_{pmG}(L),d_{pmG})\twoheadrightarrow(\M_{pmG}(L),d_{pmG})$ and its closure correspondence $f\mapsto \overline{\Phi(f)}$ are lower hemicontinuous.
\end{lemma}

\begin{lemma} \label{lem:T-weakly-meas} Suppose $f:(X,\mu)\to V$ is a Lipschitz map and $R>0$. Then, the correspondence $x\mapsto T(f,x,R):\operatorname{spt}\mu\twoheadrightarrow(\M_{pmG},d_{pmG})$ is lower hemicontinuous, and thus weakly measurable. 
\end{lemma}

\begin{lemma}\label{lem:tan-precpt}
Suppose $(X,\mu)$ is pointwise doubling and $f:(X,\mu)\to V$ is Lipschitz. Then $T(f,x,R)$ is precompact for all $R>0$ in $(\M,d_{pmG})$ for $\mu$-a.e. $x\in X$.
\end{lemma}

\begin{proof}[Proof of Lemma \ref{lem:Phi-lhc}]
Suppose $f_i\stackrel{pmG}{\longrightarrow} f$ and $T_{z,r}f\in\mathcal U$ for some $z\in \operatorname{spt}\mu\cup\{\bar x\}$, $r>0$, where $\mathcal U$ is open. By Corollary \ref{cor:lip-map-conv} we may assume that $(X_i,\bar x_i)\subset (Z,z)$ and $(Y_i,y_i)\subset (W,w)$ for some ambient spaces. By Lemma \ref{lemma:support_point} we may find $z_i\in\operatorname{spt}\mu_i$ converging to $z$. Lemma \ref{lem:T-cont} gives $r_i\to r$ with $T_{z_i,r_i}f_i\to T_{z,r}f$. Hence $\Phi(f_i)\cap\mathcal U\ne\varnothing$ eventually. This proves lower hemicontinuity. Since $\overline\Phi^\ell(\mathcal U)=\Phi^\ell(\mathcal U)$ whenever $\mathcal U\subset \M_{pmG}(L)$ is open, $\overline\Phi$ is lower hemicontinuous as well.
\end{proof}

\begin{proof}[Proof of Lemma \ref{lem:T-weakly-meas}]
Let $\mathcal U\subset \M_{pmG}$ be open and $x\in T(f,\cdot,R)^\ell(\mathcal U)$, i.e. $T_{x,r}f\in T(f,x,R)\cap \mathcal U$ for some $r\in (0,R)$. If $x_i\to x$ in $\operatorname{spt}\mu$, apply Lemma \ref{lem:T-cont} to the constant sequence of mapping packages $f$, with moving centres $x_i$. There are $r_i\to r$ such that $d_{pmG}(T_{x_i,r_i}f, T_{x,r}f)\to 0$. Since $r_i<R$ eventually and $\mathcal U$ is open, we have $T_{x_i,r_i}f\in T(f,x_i,R)\cap \mathcal U$ eventually. This proves that $x\mapsto T(f,x,R)$ is lower hemicontinuous, and therefore weakly measurable.
\end{proof}

\begin{proof}[Proof of Lemma \ref{lem:tan-precpt}]
By \cite[Proposition 5.3]{Bate22} (see also \cite[Theorem B.5]{SP21}), we have that the sets
\begin{align*}
\mathcal F_R(x)=\{T_{x,r}X:\ r\in (0,R)\}
\end{align*}
are precompact for $\mu$-a.e. $x\in X$. Note that both references the authors normalize the measure by $\frac 1{\mu(B(x,r))}$ instead of $\frac 1{\varphi_\mu(x,r)}$, but for $\mu$-a.e. $x\in X$, \eqref{eq:normalizing-const} implies the precompactness with the normalization used in this paper.

All the maps $T_{x,r}f\in T(f,x,R)$ have Lipschitz constant at most $\operatorname{LIP}(f)$. For $x\in X$ such that $\mathcal F_R(x)$ is precompact, Proposition \ref{prop:ultralim-pmG-conv-lip} therefore gives a convergent subsequence of every sequence in $T(f,x,R)$. This completes the proof.
\end{proof}

\begin{proof}[Proof of Theorem \ref{thm:tan-meas+cpt}]
Set $\overline T_{pmG}(f,x,R)=\operatorname{cl}_{d_{pmG}}T(f,x,R)$. By definition,
\begin{align}\label{eq:intersection_tangent}
\operatorname{Tan}_{pmG}(f,x)=\bigcap_{j\in\N}\overline T_{pmG}(f,x,1/j).
\end{align}
By Lemma \ref{lem:T-weakly-meas} and \cite[Lemma 18.3]{AliprantisBorder99}, these closure correspondences are weakly measurable. Lemma \ref{lem:tan-precpt} gives a Borel null set $N$ outside which their values are nonempty and compact. By \cite[Lemmata 18.2 and 18.4]{AliprantisBorder99}, their decreasing intersection is measurable on $X\setminus N$, and is nonempty and compact.
\end{proof}

\subsection{Preiss's phenomenon for pmG-tangents} In this subsection we prove the validity of Preiss's phenomenon for pmG-tangents of mapping packages.
\label{sec:Tangents2}

\begin{theorem}\label{thm:preiss-phenomenon}
Suppose $(X,\mu)$ is pointwise doubling and $f:(X,\mu)\to V$ is a Lipschitz map into a Banach space. For $\mu$-a.e. $x$, if $g:(Y,\nu,o)\to W$ belongs to $\operatorname{Tan}_{pmG}(f,x)$, then $T_{y,\rho}g$ belongs to $\operatorname{Tan}_{pmG}(f,x)$ for every $y\in\operatorname{spt}\nu$ and $\rho>0$.
\end{theorem}

Here each $g\in \operatorname{Tan}_{pmG}(f,x)$ is understood as the Lipschitz representative on $\operatorname{spt}\nu$. We need two auxiliary results. The first applies to arbitrary measurable mapping packages, not just Lipschitz maps.
\begin{lemma}\label{lem:density-pt-tangent}
Let $(X,\mu)$ be pointwise doubling, let $f:(X,\mu)\to V$ be a mapping package, and let $E\subset X$ be closed. If $x\in E\cap\operatorname{spt}(\mu|_E)$ is a Lebesgue density point of $E$ with $C_\mu(x)<\infty$, then
\[
d_{pmG}(T_{x,r}f,T_{x,r}(f|_E))\longrightarrow0\qquad(r\to0).
\]
In particular $\operatorname{Tan}_{pmG}(f,x)=\operatorname{Tan}_{pmG}(f|_E,x)$ for $\mu$-a.e. $x\in E$.
\end{lemma}
Here $f|_E$ has domain $(E,\mu|_E,x)$ and $T_{x,r}(f|_E)=\frac{f-f(x)}{r}:\Big(r\inv E,\frac{\mu|_E}{\varphi_{\mu|_E}(x,r)},x\Big)\to V$.

For the next statement, recall that the composition $X\twoheadrightarrow Z$ of two correspondences $\Phi: X\twoheadrightarrow Y$ and $\Psi:Y \twoheadrightarrow Z$ is defined by 
\[
\Psi\circ\Phi(x)=\bigcup_{y\in\Phi(x)}\Psi(y).
\]

\begin{lemma}\label{lem:PhicircTan-wm}
For a Lipschitz map $f$ and the null set $N$ above, the graph of
\[
x\longmapsto\overline{\Phi\circ\operatorname{Tan}_{pmG}}(f,x)
\]
is a Borel subset of $(X\setminus N)\times\M_{pmG}$.
\end{lemma}

\begin{proof}[Proof of Lemma \ref{lem:PhicircTan-wm}]
For open $\mathcal U$, we have
\[
(\Phi\circ\operatorname{Tan}_{pmG})^\ell(\mathcal U)
=\operatorname{Tan}_{pmG}^\ell(\Phi^\ell(\mathcal U)).
\]
The set on the right is Borel by Lemma \ref{lem:Phi-lhc} and Theorem \ref{thm:tan-meas+cpt}. Thus the composition, and hence its closure, is weakly measurable. Its values are closed in the complete separable space $\M_{pmG}$, so its graph is Borel by \cite[Theorem 18.6]{AliprantisBorder99}.
\end{proof}

\begin{proof}[Proof of Lemma \ref{lem:density-pt-tangent}]
By the triangle inequality we have
\begin{align*}
d_{pmG}(T_{x,r}f,T_{x,r}(f|_E))\le d_{pmG}(T_{x,r}f,\widetilde T_{x,r}(f|_E))+d_{pmG}(\widetilde T_{x,r}(f|_E), T_{x,r}(f|_E)),
\end{align*}
where $\displaystyle \widetilde T_{x,r}(f|_E)=\frac{f-f(x)}{r}:\Big(r\inv E,\frac{\mu|_E}{\varphi_\mu(x,r)},x\Big)\to V$. We will prove that both terms on the RHS tend to zero for $x$ as in the claim. 

For every fixed $S>0$, density, pointwise doubling and \eqref{eq:normalizing-const} give
\begin{align}\label{eq:density-pt-pmG}
\frac{\mu(B(x,Sr)\setminus E)}{\varphi_\mu(x,r)}\longrightarrow0.
\end{align}
Moreover, using the constants fixed in Section \ref{sec:conv_tangents_mms},
\[
0\le1-\frac{\varphi_{\mu|_E}(x,r)}{\varphi_\mu(x,r)}
\le\frac{M\mu(B(x,Ar)\setminus E)}{\varphi_\mu(x,r)}\longrightarrow0,
\]
so that
\begin{align}\label{eq:density-normalization}
\frac{\varphi_\mu(x,r)}{\varphi_{\mu|_E}(x,r)}\longrightarrow1.
\end{align}
Lemma \ref{lem:dist-of-mp-change-meas} and \eqref{eq:density-pt-pmG} imply that $d_{pmG}(T_{x,r}f,\widetilde T_{x,r}(f|_E))\stackrel{r\to 0}{\longrightarrow}0$. Similarly, Lemma \ref{lem:dist-of-mp-change-meas} and \eqref{eq:density-normalization} imply that $d_{pmG}(\widetilde T_{x,r}(f|_E), T_{x,r}(f|_E))\stackrel{r\to 0}{\longrightarrow}0$. This completes the proof.
\end{proof}

\begin{proof}[Proof of Theorem \ref{thm:preiss-phenomenon}]
Write $d=d_{pmG}$ and let $N$ be as in Theorem \ref{thm:tan-meas+cpt}. For each $k\in\N$, cover $\M_{pmG}$ by countably many closed sets $\M_n^k$ of diameter at most $1/(2k)$. Define
\[
H(x)=\overline{\Phi\circ\operatorname{Tan}_{pmG}}(f,x)
\]
and
\begin{align*}
\mathcal S_{n,m,k}&=\{(x,h):\;x\in X\setminus N,\ h\in H(x)\cap\M_n^k,\,d(h,T(f,x,1/m))\ge1/k\},\\
B_{n,m,k}&=\pi_X(\mathcal S_{n,m,k}).
\end{align*}
Set $\mathcal S_{n,m,k}$ is Borel by Lemma \ref{lem:PhicircTan-wm}. Its projection is analytic and therefore $\mu$-measurable, by \cite[Theorems 12.24 and 12.41]{AliprantisBorder99}. Note that $H(x)\not\subset \operatorname{Tan}_{pmG}(f,x)$ if and only if $x\in \bigcup_{n,m,k}B_{n,m,k}$. Thus it suffices to show that $\mu(B_{n,m,k})=0$ for all $n,m,k$.

Suppose $\mu(B_{n,m,k})>0$. Choose a compact $B\subset B_{n,m,k}$ of positive measure and a density point $b\in B\cap\operatorname{spt}(\mu|_B)$ with $C_\mu(b)<\infty$. There is $h\in H(b)\cap\M_n^k$. By the definition of $H(b)$, choose $g:(Y,\nu,o)\to W$ in $\operatorname{Tan}_{pmG}(f,b)$, $y\in\operatorname{spt}\nu$ and $\rho>0$ such that
\[
d(h,T_{y,\rho}g)<1/(4k).
\]
Take $r_i\to0$ with $T_{b,r_i}f\overset{pmG}{\longrightarrow} g$. By Lemma \ref{lem:density-pt-tangent}, also $T_{b,r_i}(f|_B)\overset{pmG}{\longrightarrow} g$. In the common domain realization for the full maps, the restricted domain measures have the same weak limit: the removed mass on every bounded ball tends to zero by \eqref{eq:density-pt-pmG} and \eqref{eq:density-normalization}. Therefore there are $b_i\in\operatorname{spt}(\mu|_B)\subset B$ whose images in the rescaled domains converge to $y$. Lemma \ref{lem:T-cont}, applied to the full maps $T_{b,r_i}f$ at these centers, gives $\rho_i\to\rho$ with
\[
T_{b_i,\rho_i}\big(T_{b,r_i}f\big)
=T_{b_i,\rho_i r_i}f\overset{pmG}{\longrightarrow} T_{y,\rho}g.
\]
The identity follows directly by cancelling the two normalizing constants. Put $t_i=\rho_i r_i$. For large $i$, we have $t_i<1/m$ and $d(T_{b_i,t_i}f,T_{y,\rho}g)<1/(4k)$. Since $b_i\in B_{n,m,k}$, there exists $h_i\in\M_n^k$ with $d(h_i,T_{b_i,t_i}f)\ge1/k$. But
\[
\frac1k\le d(h_i,T_{b_i,t_i}f)
\le d(h_i,h)+d(h,T_{y,\rho}g)+d(T_{y,\rho}g,T_{b_i,t_i}f)
<\frac1{2k}+\frac1{4k}+\frac1{4k}=\frac1k,
\]
a contradiction. Thus all $B_{n,m,k}$ are null, which proves the theorem.
\end{proof}

\begin{proof}[Proof of Theorem \ref{thm:Gr-preiss}]
Let $g$ be a Haj\l asz gradient of $f$, finite almost everywhere, and discard the null set in the Haj\l asz inequality. On $\{g\le n\}$ the map $f$ is $2n$-Lipschitz. By inner regularity and bounded finiteness, choose countably many compact sets $K_j$, each contained in one of these level sets, whose union has full measure. Set $\mu_j=\mu|_{K_j}$ and $f_j=f|_{K_j}$.

For $\mu_j$-almost every $x\in K_j$ we have that $x\in K_j\cap \operatorname{spt}(\mu|_{K_j})$, $x$ is a density point of $K_j$, and $C_\mu(x)<\infty$. For such $x$ we have $C_{\mu_j}(x)<\infty$, and by Theorems \ref{thm:tan-meas+cpt} and \ref{thm:preiss-phenomenon} we have that (i) $\operatorname{Tan}_{pmG}(f_j,x)$ is non-empty and compact, and (ii) $T_{y,\rho}g\in \operatorname{Tan}_{pmG}(f_j,x)$ whenever $g:(Y,\nu,o)\to W$ belongs to $\operatorname{Tan}_{pmG}(f_j,x)$ and $y\in\operatorname{spt}\mu_j$, $\rho>0$. 

On the other hand, by Lemma \ref{lem:density-pt-tangent}, $\operatorname{Tan}_{pmG}(f,x)=\operatorname{Tan}_{pmG}(f_j,x)$ at every such point. Thus (i) and (ii) hold for $f$ for $\mu$-almost every $x\in X$, completing the proof of Theorem \ref{thm:Gr-preiss}. 
\end{proof}

\begin{remark}
A similar truncation argument shows that Theorems \ref{thm:tan-meas+cpt} and \ref{thm:preiss-phenomenon} remain true for Haj\l asz--Sobolev maps. We omit the details.
\end{remark}

\subsection{Preiss's phenomenon for pmGH-tangents} 
\label{sec:Tangents3}
In this subsection we prove Theorem \ref{thm:pmGH-preiss}, that is Preiss's phenomenon for pmGH tangents. A key role in the proof is played by porosity (Section \ref{sec:MMTN}) and the behaviour of the underlying measure on porous sets.

\begin{lemma}\label{lem:porosity}
    Let $(X,d)$ be a metric space and $S\subset \X$. A set \(S\) is not porous at \(x\in S\), if and only if for every \(\varepsilon >0\) there exists \(R_{x,\varepsilon}>0\) such that \(S\cap B(x,(1+\varepsilon)r)\) is \(\varepsilon r\)-dense in \(B(x,r)\) for every \(r< R_{x,\varepsilon}\).
\end{lemma}

\begin{proof}[Proof of Lemma \ref{lem:porosity}]
First, fix a point \(x\in S\) at which \(S\) is not porous. By definition of non-porosity point \(x\), for every \(\varepsilon>0\) there exists \(R_{x,\varepsilon}>0\) such that for every \(y\in B(x,R_{x,\varepsilon})\) we have \(S\cap B(y,\varepsilon \d(x,y))\neq\varnothing\). Let \(r\in (0,R_{x,\varepsilon})\) and take any \(y\in B(x,r)\). Since \(d(x,y)<r\), we have \(B(y,\varepsilon\,\d(x,y))\subseteq B(y,\varepsilon r)\). Hence, there exists \(z\in S\cap B(y,\varepsilon r)\). Moreover, $d(x,z)\leq d(x,y)+d(y,z)<(1+\varepsilon)r$, whence $z\in B(x,(1+\varepsilon)r)$. Therefore, $z\in S\cap B(x,(1+\varepsilon)r)$ and $d(z,y)<\varepsilon r$. Thus, $B(x,r)\subset \big(B(x,(1+\varepsilon)r)\cap S\big)^{\varepsilon r}$.

Suppose now that for every \(\varepsilon>0\) there exists \(R_{x,\varepsilon}>0\) such that for all \(r\in (0,R_{x,\varepsilon})\) the set \(S\cap B(x,(1+\varepsilon)r)\) is \(\varepsilon r\)-dense in \(B(x,r)\). Fix an arbitrary \(\eta>0\). Then there exists \(R_{x,\eta}>0\) such that for all \(r\in (0,R_{x,\eta})\) the set \(S\cap B(x,(1+\eta/2)r)\) is \(\frac{\eta}{2}r\)-dense in \(B(x,r)\). Take \(\varepsilon=\frac{R_{x,\eta}}{2}\) and an arbitrary \(y\in B(x,\varepsilon)\). Then we have \(2d(x,y)<2\varepsilon=R_{x,\eta}\). Hence, \(S\cap B(x,2(1+\eta/2)d(x,y))\) is \(\eta\,d(x,y)\)-dense in \(B(x,2d(x,y))\). Since \(y\in B(x,2d(x,y))\), there exists \(z\in S\cap B(x,2(1+\eta/2)d(x,y))\) such that \(d(z,y)<\frac{\eta}{2}\cdot 2d(x,y)=\eta d(x,y)\). Therefore, \(z\in S\cap B(y,\eta\, d(x,y))\).  By the arbitrariness of \(y\in B(x,\varepsilon)\), we conclude that $x$ is a non-porosity point of $S$.  
\end{proof}

\begin{remark} \label{rem:non-porosity} Suppose that every porous set in a metric measure space \((\X,\d,\mu)\) has measure zero. Then for any set \(S \subset \X\) we have that \(S\) is not porous at \(x\) for \(\mu\)-a.e. \(x \in S\). Indeed, the set \(P_S\coloneqq\{x\in S\,:\, S \text{ is porous at } x\}\) is porous itself and therefore has measure zero. Hence, \(\mu\)-a.e. point of \(S\) is a non-porosity point of $S$.
\end{remark}

\begin{proposition}\label{prop:pmGpmGH}
Let $(X,\mu)$ be a metric measure space. Suppose $\mu$ is $(\bar C,\bar R)$-doubling along a closed set $A\subset X$ for some $\bar C,\bar R>0$. Let $\bar x\in {\rm spt}\mu \cap A$ be a point at which $A$ is not porous. Given a sequence $r_i\searrow0$, we have that $T_{\bar x,r_i}X\overset{pmG}{\longrightarrow}(X_\infty,\mu_\infty,\bar x_\infty)$ if and only if  $T_{\bar x,r_i}X\overset{pmGH}{\longrightarrow}(X_\infty,\mu_\infty,\bar x_\infty)$.
\end{proposition}

Proposition \ref{prop:pmGpmGH} has the following corollary. Recall the definition of pmG and pmGH-convergence of equi-Lipschitz mapping packages (Definition \ref{def:equi-lip-maps-conv}).

\begin{corollary}\label{cor:pmGpmGH}
Let $(X,\mu)$ be a metric measure space and $f:(X,\mu)\to V$ a Lipschitz map into a finite dimensional Banach space. Suppose $A\subset X$ is closed and $\mu$ is $(C,R)$-doubling along $A$. If $A$ is not porous at $\bar x\in A$ then, for any $r_i\searrow 0$, we have that $T_{\bar x,r_i}f\stackrel{pmGH}{\longrightarrow}f_\infty$ if and only if $T_{\bar x,r_i}f\stackrel{pmG}{\longrightarrow}f_\infty$ in the sense of Definition \ref{def:equi-lip-maps-conv}.
\end{corollary}

In the proof of Proposition \ref{prop:pmGpmGH} we use a characterization of pmG and pmGH-convergence in terms of  approximate isometries. Given $(X_i,\mu_i,\bar x_i),(X_\infty,\mu_\infty,\bar x_\infty)\in \XX$ and $\varepsilon_i,R_i>0$, a \textit{weak $(R_i,\varepsilon_i)$-approximation} $\varphi_i:X_i\to X_\infty$ is a Borel map for which there exists a Borel set $\tilde X_i\subset B(\bar x_i,R_i)$ that satisfies
 \begin{align}
    |d_i(x,y)-d_\infty(\varphi_i(x),\varphi_i(y))|\leq \varepsilon_i,\quad  \textnormal{for all } x,y\in \tilde X_i,\label{eq:weakapp1}\\
     \mu_i(B(\bar x_i,R_i)\setminus {\tilde X_i})\leq \varepsilon_i\label{eq:weakapp2}\\
     \mu_\infty(B(\bar x_\infty,R_i-\varepsilon_i)\setminus \varphi_i(\tilde X_i)^{\varepsilon_i})\leq \varepsilon_i.\label{eq:weakapp3}
 \end{align}

Similarly, $\varphi_i$ is called a \textit{$(R_i,\varepsilon_i)$-approximation} if the set $\tilde X_i$ can be chosen to be $B(\bar x_i,R_i)$ and \eqref{eq:weakapp3} is replaced by
\begin{align}
    B(\bar x_\infty,R_i-\varepsilon_i)\subset \varphi_i(B(\bar x_i,R_i))^{\varepsilon_i} \label{eq:app4}.
\end{align}

By \cite[Theorem 12.2]{SP21} we have that $X_i\stackrel{pmG}{\longrightarrow}X_\infty$ if and only if there exists a sequence of weak $(R_i,\varepsilon_i)$-approximations $\varphi_i:X_i\to X_\infty$ for some $R_i\nearrow\infty$ and $\varepsilon_i\searrow 0$, so that $\varphi_{{i}\ast}\mu_i\rightharpoonup \mu_\infty$ in duality with $C_{bbs}(X_\infty)$. We also have that $X_i\stackrel{pmGH}{\longrightarrow}X_\infty$ if and only if there exist $(R_i,\varepsilon_i)$-approximations $\varphi_i:X_i\to X_\infty$ for some $R_i\nearrow\infty$ and $\varepsilon_i\searrow 0$, so that $\varphi_{{i}\ast}\mu_i\rightharpoonup \mu_\infty$ in duality with $C_{bbs}(X_\infty)$ (see e.g. \cite[Proposition 3.28]{GMS15}).

\begin{proof}[Proof of Proposition \ref{prop:pmGpmGH}]
The convergence  $T_{\bar x,r_i}X\stackrel{pmGH}{\longrightarrow}X_\infty$ implies $T_{\bar x,r_i}X\stackrel{pmG}{\longrightarrow}X_\infty$ by definition. Thus it suffices to prove the converse implication. Let $\varphi_i:r_i\inv X\to X_\infty$ be weak $(R_i,\varepsilon_i)$-approximations, $i\in\N$ that establish convergence $(r_i\inv X,\mu_{\bar x,r_i},\bar x)\overset{pmG}{\longrightarrow}(X_\infty,\mu_\infty,x_\infty)$, where $R_i\nearrow \infty$ and $\varepsilon_i\searrow 0$. That is, there exist Borel sets $\tilde X_i\subset B_i(\bar x,R_i)$, $i\in\N$ that satisfy properties \eqref{eq:weakapp1} - \eqref{eq:weakapp3}. Moreover by \cite[Lemma 5.6]{Bate22} we may assume that $(X_\infty,\mu_\infty,x_\infty)$ is doubling. 

\textsc{Step 1.} We claim that for any fixed $R>1$, $\varepsilon>0$ there exists $i_{R,\varepsilon}\in\N$ such that for all $i\geq i_{R,\varepsilon}$ and every $y\in B_i(\bar x,R)$ we have $\tilde X_i \cap B_i(y,\varepsilon)\neq \varnothing$. Fix $R>1$, $\varepsilon>0$. Suppose the opposite, that there exists a subsequence $i_k\to \infty$ such that for every $k\in \N$ there exists $y^{i_k}\in B_{i_k}(\bar x,R)$ such that $\tilde X_{i_k}\cap B_{i_k}(y^{i_k},\varepsilon)=\varnothing$. Therefore, we have $B_{i_k}(y^{i_k},\varepsilon)\subset B_{i_k}(\bar x,R+\varepsilon)\setminus \tilde X_{i_k}$. Moreover, since \(\bar x\in A\) is a point at which $A$ is not porous, there exists $R_{\bar x,\frac{\varepsilon}{2R}}>0$ such that for all $r<R_{\bar x,\frac{\varepsilon}{2R}}$ the set $A\cap B(\bar x,(1+\frac{\varepsilon}{2R})r)$ is $\frac{\varepsilon r}{2R}$-dense in $B(\bar x,r)$. Since  $r_{i_k}R<R_{\bar x,\frac{\varepsilon}{2R}}$ for $k$ large enough, it follows that $A\cap B(\bar x,r_{i_k}R(1+\frac{\varepsilon}{2R}))$ is $\frac{\varepsilon r_{i_k}}{2}$-dense in $B(\bar x,r_{i_k}R)$. So, for $y^{i_k}\in B_{i_k}(\bar x,R)$ there exists $z^{i_k}\in A\cap B(\bar x,r_{i_k}R(1+\frac{\varepsilon}{2R}))$ such that $d(y^{i_k},z^{i_k})<\frac{\varepsilon r_{i_k}}{2}$. Therefore, $B_{i_k}(z^{i_k},\varepsilon/2)\subset B_{i_k}(y^{i_k},\varepsilon)$. Hence, for $k$ large enough, we get $z^{i_k}\in A$ so that
\[
B_{i_k}(z^{i_k},\varepsilon/2)\subset B_{i_k}(y^{i_k},\varepsilon)\subset B_{i_k}(\bar x,R+\varepsilon)\setminus \tilde X_{i_k}.
\]
If we denote $\mu_i\coloneqq\mu_{\bar x,r_i}$, then for $k$ so large that $R_{i_k}>R+\varepsilon$, we have
\[
\mu_{i_k}(B_{i_k}(z^{i_k},\varepsilon/2))\leq \mu_{i_k}(B_{i_k}(\bar x,R+\varepsilon)\setminus \tilde X_{i_k})\leq \mu_{i_k}(B_{i_k}(\bar x,R_{i_k})\setminus \tilde X_{i_k})\leq \varepsilon_{i_k}.
\]
On the other hand, by \eqref{eq:normalizing-const} there is $K_{\bar x}<\infty$ with $\varphi_\mu(\bar x,r)\le K_{\bar x}\mu(B(\bar x,r))$ for all sufficiently small $r$, so we have 
\[
\mu_{i_k}(B_{i_k}(z^{i_k},\varepsilon/2))\underset{R>1}{\geq}\frac1{K_{\bar x}}\frac{\mu(B(z^{i_k},\frac{r_{i_k}\varepsilon}{2}))}{\mu(B(\bar x,(R+\varepsilon)r_{i_k}))}.
\]
Moreover, 
\[
A\ni z^{i_k}\in B(\bar x,(R+\varepsilon)r_{i_k})\subseteq B(z^{i_k},2(R+\varepsilon)r_{i_k})\subset B(z^{i_k},\bar R),
\]
for $k$ so large that $2(R+\varepsilon)r_{i_k}< \bar R$ and $\frac{\varepsilon r_{i_k}}{2}<\bar R$. Now we may apply \cite[Lemma 2.1]{Bate2018} and conclude that for $k$ large enough, we have
\[
\frac{\mu(B(z^{i_k},\frac{r_{i_k}\varepsilon}{2}))}{\mu(B(\bar x,(R+\varepsilon)r_{i_k}))}\geq 4^{-\log_2\bar C}\left(\frac{\varepsilon}{2(R+\varepsilon)}\right)^{\log_2\bar C}.
\]
Therefore, it follows 
\[
\frac1{K_{\bar x}}4^{-\log_2\bar C}\left(\frac{\varepsilon}{2(R+\varepsilon)}\right)^{\log_2\bar C}\leq \varepsilon_{i_k}\overset{k\to\infty}{\longrightarrow} 0.
\]
Thus, we get a contradiction.
This concludes the proof of \textsc{Step 1}. Moreover, since we proved the claim for an arbitrary $\varepsilon>0$, it follows that for any fixed $R>1$ and $\varepsilon>0$ there exists $i_{R,\varepsilon}$ such that for all $i\geq i_{R,\varepsilon}$ and every $y\in B_i(\bar x,R)$ we have $D_i\cap B_i(y,\varepsilon)\neq \varnothing$, where $D_i\subset \tilde X_i$ is a countable and dense subset with respect to the metric $r_i^{-1}d$. 

\textsc{Step 2.} In this step we prove that for every $R>1$, and every $\varepsilon>0$ such that $R-2\varepsilon>0$, there exists $\tilde i_{R,\varepsilon}\in\N$ so that for all $i\geq \tilde i_{R,\varepsilon}$ there exists an $(R,\varepsilon)$-approximation $\psi_i: (\X,r_i^{-1}\d)\to X_\infty$. Fix $R>1$ and $\varepsilon>0$ so that $R-2\varepsilon>0$. There exists $i_{R,\varepsilon}\in\N$ such that $\varepsilon_i<\varepsilon/3$,  $R<R_i$, and thus $B_i(\bar x,R)\subset B_i(\bar x,R_i)$ for all $i\geq i_{R,\varepsilon}$. For every $i\geq i_{R,\varepsilon}$ we define 
\[
\psi_i(x)=\begin{cases}
    \varphi_i(x), & x\in (\tilde X_i\cap B_i(\bar x,R))\cup (\X\setminus B_i(\bar x,R)),\\
    \varphi_i(\tilde x), & x\in B_i(\bar x,R)\setminus \tilde X_i,
\end{cases}
\]
where $\tilde x\coloneqq x_{j_x}^i\in D_i\subset \tilde X_i$ is the point in a countable dense subset $D_i=\{x^i_j: j\in\N\}$ of $\tilde X_i$ with  the minimal index $j$ satisfying $d_i(x,x_j^i)<\varepsilon/3$. Let us first justify that such a defined $\psi_i$ is a Borel map. Take an arbitrary Borel subset $V\subset X_\infty$. Then we have 
\[
\psi_i^{-1}(V)=\left[\varphi_i^{-1}(V)\cap((\X\setminus B_i(\bar x,R))\cup (\tilde X_i\cap B_i(\bar x,R)))\right]\cup \left[\psi_i^{-1}(V)\cap(B_i(\bar x,R)\setminus \tilde X_i)\right].
\]
The first set in the union on the right is Borel since $\varphi_i$ is a Borel map, and all of the other sets included are Borel. For the second set, to conclude that it is Borel notice that 
\[
\psi_i^{-1}(V)\cap(B_i(\bar x,R)\setminus \tilde X_i)=\bigcup_{j\,:\,\varphi_i(x^i_j)\in V}(Q_j\setminus \tilde X_i),
\]
where $W_j=\{x\in B_i(\bar x,R)\,:\, d_i(x,x^i_j)<\varepsilon/3\}$, $j\geq \N$ and $Q_1=W_1$, $Q_j=W_j\cap \big(\bigcap_{k<j}W_k^c\big)$, $j\geq 2$.

Further, let us prove that the distortion of $\psi_i$ on $B_i(\bar x,R)$ is less than $\varepsilon$. If $x\in B_i(\bar x,R)\setminus \tilde X_i$ and $y\in B_i(\bar x,R)\cap \tilde X_i$ we have 
\begin{align*}
    |d_i(x,y)-d_\infty(\psi_i(x),\psi_i(y))| & \leq |d_i(x,y)-d_i(\tilde x,y)|+|d_i(\tilde x,y)-d_\infty(\varphi_i(\tilde x),\varphi_i(y))|\\
    & \leq d_i(\tilde x,x)+\varepsilon_i<\varepsilon.
\end{align*}
Similarly, for $x,y\in B_i(\bar x,R)\setminus \tilde X_i$ we get 
\begin{align*}
   |d_i(x,y)-d_\infty(\psi_i(x),\psi_i(y))| & \leq |d_i(x,y)-d_i(\tilde x,y)|+|d_i(\tilde x,y)-d_i(\tilde x,\tilde y)|+|d_i(\tilde x,\tilde y)-d_\infty(\varphi_i(\tilde x),\varphi_i(\tilde y))|\\
   & \leq d_i(x,\tilde x)+d_i(y,\tilde y)+\varepsilon_i< \varepsilon.
\end{align*}
To conclude, in the case when $x,y\in B_i(\bar x,R)\cap \tilde X_i$ we immediately use the distortion of $\varphi_i$ on $\tilde X_i$.

Next, we are proving that there exists $\tilde i_{R,\varepsilon}\geq i_{R,\varepsilon}$ such that $B(x_\infty,R-\varepsilon)\subset \big(\psi_i(B_i(\bar x,R))\big)^\varepsilon$ for all $i\geq \tilde i_{R,\varepsilon}$. Suppose the opposite, that there exists a subsequence $i_k\to \infty$ so that for every $k\in \N$ there exists $y^{i_k}\in B(x_\infty,R-\varepsilon)$ such that for all $x\in B_{i_k}(\bar x,R)$ we have $d_\infty(\psi_{i_k}(x),y^{i_k})\geq \varepsilon$. That is, for all $k\in \N$ we have
\[
B_\infty(y^{i_k},\varepsilon/4)\subset B\left(x_\infty,R-\frac{3}{4}\varepsilon\right)\setminus \psi_{i_k}(B_{i_k}(\bar x,R))^{\varepsilon/4}.
\]
Let us now prove that for $k$ large enough, we have 
\[
B\left(x_\infty,R-\frac{3}{4}\varepsilon\right)\setminus \psi_{i_k}(B_{i_k}(\bar x,R))^{\varepsilon/4}\subset B(x_\infty, R_{i_k}-\varepsilon_{i_k})\setminus \varphi_{i_k}(\tilde X_{i_k})^{\varepsilon_{i_k}}.
\]
If not, then for a fixed $k_0\in \N$ there exists $k>k_0$ so that $R_{i_k}-\varepsilon_{i_k}>R-\frac{3}{4}\varepsilon$, $\varepsilon_{i_k}<\varepsilon/4$ and there exists $u_{i_k}\in \big(B\left(x_\infty,R-\frac{3}{4}\varepsilon\right)\setminus \psi_{i_{k}}(B_{i_k}(\bar x,R))^{\varepsilon/4}\big)\cap \varphi_{i_k}(\tilde X_{i_k})^{\varepsilon_{i_k}}$. Then it follows that there exists $\tilde x_{i_k}\in \tilde X_{i_k}$ such that $d_\infty(u_{i_k},\varphi_{i_k}(\tilde x_{i_k}))<\varepsilon_{i_k}$. Hence $d_\infty(\varphi_{i_k}(\tilde x_{i_k}),x_\infty)\leq d_\infty(\varphi_{i_k}(\tilde x_{i_k}),u_{i_k})+d_\infty(u_{i_k}, x_\infty)\leq \varepsilon_{i_k}+R-\frac{3}{4}\varepsilon$ and $d_{i_k}(\tilde x_{i_k},\bar x)\leq d_\infty(\varphi_{i_k}(\tilde x_{i_k}),\varphi_{i_k}(\bar x))+\varepsilon_{i_k}<R-\varepsilon/4$. Since $\varphi_{i_k}(\tilde x_{i_k})\in \varphi_{i_k}(\tilde X_{i_k}\cap B_{i_k}(\bar x,R-\varepsilon/4))\subset \psi_{i_k}(B_{i_k}(\bar x,R))$ and $\varphi_{i_k}(\tilde x_{i_k})\in B(u_{i_k},\varepsilon/4)$, we get a contradiction with the fact that $B(u_{i_k},\varepsilon/4)\cap \psi_{i_k}(B_{i_k}(\bar x,R))=\varnothing$. Hence, for $k$ large enough we have 
\begin{align*}
    \mu_\infty(B_\infty(y^{i_k},\varepsilon/4)) & \leq\mu_\infty \left( B\left(x_\infty,R-\frac{3}{4}\varepsilon\right)\setminus \psi_{i_k}(B_{i_k}(\bar x,R))^{\varepsilon/4}\right)\\
    & \leq \mu_\infty\left(B(x_\infty, R_{i_k}-\varepsilon_{i_k})\setminus \varphi_{i_k}(\tilde X_{i_k})^{\varepsilon_{i_k}}\right)\leq \varepsilon_{i_k}. 
\end{align*}
On the other hand, since the pmG-limit space is doubling, and $y^{i_k}\in B(x_\infty,R-\varepsilon)$ we have that 
\[
\mu_\infty(B(y^{i_k},\varepsilon/4))\geq 2^{-s}\left(\frac{\varepsilon}{8(R-\varepsilon)}\right)^s\mu_\infty(B(x_\infty,R-\varepsilon)),
\]
where $s=\log_2C$ and $C$ is a doubling constant of the limit space $X_\infty$ (see, e.g.  \cite{Heinonen2012}). Since $\varepsilon_{i_k}\to 0$ as $k\to\infty$, we have a contradiction. Finally, there exists $\tilde i_{R,\varepsilon}$ such that for all $i\geq \tilde i_{R,\varepsilon}$ $\varepsilon$-neighborhood of the set $\psi_i(B_i(\bar x,R))$ contains the ball
$B(x_\infty,R-\varepsilon)$.  

\textsc{Step 3.} In the last step we prove that starting from some $\tilde i\in \N$ we have $(\tilde R_i,\tilde \varepsilon_i)$-approximations $\tilde \psi_i:(\X,r_i^{-1}\d)\to X_\infty$, where $\tilde R_i\nearrow \infty$, $\tilde \varepsilon_i\searrow 0$ and $\tilde\psi_{i\ast}\mu_i\rightharpoonup \mu_\infty$ in duality with $C_{bbs}(X_\infty)$. Indeed, in the \textsc{Step 2} we proved that for every $n\in\N$, $n\geq 2$ there exists $i_n\in \N$ such that for all $i\geq i_n$ we have $R_i>n$ and $(n,1/n)$-approximation $\psi_i^n:(\X,r_i^{-1}\d)\to X_\infty$. We may assume $i_{n+1}>i_n$ and $i_n\nearrow \infty$. Starting with $i_2\in \N$ we define 
\[
\tilde R_i\coloneqq n,\quad \tilde \varepsilon_i\coloneqq 1/n, \quad\textnormal{for all }\, i_n\leq i<i_{n+1},\, n\in \N,\, n\geq 2.
\]
Clearly, we have $\tilde R_i\nearrow\infty$ and $\tilde \varepsilon_i\searrow 0$. Furthermore, we define 
\[
\tilde \psi_i\coloneqq \psi_i^n, \quad \textnormal{for all }\,i_n\leq i<i_{n+1},\, n\in\N,\,n\geq 2.
\]
That is, for every $i\geq \tilde i\coloneqq i_2$ we have an $(\tilde R_i,\tilde \varepsilon_i)$-approximation 
\[
\tilde \psi_i(x)=\begin{cases}
    \varphi_i(x), & x\in (\X\setminus B_i(\bar x,\tilde R_i))\cup (\tilde X_i\cap B_i(\bar x,\tilde R_i))\\
    \varphi_i(\tilde x), & x\in B_i(\bar x,\tilde R_i)\setminus \tilde X_i,
\end{cases}
\]
where $\tilde x\coloneqq x_{j_x}^i\in D_i\subset \tilde X_i$ is the point in a countable dense subset $D_i=\{x^i_j: j\in\N\}$ of $\tilde X_i$ with  the minimal index $j$ satisfying $d_i(x,x_j^i)<\tilde\varepsilon_i/3$. It remains to be proved that pushforward measures weakly converge to $\mu_\infty$. Hence, fix an arbitrary $f\in C_{bbs}(X_\infty)$. Then, for every $i\geq \tilde i$ we have 
\begin{align*}
    \left|\int_{X_\infty}f\,d\tilde\psi_{i\ast}\mu_i-\int_{X_\infty}f\,d\mu_\infty\right|\leq \left|\int_{X_\infty}f\,d\tilde\psi_{i\ast}\mu_i-\int_{X_\infty}f\,d\varphi_{i\ast}\mu_i\right|+\left|\int_{X_\infty}f\,d\varphi_{i\ast}\mu_i-\int_{X_\infty}f\,d\mu_\infty\right|.
\end{align*}
The second term on the right tends to $0$ as $i\to \infty$ due to the weak convergence $\varphi_{i\ast}\mu_i\rightharpoonup \mu_\infty$. To estimate the first term, we use the fact that $\tilde \psi_i$ and $\varphi_i$ might differ only on $B_i(\bar x,\tilde R_i)\setminus \tilde X_i$. Hence, 
\begin{align*}    \left|\int_{X_\infty}f\,d\tilde\psi_{i\ast}\mu_i-\int_{X_\infty}f\,d\varphi_{i\ast}\mu_i\right|& \,\,\,\,\leq \frac{1}{\varphi_\mu(\bar x,r_i)}\int_{B_i(\bar x,\tilde R_i)\setminus \tilde X_i}|f\circ \tilde \psi_i-f\circ\varphi_i|\,d\mu\\
& \,\,\,\,\leq \frac{1}{\varphi_\mu(\bar x,r_i)}\mu(B_i(\bar x,\tilde R_i)\setminus \tilde X_i)\cdot 2\sup_{X_\infty}|f|\\
& \underset{\underset{\tilde R_i<R_i}{\downarrow}}{\leq} \frac{1}{\varphi_\mu(\bar x,r_i)} \mu(B_i(\bar x,R_i)\setminus \tilde X_i)\cdot 2\sup_{X_\infty}|f|\\
& \,\,\,\,=\mu_i(B_i(\bar x,R_i)\setminus \tilde X_i)\cdot 2\sup_{X_\infty} |f|\leq \varepsilon_i\cdot 2\sup_{X_i\infty}|f|\overset{i\to\infty}{\longrightarrow}0.
\end{align*}
So, $\int_{X_\infty}f\,d\tilde\psi_{i\ast}\mu_i\to \int_{X_\infty}f\,d\mu_\infty$ for every $f\in C_{bbs}(X_\infty)$ and thus $\tilde \psi_{i\ast}\mu_i\rightharpoonup \mu_\infty$. This concludes the proof that $(r_i\inv X,\mu_{\bar x,r_i},\bar x)\stackrel{pmGH}{\longrightarrow}(X_\infty,\mu_\infty,x_\infty)$.
\end{proof}

\begin{proof}[Proof of Corollary \ref{cor:pmGpmGH}]
The convergence $T_{\bar x,r_i}f\stackrel{pmGH}{\longrightarrow}f_\infty$ implies $T_{\bar x,r_i}f\stackrel{pmG}{\longrightarrow}f_\infty$ by definition. Suppose that $T_{\bar x,r_i}f\stackrel{pmG}{\longrightarrow}f_\infty$. Then $T_{\bar x,r_i}X\stackrel{pmG}{\longrightarrow}X_\infty$. Proposition \ref{prop:pmGpmGH} yields $T_{\bar x,r_i}X\stackrel{pmGH}{\longrightarrow} X_\infty$. Let $\iota_i:T_{\bar x,r_i}X\to (Z,z)$ and $\iota_\infty:X_\infty\to (Z,z)$ be isometric embeddings realizing the convergence $T_{\bar x,r_i}X\stackrel{pmGH}{\longrightarrow} X_\infty$. Suppose $x\in X_\infty$ and $x_i\in T_{x,r_i}X$ be such that $\iota_i(x_i)\to \iota_\infty(x)$. By \cite[Lemma 5.6]{Bate22} we may assume $(X_\infty,\mu_\infty,x_\infty)$ to be doubling, and thus $X_\infty={\rm spt}\mu_\infty$, so there exits a sequence $\tilde x_i\in {\rm spt}\mu_i$ such that $\iota_i(\tilde x_i)\to \iota_\infty(x)$ (cf. Lemma \ref{lemma:support_point}).Therefore $d(x_i,\tilde x_i)/r_i\to 0$. Moreover, by \eqref{eq:map-precpt} we have $T_{\bar x,r_i}f(\tilde x_i)\to f_\infty(x)$. Finally, we have 
\begin{align*}   \|T_{x,r_i}f(x_i)-T_{x,r_i}f(\tilde x_i)\|_V=\frac{\|f(x_i)-f(\tilde x_i)\|_V}{r_i}\le \frac{Ld(x_i,\tilde x_i)}{r_i}\stackrel{i\to\infty}{\longrightarrow}0,
\end{align*}and thus $T_{\bar x,r_i}f(x_i)\to f_\infty(x)$. This proves the convergence $T_{\bar x,r_i}f\stackrel{pmGH}{\longrightarrow}f_\infty$.
\end{proof}

\begin{proof}[Proof of Theorem \ref{thm:pmGH-preiss}]
By \cite[Lemma 2.2 and 2.3]{Bate2018} $X$ can be decomposed into countably many disjoint sets $A_n\subset X$ with $\mu(X\setminus\bigcup_nA_n)=0$ so that $\mu$ is $(C_n,R_n)$-doubling along $A_n$ for each $n$. By the inner regularity of $\mu$ we can assume the sets $A_n$ to be closed (or compact). Note that for $\mu$-a.e. $x\in A_n$, $A_n$ is non-porous at $x$ (cf. Remark \ref{rem:non-porosity}). This and Corollary  \ref{cor:pmGpmGH} imply that
\begin{align}\label{eq:pmG=pmGH}
\operatorname{Tan}_{pmG}(f,x)=\operatorname{Tan}_{pmGH}(f,x)
\end{align}
$\mu$-a.e. $x\in X$. Since $\operatorname{Tan}_{pmG}(f,x)$ is compact and non-empty $\mu$-a.e. by Theorem \ref{thm:tan-meas+cpt}(ii), this yields (i). Let $N$ be a $\mu$-null set outside which Theorem \ref{thm:preiss-phenomenon} and \eqref{eq:pmG=pmGH} hold, and for which any $x\in X\setminus N$ is a non-porosity point of some $A_n$. Suppose $f_\infty:(Y,\nu,o)\to V$ belongs to $\operatorname{Tan}_{pmGH}(f,x)$ and $y\in \operatorname{spt}\nu$, $\rho>0$. Then $\frac{f_\infty-f_\infty(y)}{\rho}:T_{y,\rho}Y\to V$ belongs to $\operatorname{Tan}_{pmG}(f,x)$ by Theorem \ref{thm:preiss-phenomenon}, and by Proposition \ref{prop:pmGpmGH} and Corollary \ref{cor:pmGpmGH} we have that $\frac{f_\infty-f_\infty(y)}{\rho}:T_{y,\rho}Y\to V$ belongs to $\operatorname{Tan}_{pmGH}(f,x)$. This completes the proof of the Theorem.
\end{proof}

\appendix
\section{An Arzel\`a--Ascoli type theorem}

Given a map $f:X\to Y$ between metric spaces, we consider the following \textit{coarse co-uniform condition}. Let $h:[0,\infty)\to [0,\infty)$ be a continuous non-decreasing function\footnote{Note that we do not require $h(0)=0$, which is why this condition is called coarse.}. We say $f:X\to Y$ is \emph{$h$-co-uniform} if 
\begin{align}\label{eq:coarse-co-unif}
 \operatorname{diam}_X(f\inv (B))\le h(\operatorname{diam}_YB), \quad B\subset Y.
\end{align}
We say that $f$ is \textit{coarsely co-uniform} if it satisfies \eqref{eq:coarse-co-unif} for some $h$. In the following statement, we use the observation that, if $\mu$ is a boundedly finite Borel measure on $X$ and $f:X\to Y$ is $\mu$-measurable Lipschitz and coarsely co-uniform, then $f_\ast\mu$ is a boundedly finite measure on $Y$.

The following Arzel\`a--Ascoli type theorem for pmG-convergence is of independent interest.

\begin{theorem}\label{thm:arzela-ascoli-ultra}
    Suppose $X_i=(X_i,\mu_i,\bar x_i)$ and $Z_i=(Z_i,\nu_i,\bar z_i)$ be precompact families in $\XX$, with $\omega$-limits $X_\omega$, $Z_\omega$ in $\XX$ respectively. Let $f_i:X_i\to Z_i$ be $L$-Lipschitz maps satisfying \eqref{eq:coarse-co-unif} for some $h$ with $f_i(\bar x_i)=\bar z_i$ and $f_{i\ast}\mu_i=\nu_i$. Given $\iota_i:X_i\to (W,w)$, $\tilde\iota_i:Z_i\to (\widetilde W,\tilde w)$ as in Lemma \ref{lem:precpt}, respectively,
    there exist isometric embeddings $\iota_\omega:X_\omega\to (W,w)$, $\tilde\iota_\omega:Z_\omega\to (\widetilde W,\tilde w)$ so that the map
    $f_\omega:\{x_\omega\}\cup\operatorname{spt}\mu_\omega\to \{z_\omega\}\cup\operatorname{spt}\nu_\omega$ given by
    \begin{align}\label{eq:omega-limit}
    \tilde\iota_\omega\circ f_\omega(x)=\lim_\omega\tilde\iota_i\circ f_i( x_i),\quad\mathrm{whenever}\ x\in \operatorname{spt}\mu_\omega,\ x_i\in X_i,\quad \iota_\omega(x)=\lim_\omega \iota_i(x_i)
    \end{align}
    is well-defined and satisfies the following properties: 
    \begin{itemize}
        \item[(i)] $f_\omega$ is $L$-Lipschitz and $h$-co-uniform;
        \item[(ii)] $f_\omega(x_\omega)=z_\omega$ and $f_{\omega\ast}\mu_\omega=\nu_\omega$.
    \end{itemize}
    Moreover, there exists a subsequence $i_k$ so that $X_{i_k}\stackrel{pmG}{\longrightarrow}X_\omega$, $Z_{i_k}\stackrel{pmG}{\longrightarrow}Z_\omega$ and 
    \begin{align}\label{eq:subseq-omega}
    \lim_\omega\tilde\iota_i\circ f_i( x_i)=\lim_{k\to\infty}\tilde\iota_{i_k}\circ f_{i_k}(x_{i_k})
    \end{align}
    for any sequence $ x_{i_k}\in X_{i_k}$ with $\iota_{i_k}( x_{i_k})\to \iota_\omega(x)$, $x\in \{x_\omega\}\cup\operatorname{spt}\mu_\omega$.
\end{theorem}

In the proof, we will need the following auxiliary lemma.

\begin{lemma}
\label{lem:Hausdorff_limit_K_omega}
    Fix a non-principal ultrafilter $\omega$. Let $(Z,d)$ be a compact metric space and  $K_i\subset Z$ be compact subsets, $i\in \N$. Then $\lim_\omega d_H(K_i,K_\omega)=0$, where $K_\omega\coloneqq \{\lim_{\omega}x_i\,:\,x_i\in K_i\}$ and $d_H(K_i,K_\omega)\coloneqq \inf\{\delta>0\,:\,K_i\subset K_\omega^\delta\,\land\,K_\omega\subset K_i^\delta\}$ is the Hausdorff distance between compact sets.
\end{lemma}

\begin{proof}
    We argue by contradiction. Suppose that the claim is not true. Then there exists $\delta>0$ such that $\mathcal{U}_{\delta}\coloneqq \{i\in\N\,:\, d_H(K_i,K_\omega)>\delta\}\in\omega$. Since we have $\mathcal{U}_\delta\subset\mathcal{U}_\delta^1\cup \mathcal{U}_\delta^2$, where 
    \[\mathcal{U}_\delta^1\coloneqq \{i\in \N\,:\, (\exists x_i\in K_i)({\rm dist} (x_i,K_\omega)\geq \delta)\},\quad \mathcal{U}_\delta^2\coloneqq \{i\in \N\,:\, (\exists x_i\in K_\omega)({\rm dist}(x_i,K_i)\geq \delta)\},
    \]
    by the non-principal ultrafilter properties, at least one of the sets $\mathcal{U}_\delta^1$ and $\mathcal{U}_\delta^2$ belongs to $\omega$. Suppose first that $\mathcal{U}_\delta^1\in \omega$. Define the sequence 
    \[
    \tilde x_i\coloneqq \begin{cases}
        x_i, & \text{if } i\in \mathcal{U}_\delta^1,\\
        a_i, & \text{if } i\notin \mathcal{U}_\delta^1
    \end{cases},
    \]
    where $x_i\in K_i$ is from the definition of $\mathcal{U}_\delta^1$ and $a_i\in K_i$ is an arbitrary point in $K_i$. Then $\lim_\omega \tilde x_i\in K_\omega$ by the definition of $K_\omega$, contradicting the fact that  ${\rm dist}(K_\omega,\lim_\omega \tilde x_i)=\lim_\omega {\rm dist}(K_\omega, \tilde x_i)\geq \delta$. Suppose now that $\mathcal{U}_\delta^2\in\omega$. Define the sequence $\tilde x_i$ as above, just taking $\mathcal{U}_\delta^2$ instead of $\mathcal{U}_\delta^1$. Denote $x=\lim_\omega\tilde x_i\in K_\omega$. Then $\mathcal{U}_\delta^3=\{i\in \mathcal{U}_\delta^2\,:\,d(x_i,x)<\delta/3\}\in\omega$. Moreover, for $i\in \mathcal{U}_\delta^3$ we have
    \[
    {\rm dist}(K_i,x)\geq {\rm dist}(K_i,x_i)-d(x_i,x)\geq \frac{2\delta}{3},
    \]
    which is a contradiction since $x\in K_\omega$. This completes the proof.    
\end{proof}

\begin{proof}[Proof of Theorem \ref{thm:arzela-ascoli-ultra}]
    Let $\iota_i:X_i\to (W,w)$, $\tilde\iota_i:Z_i\to (\widetilde W,\tilde w)$ be isometric embeddings into complete separable metric spaces given by Lemma \ref{lem:precpt}, and denote $\bar\mu_i=\iota_{i\ast}\mu_i$, $\bar\nu_i=\tilde \iota_{i\ast}\nu_i$. Then $(\bar\mu_i)\subset \mathcal M_{loc}(W)$ and $(\bar\nu_i)\subset\mathcal M_{loc}(\widetilde W)$ are precompact. We identity $X_i$ and $Z_i$ with their isometric images, so that $f_i:\{w\}\cup\operatorname{spt}\bar\mu_i\to \{\tilde w\}\cup\operatorname{spt}\bar\nu_i$ are regarded as maps between subsets of $W$ and $\widetilde W$, respectively, with $f_i(w)=\tilde w$ and $f_{i\ast}\bar\mu_i=\bar\nu_i$. Denote $F=F_{(W,w)}$ and $\widetilde F=F_{(\widetilde W,\tilde w)}$ and the $\omega$-limits of $(\bar\mu_i)$ and $(\bar\nu_i)$ by $\mu_\omega$ and $\nu_\omega$, respectively. Then $\bar X_\omega=(\{w\}\cup\operatorname{spt}\mu_\omega,\mu_\omega,w)$ and $\bar Z_\omega=(\{\tilde w\}\cup\operatorname{spt}\nu_\omega,\nu_\omega,\tilde w)$ are the $\omega$-limits of $(X_i)$ and $(Z_i)$ in $\XX$. We may define $\iota_\omega$ and $\tilde\iota_\omega$ simply as the inclusion maps $\bar X_\omega\hookrightarrow W$, $\bar Z_\omega\hookrightarrow \widetilde W$.

    Now \eqref{eq:omega-limit} reads as 
    \begin{align*}
    f_\omega(x)=\lim_\omega \tilde\iota_i\circ f_i( x_i),\quad x=\lim_\omega\iota_i(x_i).
    \end{align*}
    We claim that this is well-defined. Firstly, note that if $x=\lim_\omega \iota_i(\hat x_i)=\lim_\omega\iota_i( x_i)$, then $\lim_\omega d_{\widetilde W}(\tilde\iota_i\circ f_i(\hat x_i),\tilde\iota_i\circ f_i(x_i))\le L\lim_\omega d_W(\iota_i(\hat x_i),\iota_i(x_i))=0$, so that the $\omega$-limit, if it exists, is uniquely defined. Next, we show that the $\omega$-limit of $\tilde\iota_i\circ f_i(x_i)$ exists for any sequence $x_i\in \operatorname{spt}\bar\mu_i$ with $x=\lim_\omega\iota_i(x_i)$. By the  statement of uniqueness above, to this end it suffices to prove the existence of a sequence $x_i\in\operatorname{spt}\mu_i$ with $x=\lim_\omega\iota_i(x_i)$  and $\tilde\iota_i\circ f_i(x_i)$ precompact for any $x\in \operatorname{spt}\mu_\omega$.

    For each $m$ choose $R_m\in [2^m,2^m+1]$ so that $\bar \mu_i(S(w,R_m))=\bar \nu_i(S(\tilde w,R_m))=0=\mu_\omega(S(w,R_m))=\nu_\omega(S(\tilde w,R_m))$ for each $i$. Using the precompactness of $(\bar\mu_i)$ and $(\bar\nu_i)$ we find for each $m\in \N$  increasing compact subsets $C^m\subset \bar B(w,R_m)$, $K^m\subset \bar B(\tilde w,LR_m)$ so that 
    \[
    \sup_i\bar\mu_i(\bar B(w,R_m)\setminus C^m)<2^{-m},\quad \sup_i \bar\nu_i(\bar B(\tilde w,LR_m)\setminus K^m)<2^{-m}.
    \]
    Denote $C_i^m:=C^m\cap f_i\inv (K^m)\cap ({\rm spt}\mu_i\cup\{\bar x_i\})$.
We may estimate 
\begin{align*}
\bar\mu_i(\bar B(w,R_m)\setminus C_i^m)&\le \bar\mu_i(\bar B(w,R_m)\setminus C^m)+\bar\mu_i(\bar B(w,R_m)\setminus f_i\inv(K^m))\\
&\le 2^{-m}+\bar\mu_i(f_i\inv(\bar B(\tilde w,LR_m)\setminus K^m))\le 2^{-m+1},\quad i\in\N.
\end{align*}
We set $C_\omega^m$ to be the collection of $\omega$-limits of sequences $x_i\in C^m_i$. Note that, since $C_i^m\subset C^m$ for all $i$ we have $C_\omega^m\subset C^m$. It is not difficult to see that $C_\omega^m$ is closed, and thus compact. We prove that 
\begin{align}\label{eq:omega-est}
\mu_\omega(\bar B(w,R_m)\setminus C_\omega^m)\le 2^{-m+1}.
\end{align}
To this end, let $K\subset B(w,R_m)\setminus C_\omega^m$ be compact. There exists $\delta>0$ so that $\operatorname{dist}(K,C_\omega^m)\ge 2\delta$ and $B(K,\delta)\subset B(w,R_m)$. We claim that the set $\{i\in\N:\ C_i^m\subset B(C_\omega^m,\delta)\}\in \omega$. Indeed, otherwise, its complement would be in $\omega$, yielding a sequence $\tilde x_i\in C_i^m$ with $\{i\in\N:\ \operatorname{dist}(\tilde x_i,C_\omega^m)\ge \delta\}\in\omega$. Then $\tilde x:=\lim_\omega\tilde x_i$ would satisfy $\operatorname{dist}(\tilde x,C_\omega^m)\ge \delta$, contradicting the definition of $C_\omega^m$. From this we obtain $\{i\in\N:\ B(K,\delta)\subset B(w,R_m)\setminus C_i^m\}\in\omega$ and so 
\begin{align*}
    \mu_\omega(K)\le \mu_\omega(B(K,\delta))\le\lim_\omega\bar\mu_i(B(K,\delta))\le \lim_\omega \bar\mu_i(B(w,2^m)\setminus C_i^m)\le 2^{-m+1}.
\end{align*}
Taking the supremum over $K$ and using that $S(w,R_m)$ has zero $\mu_\omega$-measure, we obtain \eqref{eq:omega-est}. It follows that $\mu_\omega(W\setminus \bigcup_mC_\omega^m)=0$. Thus $\operatorname{spt}\mu_\omega\subset \overline{\bigcup_mC_\omega^m}$ and, for any $x\in \overline{\bigcup_mC_\omega^m}$, we can find $x^m\in C_\omega^m$ with $x^m\to x$ as $m\to\infty$. For each $m\in\N$ take $x_i^m\in C_i^m$ with $x^m=\lim_\omega x_i^m$. Note that $(\tilde\iota_i\circ f_i( x_i^m))_{i}\subset K^m$ is precompact and thus its $\omega$-limit exists and belongs to $K^m$. The sets $\mathcal F^m\subset \N$ defined inductively by $\mathcal F^1=\N$,
\begin{align*}
    \mathcal F^m=\{i\in\mathcal F^{m-1}:\ d(f_\omega(x^m),\tilde\iota_i\circ f_i(x_i^m))<2^{-m}\}\cap \{i\in \mathcal{F}^{m-1}: d(x^m,x_i^m)<2^{-m}\}\cap [m,\infty)
\end{align*}
belong to $\omega$ for each $m$ and are decreasing. Defining $x_i=x_i^m$ for the largest $m\in \N$ such that $i\in \mathcal{F}^m$ and $x_i=w$ if such $m$ does not exist,
we obtain a sequence $x_i$ with $x=\lim_\omega x_i$ and $(\tilde\iota_i\circ f_i(x_i))$ precompact with $\omega$-limit $\lim_\omega \tilde\iota_i\circ f_i(x_i)=\lim_{m\to\infty}f_\omega(x^m)$. Note that we have defined $f_\omega$ on the set $S:=\{w\}\cup\overline{\bigcup_mC_\omega^m}$, which contains $\{w\}\cup\operatorname{spt}\mu_\omega$.

Next we prove (i). The estimate
\begin{align*}
    d(f_\omega(x),f_\omega(y))=\lim_\omega d(\tilde\iota_i\circ f_i(x_i),\tilde\iota_i\circ f_i(y_i))\le L\lim_\omega d(\iota_i(x_i),\iota_i(y_i))=Ld(x,y)
\end{align*}
for $x,y\in S$ 
shows that $f_\omega$ is $L$-Lipschitz. Let $B\subset \widetilde W$ be bounded, and $A:=f_\omega\inv(B)$. For $x,y\in A$ we have $f_\omega(x),f_\omega(y)\in B$. If $x=\lim_\omega \iota_i(x_i)$, $y=\lim_\omega \iota_i(y_i)$ we have 
$\tilde\iota_i(f_i(x_i)),\tilde\iota_i(f_i(y_i))\in B(B,\varepsilon)$ for $\omega$-a.e. $i\in\N$ for each $\varepsilon$. By the $h$-co-uniformity of $f_i$ we have 
\[
d(x_i,y_i)\le h(\operatorname{diam} B(B,\varepsilon))\le h(\operatorname{diam} B+2\varepsilon)\quad\omega-\mathrm{a.e. }\ i\in\N
\]
Since $\varepsilon>0$ is arbitrary we obtain $d(x,y)\le h(\operatorname{diam} B)$. This proves that $f_\omega$ is $h$-co-uniform.

To see \eqref{eq:subseq-omega} consider a countable dense subset $\{y^1,y^2,\ldots\}$ of $S$ 
and for each $y^m$ a sequence $\tilde y_i^m\in \{\bar x_i\}\cup\operatorname{spt}\mu_i$ with $y^m=\lim_\omega\iota_i(\tilde y_i^m)$. Let 
\begin{align*}
\mathcal F_m = &\{ i\ge m:\ F(\bar\mu_i,\mu_\omega)+\widetilde F(\bar\nu_i,\nu_\omega)+d_H(C_i^m,C_\omega^m)<1/m\}\\
& \cap \{i\geq m: d(y^j,\iota_i(\tilde y^j_i))+d(f_\omega(y^j),\tilde\iota_i(f_i(\tilde y_i^j))<1/m,\ j=1,\ldots,m\}
\end{align*}
for each $m$. Then $\omega(\mathcal F_m)=1$ for each $m\in\N$, and choosing a subsequence  with $i_m\in \mathcal F_{m}$ we obtain \eqref{eq:subseq-omega}.

It remains to prove (ii). The claim $f_\omega(x_\omega)=z_\omega$ follows directly from the fact that $f_i(w)=\tilde w$ for all $i$. To prove that $f_{\omega\ast}\mu_\omega=\nu_\omega$, pass to a subsequence $i_k$ satisfying \eqref{eq:subseq-omega}, and note that by the definition of $\mathcal F_k$ above we have that $\bar\mu_{i_k}\rightharpoonup\mu_\omega$, $\bar \nu_{i_k}\rightharpoonup \nu_\omega$.
According to Lemma \ref{lem:Hausdorff_limit_K_omega}, by possibly passing to a further subsequence, we may also assume that $C_{i_k}^m\to C_\omega^m$ in the Hausdorff distance in $C^m\subset W$. Arguing as above, we obtain the estimate
\[
\mu_\omega(\bar B(w,R_m)\setminus  C^m_\omega)\le 2^{-m+1}.
\]
Replacing $W$ by $W\times \R$ and considering the copies $\hat\mu_k$ of $\bar\mu_{i_k}$ on $W\times\{1/k\}$ (and $\hat\mu_\omega$ on $W\times\{0\}$) we have that $\operatorname{spt}\hat\mu_k\subset W\times\{1/k\}\subset W\times\R$ are disjoint. Then $\hat\mu_k\rightharpoonup \hat\mu_\omega$. Set
\begin{align*}
X^m:=\widehat C_\omega^m\cup \bigcup_k \widehat C_{i_k}^m\subset W\times\R\quad (\widehat C^{m}_{i_k}=C^m_{i_k}\times\{1/k\},\ \widehat C^m_\omega=C^m_\omega\times\{0\})
\end{align*}
and note that $X^m$ is a closed subset of $\bar B(w,R_m)\times[0,1]$ and furthermore $\operatorname{spt}\hat\mu_\omega\subset \overline{\bigcup_mX^m}$. Note, moreover, that
\begin{align*}
    F:X^m\to \widetilde W,\quad F|_{\widehat C_{i_k}^m}=f_{i_k}|_{C^m_{i_k}},\quad F|_{\widehat C_\omega^m}=f_\omega|_{C_\omega^m}
\end{align*}
is continuous by \eqref{eq:omega-limit} and \eqref{eq:subseq-omega}. Fix $\varphi\in C_{bbs}(\widetilde W)$. If $\varphi$ is supported in a ball $B(\tilde w,R_0)$, then by the co-uniformity condition 
\[
\operatorname{spt}(\varphi\circ f_i)\subset f_i\inv(\operatorname{spt}\varphi)\subset f_i\inv(B(\tilde w,R_0))\subset B(w,h(2R_0))
\]
for all $i\in\N\cup\{\omega\}$. Thus, for large enough $m$ the functions $\varphi\circ f_{i_k}$ and $\varphi\circ f_\omega$ are supported in $B(w,R_m)$ for all $k$. By Tietze's extension theorem and multiplying by a cut-off function (see \cite[Lemma 2.46, Theorem 2.47]{AliprantisBorder99}) there is a continuous function $\Phi^m:W\times\R\to \R$ with $\Phi^m|_{X^m}=\varphi\circ F$, $\|\Phi^m\|\le \|\varphi\|_\infty$ and $\operatorname{spt}\Phi^m\subset \bar B(w,R_m)\times[-1,2]$. By the weak convergence $\hat\mu_k\rightharpoonup\hat\mu_\omega$ and $\bar\nu_{i_k}=f_{i_k\ast}\bar \mu_{i_k}\rightharpoonup \nu_\omega$ we have
\begin{align}\label{eq:limits}
\int \Phi^m\ud\hat \mu_\omega=\lim_{k\to\infty}\int\Phi^m\ud\hat\mu_k,\quad \int\varphi\ud\nu_\omega=\lim_{k\to\infty}\int\varphi\circ f_{i_k}\ud\bar\mu_{i_k}.
\end{align}
However
\begin{align*}
&\left| \int\Phi^m\ud\hat\mu_\omega-\int \varphi\circ f_\omega\ud\mu_\omega \right|=\left| \int_{\bar B(w,R_m)\times\{0\}}\Phi^m\ud\hat\mu_\omega-\int_{\bar B(w,R_m)\times\{0\}} \varphi\circ f_\omega\ud\hat\mu_\omega \right|\\
\le &\int_{(B(w,R_m)\setminus C^m_\omega)\times\{0\}}|\Phi^m-\varphi\circ f_\omega|\ud\hat\mu_\omega\le 2\|\varphi\|_\infty\mu_\omega(B(w,R_m)\setminus C^m_\omega)\le 2^{-m+2}\|\varphi\|_\infty,
\end{align*}
and similarly
\begin{align*}
&\left| \int\Phi^m\ud\hat\mu_k-\int \varphi\circ f_{i_k}\ud\bar\mu_{i_k} \right|=\left| \int_{B(w,R_m)\times\{1/k\}}\Phi^m\ud\hat\mu_k-\int_{B(w,R_m)\times\{1/k\}} \varphi\circ f_{i_k}\ud\hat\mu_{k} \right|\\
\le & \int_{B(w,R_m)\times\{1/k\}}|\Phi^m-\varphi\circ f_{i_k}|\ud\hat\mu_{k}\le 2\|\varphi\|_\infty\bar\mu_{i_k}(B(w,R_m)\setminus C^m_{i_k})\le 2^{-m+2}\|\varphi\|_\infty.
\end{align*}
Sending $m\to\infty$, these two estimates together with \eqref{eq:limits} yield $\int\varphi\ud\nu_\omega=\int\varphi\circ f_\omega\ud\mu_\omega$. Since $\varphi\in C_{bbs}(\widetilde W)$ is arbitrary, this completes the proof of (ii) and the theorem. 
\end{proof}

\begin{remark}
The coarse co-uniformity condition \eqref{eq:coarse-co-unif} in Theorem \ref{thm:arzela-ascoli-ultra} cannot be removed, as the following example shows. Let 
\begin{align*}
X_i&=[0,1]^2\cup\{(x,y):\ 0\le x\le 1/i,\ 1\le y\le i+1\}\subset \R^2,\ \mu_i=\mathcal L^2|_{X_i},\ x_i^0=(0,0),\\
f_i&:X_i\to \R, \ f_i: (x,y)\mapsto x.
\end{align*}
Then $(X_i,\mu_i,x_i^0)\stackrel{pmG}{\longrightarrow}([0,1]^2,\mathcal L^2|_{[0,1]^2},\bar 0)$ and $f_i\to f_\infty : (x,y)\mapsto x$, whence $f_{\infty\ast}\mu_\infty=\mathcal{L}^1|_{[0,1]}$. However, 
$$f_{i\ast}\mu_i=\mathcal L^1|_{[0,1]}+i\mathcal L^1|_{[0,1/i]}\rightharpoonup \mathcal L^1|_{[0,1]}+\delta_0=\nu_\infty,$$
which means that the equality $f_{\infty\ast}\mu_\infty=\nu_\infty$ does not hold. 
\end{remark}

\begin{figure}[H]
    \centering    
    \begin{subfigure}{0.22\textwidth}
        \centering
        \includegraphics[scale=0.05]{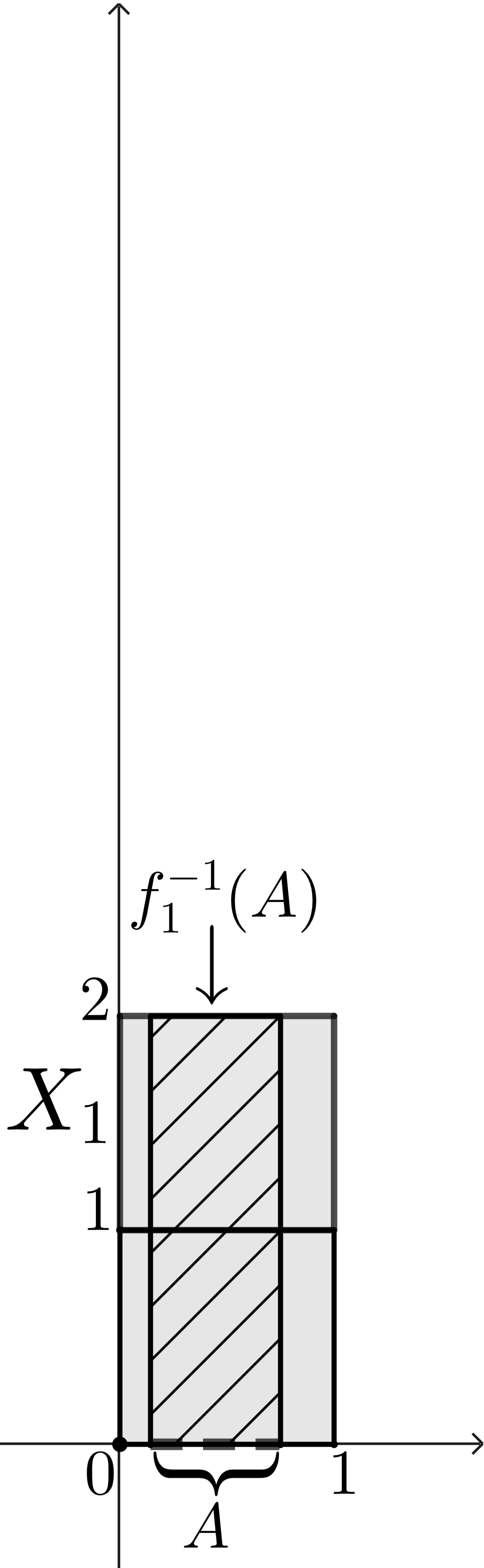}
    \end{subfigure}
    \hfill
    \begin{subfigure}{0.22\textwidth}
        \centering
        \includegraphics[scale=0.05]{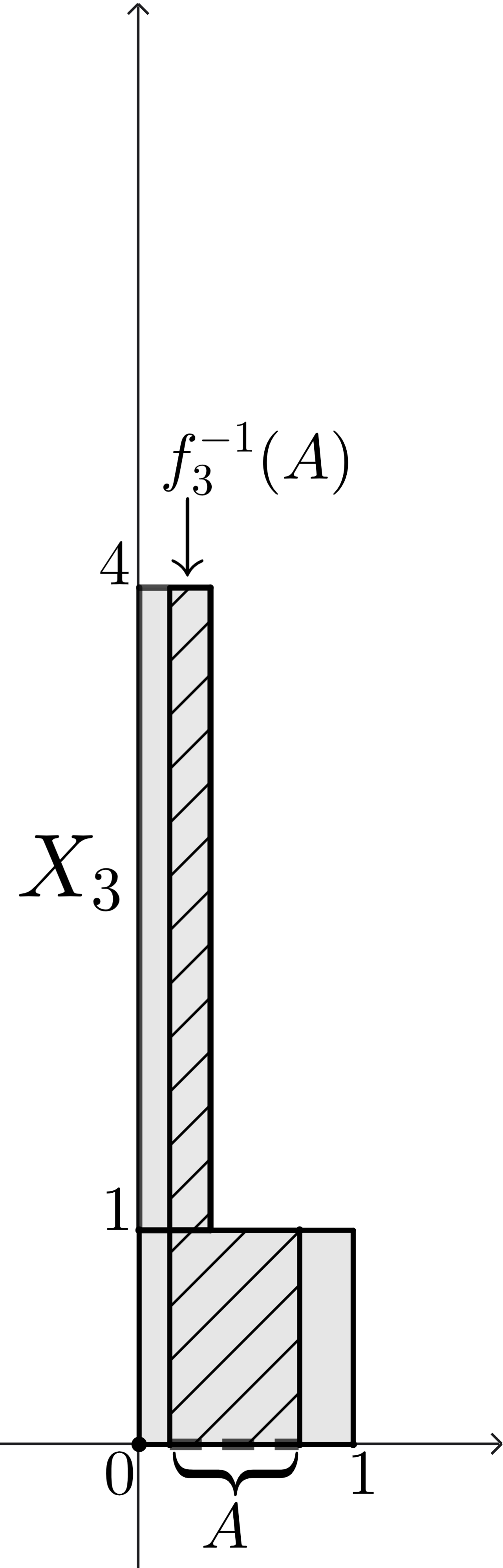}
    \end{subfigure}
    \hfill
    \begin{subfigure}{0.22\textwidth}
        \centering
        \includegraphics[scale=0.05]{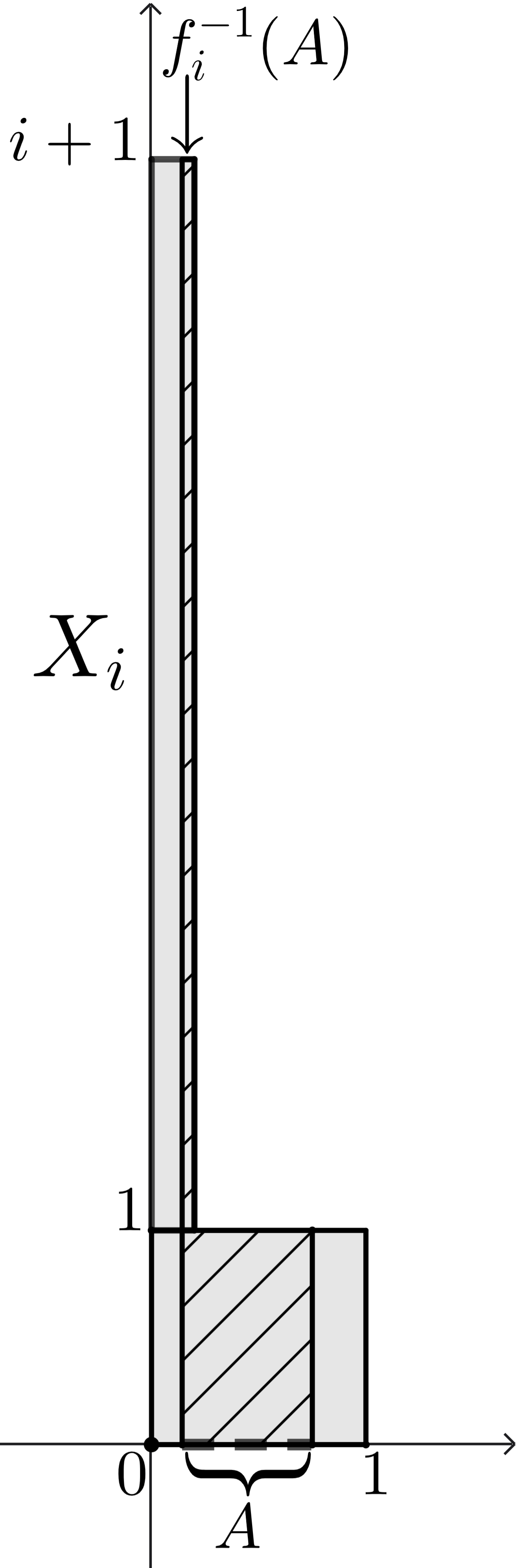}
    \end{subfigure}
    \hfill
    \begin{subfigure}{0.22\textwidth}
        \centering
        \includegraphics[scale=0.05]{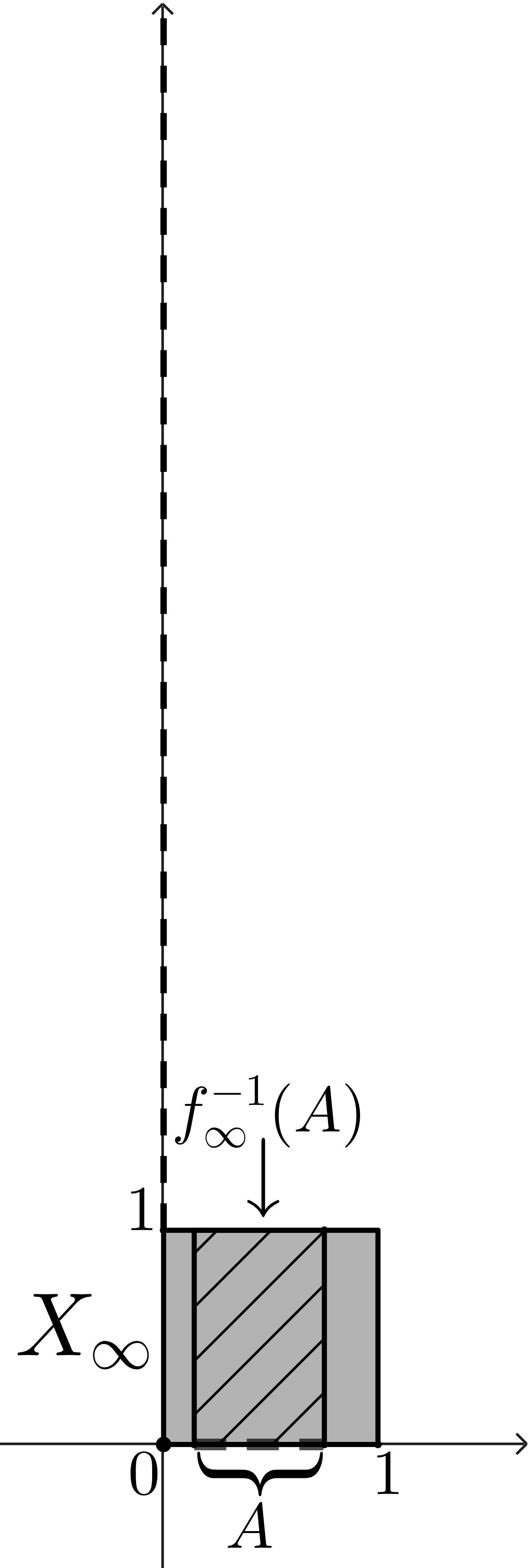}
    \end{subfigure}
    \caption{${f_{i}}_{\ast}\mu_i=\mathcal L^1|_{[0,1]}+i\mathcal L^1|_{[0,1/i]}\rightharpoonup \mathcal{L}^1|_{[0,1]}+\delta_0\neq {f_{\infty}
}_\ast\mu_\infty$}
\end{figure}

\def\cprime{$'$} \def\cprime{$'$}

\end{document}